\documentclass[pdflatex,sn-mathphys-num]{sn-jnl}% Math and Physical Sciences Numbered Reference Style
\usepackage{graphicx}%
\usepackage{multirow}%
\usepackage{amsmath,amssymb,amsfonts}%
\usepackage{amsthm}%
\usepackage{mathrsfs}%
\usepackage[title]{appendix}%
\usepackage{xcolor}%
\usepackage{textcomp}%
\usepackage{manyfoot}%
\usepackage{booktabs}%
\usepackage{algorithm}%
\usepackage{algorithmicx}%
\usepackage{algpseudocode}%
\usepackage{listings}%

\usepackage{graphicx}
\usepackage{subcaption}
\usepackage{array}
\usepackage{comment}

\usepackage{booktabs}
\usepackage{siunitx}
\theoremstyle{thmstyleone}%
\newtheorem{theorem}{Theorem}%  meant for continuous numbers
\newtheorem{assumption}[theorem]{Assumption}
\newtheorem{property}[theorem]{Property}
\newtheorem{lemma}[theorem]{Lemma}

\theoremstyle{thmstyletwo}%
\newtheorem{remark}{Remark}%

\theoremstyle{thmstylethree}%

\begin{document}

\title[Article Title]{LegONet: Plug-and-Play Structure-Preserving Neural Operator Blocks for Compositional PDE Learning}

%%=============================================================%%
%% GivenName	-> \fnm{Joergen W.}
%% Particle	-> \spfx{van der} -> surname prefix
%% FamilyName	-> \sur{Ploeg}
%% Suffix	-> \sfx{IV}
%% \author*[1,2]{\fnm{Joergen W.} \spfx{van der} \sur{Ploeg} 
%%  \sfx{IV}}\email{iauthor@gmail.com}
%%=============================================================%%

\author[1]{\fnm{Jiahao} \sur{Zhang}}
\author[1]{\fnm{Yueqi} \sur{Wang}}
\equalcont{These authors contributed equally to this work.}

\author*[1,2]{\fnm{Guang} \sur{Lin}}\email{guanglin@purdue.edu}

\affil[1]{\orgdiv{Department of Mathematics}, \orgname{Purdue University},
\orgaddress{\city{West Lafayette}, \state{IN}, \postcode{47907}, \country{USA}}}

\affil[2]{\orgdiv{School of Mechanical Engineering}, \orgname{Purdue University},
\orgaddress{\city{West Lafayette}, \state{IN}, \postcode{47907}, \country{USA}}}

%%==================================%%
%% Sample for unstructured abstract %%
%%==================================%%

\abstract{
Learned PDE solvers are often trained as monolithic surrogates for a specific equation, boundary condition and discretization. This makes them difficult to reuse when mechanisms change and it can limit stability under long-horizon rollout. We introduce Lego-like Operator Network (LegONet), a compositional framework that builds PDE solvers from plug-and-play, structure-preserving operator blocks defined on shared boundary-adapted spectral representations. LegONet separates boundary handling from mechanism learning, satisfying boundary conditions by construction. It also separates mechanism learning from time integration, enabling pretrained blocks to be assembled into new solvers without retraining. We also derive a finite-horizon error decomposition that separates block mismatch from splitting error and provides mechanism-level diagnostics for long-horizon predictions. Across ten time-dependent PDEs, LegONet delivers accurate closed-loop rollouts with improved stability under cross-PDE recombination and boundary reconfiguration. More broadly, this modular formulation suggests a path from task-specific neural solvers towards plug-and-play operator libraries for scientific computing.
}


%%\pacs[JEL Classification]{D8, H51}

%%\pacs[MSC Classification]{35A01, 65L10, 65L12, 65L20, 65L70}

\maketitle

%733 words

Scientific machine learning is reshaping how we model and simulate partial differential equations (PDEs), complementing
discretization-based solvers that remain the standard when accuracy and stability guarantees are required
\cite{leveque2007finite,strikwerda2004finite,hughes2003finite,brenner2008mathematical,boyd2001chebyshev}.
A key promise of learned solvers is amortization: once trained, they can accelerate repeated solves, design loops, and control tasks
\cite{raissi2019physics,karniadakis2021physics,kovachki2024operator}.
In practice, however, the main bottleneck is rarely a single forward evaluation. It is reconfiguration. Equations change as operators are added, removed or retuned. Boundary conditions vary across instances. Long-horizon rollout is often needed to resolve stiff, multiscale or turbulent regimes. Yet most learned PDE solvers remain difficult to reuse under such changes.

Two paradigms have shaped current practice.
Physics-informed neural networks (PINNs) impose PDE and boundary constraints through residual penalties and typically optimize per instance, which
limits reuse across operators and configurations \cite{raissi2019physics}.
Operator-learning architectures such as DeepONet and Fourier neural operator (FNO) learn mappings between function spaces from trajectory data
\cite{lu2019deeponet,li2020fourier,kovachki2023neural,pfaff2020learning,li2024physics,hao2023gnot,ye2024pdeformer}.
These models can deliver fast inference once trained, but their learned dynamics are usually encoded in a single end-to-end map.
Boundary handling, discretization and operator identity therefore become entangled, and plug-and-play reconfiguration is rare.
Moreover, reliable long-horizon rollout remains challenging when the learned vector field is not explicitly tied to mechanisms that control stability or invariants,
especially under stiffness or non-commuting multi-physics coupling \cite{wang2023long,sharma2023stiff,wang2024understanding}.

A separate line of work improves qualitative stability by hard-wiring structure into learned dynamics.
Energy-based and structure-preserving formulations parameterize scalar generators and recover vector fields through analytic relations,
including Hamiltonian and Lagrangian neural networks and their dissipative extensions
\cite{yu2018deep,greydanus2019hamiltonian,cranmer2020lagrangian,sosanya2022dissipative,desai2021port,roth2025stable,zhang2022gfinns,bouziani2024structure,eidnes2024pseudo}.
These models can encode conservation or dissipation more robustly than unconstrained time-steppers, but they are usually designed as monolithic surrogates for a specific system. As a result, they offer limited support for compositional reuse when operators are reconfigured, and they provide little guidance on whether rollout failures arise from a
particular mechanism approximation or from the time-integration scheme.

We introduce LegONet, a compositional framework for time-dependent PDEs that replaces monolithic neural solvers with a plug-and-play library of
structure-preserving operator blocks coupled through a shared coefficient representation (Fig.~\ref{fig:compare}).
Our central premise is simple: plug-and-play learned PDE solvers require an explicit layer between PDE specification and time integration. That layer should provide a boundary-compatible state representation that all blocks can share, together with single-purpose operator primitives that can be selected, recombined and diagnosed without retraining an entire model.

LegONet is built around two separations. The first separates boundary handling from mechanism learning. The second separates mechanism learning from time integration. We begin from a decomposition of the target PDE into a sum of mechanisms,
\begin{equation}\label{eq:intro_pde_sum}
u_t(\cdot,t)=\sum_{i=1}^{N_{\mathrm{blk}}} L_i\big(u(\cdot,t)\big),
\end{equation}
and represent the homogeneous component of the solution on a boundary-adapted spectral baseplate. For non-homogeneous boundary data, we apply a lifting $u=u_{\mathrm{lift}}+u_0$ so that $u_0$ satisfies homogeneous constraints, and approximate
\begin{equation}\label{eq:intro_coeff_rep}
u_0(\cdot,t)\approx\sum_{k=1}^{K} a_k(t)\,\phi_k^{(b)}(\cdot),
\qquad
\mathbf{a}(t):=[a_1(t),\cdots,a_K(t)]^\top\in\mathbb{R}^K.
\end{equation}
This construction fixes $\mathbf{a}(t)$ as a shared coefficient state on the chosen baseplate. All blocks act on this same state representation, while boundary conditions are enforced by construction through the basis $\{\phi_k^{(b)}\}_{k=1}^K$ and lifting $u_{\mathrm{lift}}$.

On this shared representation, LegONet implements each mechanism by a structured coefficient-space vector field induced from scalar generators and fixed
structure operators,
\begin{equation}\label{eq:intro_block_form}
F_i^{\boldsymbol{\theta}}(\mathbf{a})
=
-\,G_{i}\,\nabla_{\mathbf{a}} E_{i}^{a,\boldsymbol{\theta}}(\mathbf{a})
\;+\;
J_{i}\,\nabla_{\mathbf{a}} H_{i}^{a,\boldsymbol{\theta}}(\mathbf{a})
\;+\;
R_{i}^a(\mathbf{a}).
\end{equation}
Here $G_i$ and $J_i$ are baseplate-consistent operators that encode dissipative or conservative structure on the retained modes, while
$E_i^{a,\boldsymbol{\theta}}$ and $H_i^{a,\boldsymbol{\theta}}$ parameterize scalar generators. The residual term $R_i^a$ accommodates contributions outside the selected variational forms, such as explicit forcing or constraints.
This yields plug-and-play operator blocks for mechanisms such as diffusion, transport and reaction, all defined on the same coefficient state.

LegONet also separates training from deployment. Blocks are pretrained offline by instantaneous operator matching in coefficient space:
we sample admissible states $\mathbf{a}$, evaluate a trusted discretization to obtain reference targets, and learn each block so that
$F_i^{\boldsymbol{\theta}}(\mathbf{a})$ matches the reference mechanism update without trajectory fitting. At deployment, a new PDE instance is created by selecting a baseplate, choosing the relevant blocks and advancing the resulting reduced dynamics by symmetric Strang splitting.
Reconfiguring operators or boundary settings therefore becomes a selection-and-assembly problem rather than a retraining problem. This modular view also makes long-horizon behavior more interpretable: our finite-horizon analysis separates rollout error into block mismatch and splitting error, allowing failures to be attributed to the learned mechanisms or to the composition scheme.
Fig.~\ref{fig:workflow} summarizes the full workflow. More broadly, by exposing plug-and-play blocks on a common coefficient representation, LegONet suggests a path toward community-built libraries of interoperable operators that could mature into practical scientific-computing packages.

\section*{Results}\label{sec:Experiments}

LegONet is evaluated as a compositional solver assembled from pretrained mechanism blocks on a shared coefficient representation (Fig.~\ref{fig:workflow}).
Across all experiments, each target PDE is decomposed into a sum of mechanisms instantiated on a boundary-adapted baseplate and advanced by symmetric Strang composition in coefficient space.
Reference trajectories are generated on the same trial space, with the same projection operators and the same splitting schedule. The only difference is that exact operators are used in place of learned blocks. This design isolates the effect of block learning from differences in discretization or rollout protocol.

\begin{figure}[H]
  \centering

  \begin{subfigure}[t]{\linewidth}
    \centering
    \includegraphics[width=\linewidth,
    trim={0cm 0cm 0cm 0.0cm},
    clip]{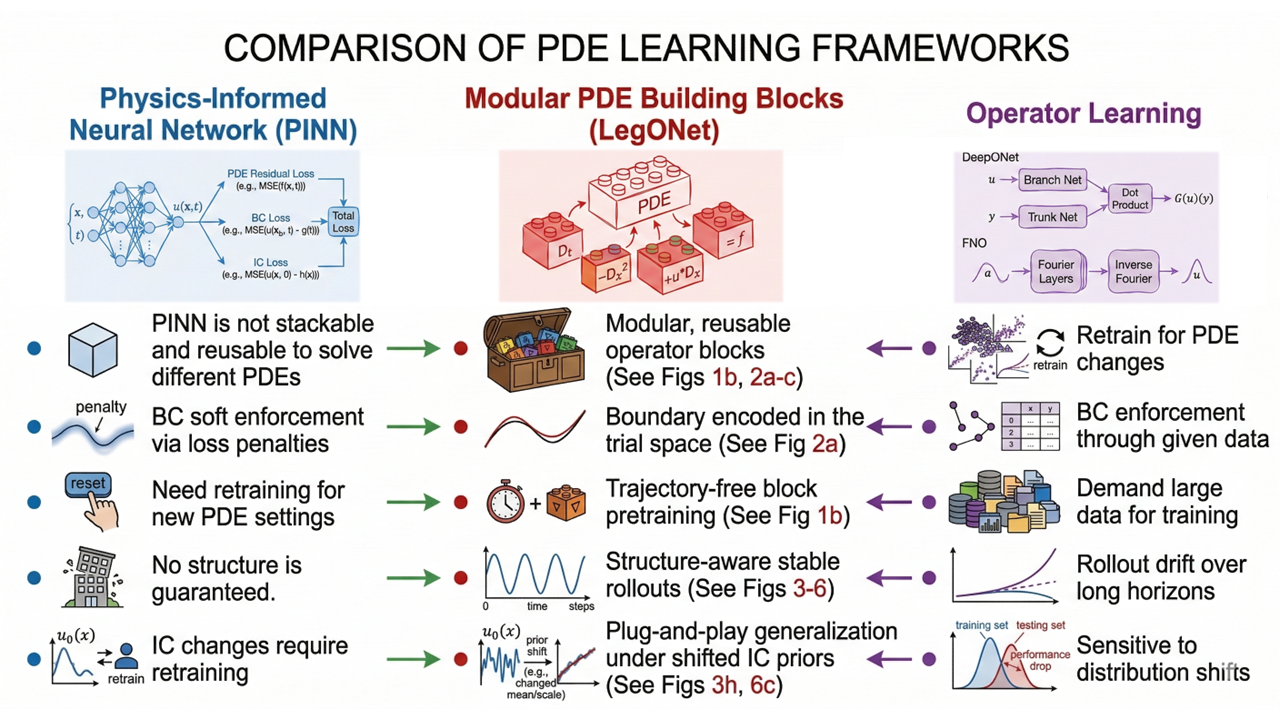}
    \caption{\textbf{Comparing neural PDE paradigms}: from monolithic solvers to plug-and-play modular PDE building blocks.
}
    \label{fig:compare}
  \end{subfigure}

  \vspace{6pt}

  \begin{subfigure}[t]{\linewidth}
    \centering
    \includegraphics[width=\linewidth]{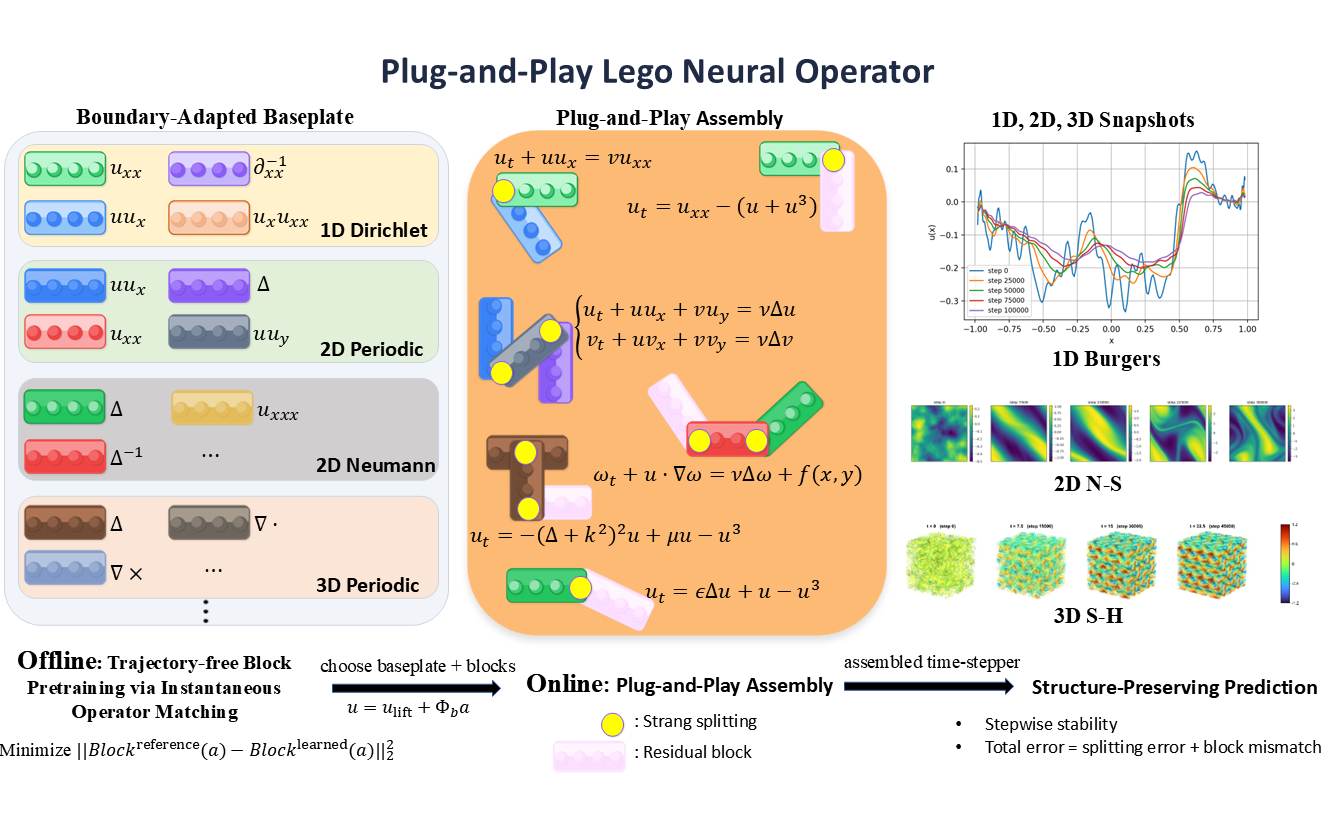}
    \caption{\textbf{LegONet pipeline}:
    blocks are pretrained offline by instantaneous operator matching on a boundary-adapted coefficient baseplate, then selected and composed at deployment to form a time-stepper for new PDE instances.}
    \label{fig:workflow}
  \end{subfigure}

  \caption{\textbf{From monolithic neural solvers to modular operator blocks.}
  Fig. \ref{fig:compare}: conceptual comparison of operator learning, physics-informed optimization and LegONet.
  Fig. \ref{fig:workflow}: LegONet workflow.}
  \label{fig:lego_compare_workflow}
\end{figure}

Trajectory accuracy is quantified on the common evaluation nodes $\{x_q,w_q\}_{q=1}^{Q}$ by the weighted relative $L^2$ error
\begin{equation}\label{eq:err_metrics}
\mathrm{rel}(\mathbf p,\mathbf t)
=\frac{\|\mathbf p-\mathbf t\|_{w,2}}{\|\mathbf t\|_{w,2}},
\qquad
\|\mathbf u\|_{w,2}^2
=\sum_{q=1}^{Q} w_q\,|u_q|^2,
\end{equation}
together with the normalized pointwise error profile
$e(x_q)=(p_q-t_q)/\|\mathbf t\|_{w,2}$,
where $\mathbf p=(p_q)_{q=1}^{Q}$ and $\mathbf t=(t_q)_{q=1}^{Q}$ are the predicted and reference values evaluated at the same nodes.
Here and below, non-bold symbols such as $u(x,t)$ and $v(x,t)$ denote continuous fields, whereas bold symbols such as $\mathbf u(t)$ and $\mathbf v(t)$ denote their nodal evaluations on the common quadrature/grid nodes.
Beyond trajectory error, we report structure-aware diagnostics aligned with mechanism type, including energy drift for dissipative generators and invariant preservation for Hamiltonian blocks.
We also track a kinetic-energy diagnostic on the relevant field $v(\cdot,t)$, defined by
$E(t):=\tfrac12\|\mathbf v(t)\|_{w,2}^2$,
and summarize its deviation from the reference by the energy relative error
\[
\mathrm{relE}(t):=\frac{\big|E^{\mathrm{pred}}(t)-E^{\mathrm{ref}}(t)\big|}{\big|E^{\mathrm{ref}}(t)\big|}.
\]
These measurements are chosen to reflect the two sources of rollout error highlighted by Theorem~\ref{thm:main_structure_error}:
mismatch in the learned blocks and discretization error from symmetric composition.
Complementary per-task rollouts, error maps, and block-pretraining diagnostics are provided in Extended Data.

\subsection*{Benchmark scope: plug-and-play blocks across four baseplates and ten PDEs}

Fig.~\ref{fig:benchmark_big_combined1} provides an overview of the full benchmark suite and the plug-and-play modular PDE building blocks enabled by LegONet.
We construct four boundary-adapted baseplates—1D Dirichlet (Shen--Legendre), 2D periodic (Fourier), 2D Neumann (cosine), and 3D periodic (Fourier)—each equipped with fixed structure operators $(G,J)$ consistent with the trial space \footnote{On the 1D Shen baseplate, $M_{ij}=\langle\phi_i,\phi_j\rangle_{L^2}$ and $S_{ij}=\langle \partial_x\phi_i,\phi_j\rangle_{L^2}$, so $G=M^{-1}$ and $J=M^{-1}S$ realizes $\partial_x$ on the trial space.
In Fourier/cosine baseplates, $J_x$ and $J_y$ denote fixed coefficient-space representations of $\partial_x$ and $\partial_y$ on the retained modes.
Blocks $u_{xx}$ and $\Delta$ are E-blocks, $uu_x$ and $uu_y$ are H-blocks, and $\Delta^{-1}$ denotes the Poisson-inversion block.}. %\footnotemark.
On these shared coefficient representations, we pretrain plug-and-play blocks in the form of Eq.~\eqref{eq:intro_block_form}, including diffusion ($u_{xx}$ or $\Delta$), Hamiltonian transport ($uu_x$, $uu_y$), Poisson inversion ($\Delta^{-1}$), and higher-order reuse through repeated application.

Using only these primitives, we assemble solvers for ten time-dependent PDEs spanning one to three spatial dimensions, multiple boundary conditions and qualitatively different dynamics.
Fig.~\ref{fig:sum} reports end-to-end rollout accuracy, including worst-case and time-averaged relative errors over task-specific horizons chosen to probe representative regimes. Fig.~\ref{fig:benchmark_big_combined} then shows representative visual rollouts and normalized pointwise error maps, with color scales ranging from $10^{-4}$ down to $10^{-8}$ across cases.
Extended Data further reports the complete set of rollout visualizations and normalized error maps, including cases not shown in the main text.
Implementation details for baseplates, block pretraining, and equation-specific discretization settings are provided in Supplementary Information.

\subsection*{Plug-and-play time stepping: symmetric block composition without retraining}

Fig.~\ref{fig:strang} illustrates how pretrained blocks are assembled into a solver for each target PDE. For every system, we construct a symmetric Strang scheme in coefficient space, where each selected block contributes either a half-step ($\Delta t/2$) or a full step ($\Delta t$) update. Because all blocks act on the same coefficient state and are induced by fixed $(G,J)$ operators, the resulting composition is modular by construction. Changing the PDE therefore does not require retraining a monolithic model. It only requires selecting, reweighting and ordering the appropriate blocks.

This plug-and-play property is exercised repeatedly across the benchmark suite. We test operator recombination by adding or removing mechanisms such as transport or Poisson inversion, and we test boundary reconfiguration by switching baseplates while preserving the same block-based formulation. Together, these experiments evaluate the central claim of LegONet: the plug-and-play, structure-preserving blocks can support accurate and stable rollout across PDEs that differ in operator content, boundary conditions and dimensionality.

\subsection*{Case study I: coupling independently trained dissipation and transport blocks}

We first ask whether independently pretrained dissipative and Hamiltonian blocks remain compatible when they are assembled into a boundary-constrained solver. This is a basic test of the LegONet design: if plug-and-play blocks are to replace monolithic training, then separately learned mechanisms must still behave coherently after composition.

We consider 1D viscous Burgers with Dirichlet boundaries
\begin{equation}\label{eq:Burgers_1D_PDE_rewrite_res}
  u_t + u\,u_x = \nu u_{xx},\qquad x\in(-1,1),\qquad u(\pm1,t)=0,
\end{equation}
with $\nu=0.03$, time step $\Delta t=10^{-5}$, and final time $T=1$.
On the Shen baseplate, LegONet composes a diffusion block $u_{xx}$ represented as a dissipative generator $-G\nabla E^{a,\boldsymbol{\theta}}_{u_{xx}}$ and a transport block $uu_x$ represented as a Hamiltonian generator $J\nabla H^{a,\boldsymbol{\theta}}_{uu_x}$.
Here $E^{a,\boldsymbol{\theta}}_{u_{xx}}$ and $H^{a,\boldsymbol{\theta}}_{uu_{x}}$ are parameterized by small multilayer perceptrons with four hidden layers, width 128 and GELU activation.
We advance coefficients by Strang composition with one full diffusion step and two transport half-steps per macro-step (Fig.~\ref{fig:strang}).

\begin{figure}[H]
  \centering

  \begin{subfigure}[t]{\linewidth}
    \centering
    \includegraphics[width=\linewidth,trim={0cm 2cm 0cm 2cm},clip]{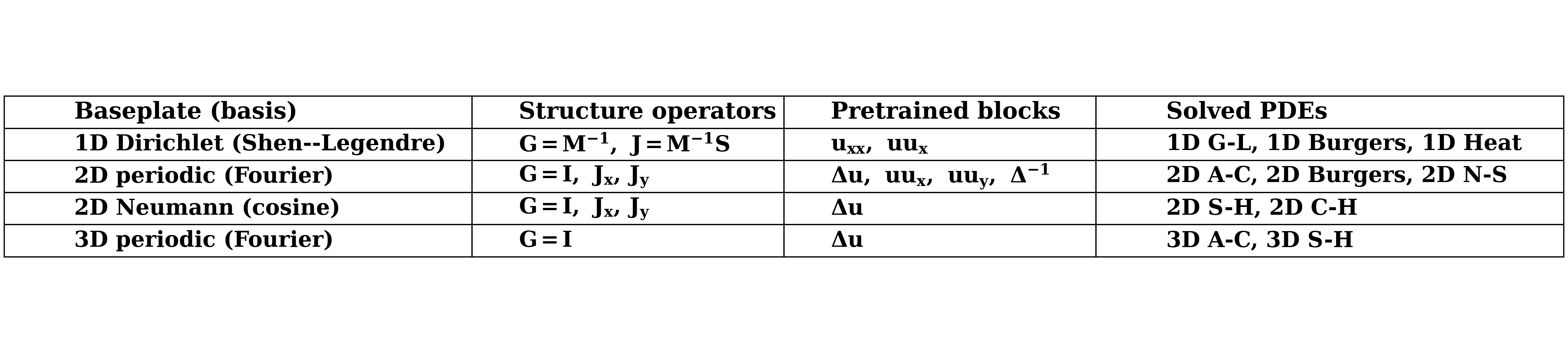}
    \caption{Baseplates and blocks.}\label{fig:block}
  \end{subfigure}\\[0.9ex]

  \begin{subfigure}[t]{\linewidth}
    \centering
    \includegraphics[width=\linewidth]{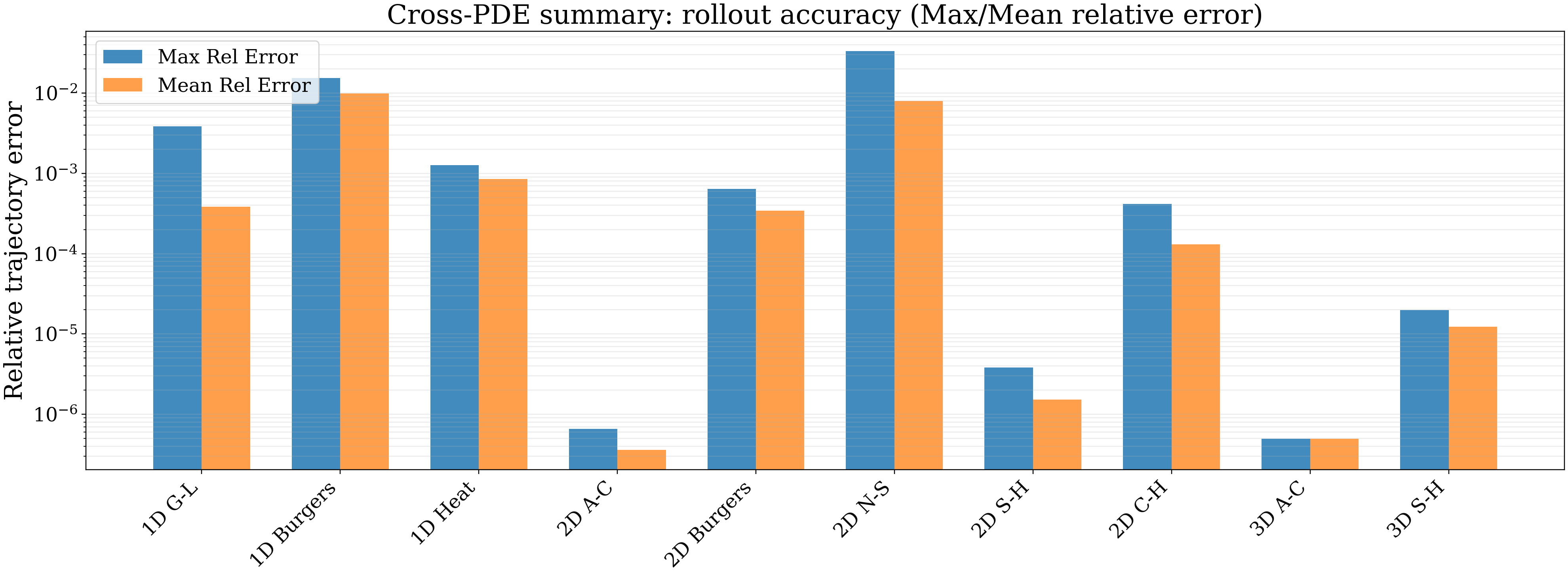}
    \caption{Cross-PDE rollout accuracy.}\label{fig:sum}
  \end{subfigure}\\[0.9ex]

  \begin{subfigure}[t]{\linewidth}
    \centering
    \includegraphics[width=\linewidth]{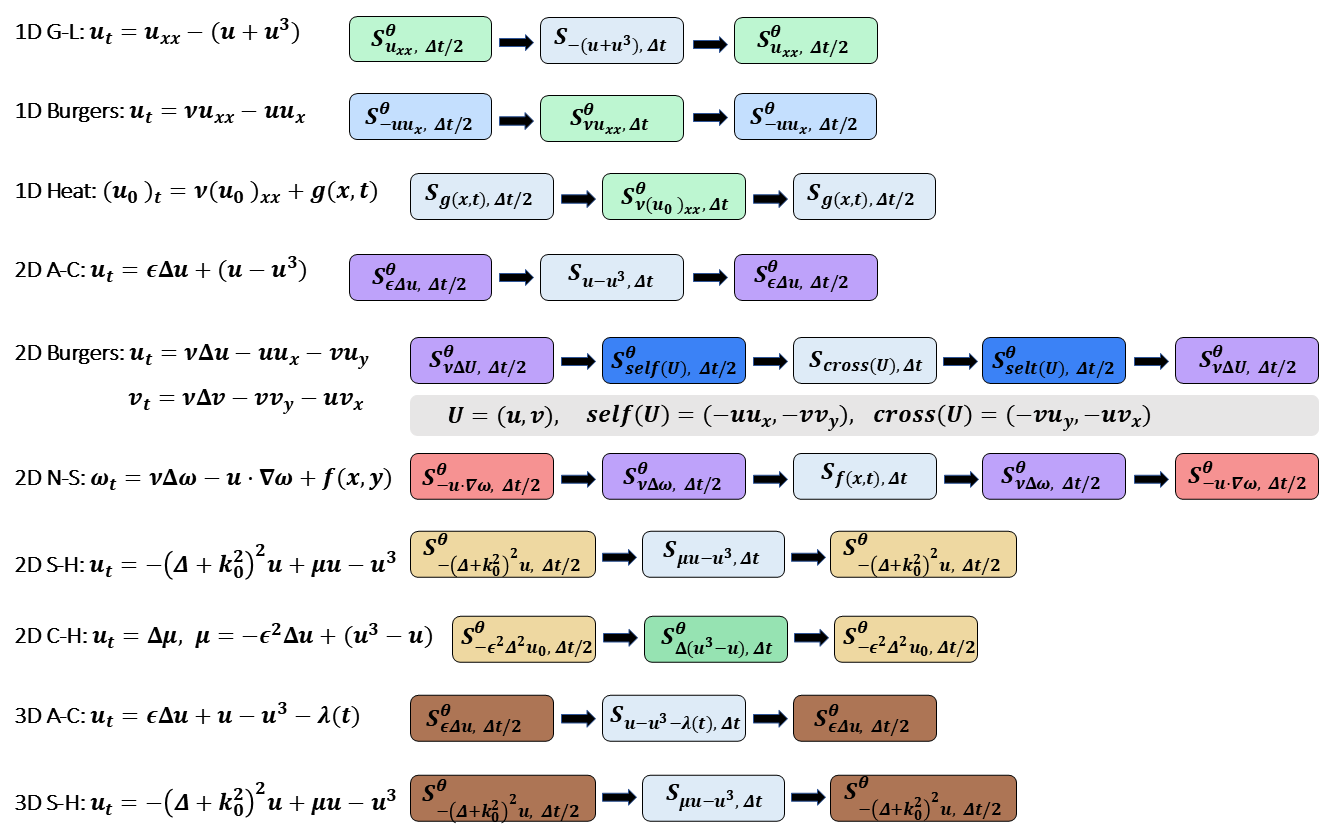}
    \caption{Strang composition recipes.}\label{fig:strang}
  \end{subfigure}

  \caption{\textbf{Overview of LegONet benchmarks.} We report four baseplates together with their structure-preserving operators, the pretrained operator blocks used on each coefficient interface, and the target PDEs solved in Fig. \ref{fig:block}. We also summarize cross-PDE rollout accuracy in Fig. \ref{fig:sum} and list the Strang-splitting composition order used for each experiment in Fig. \ref{fig:strang}.}
  \label{fig:benchmark_big_combined1}
\end{figure}

\begin{figure}[H]
  \centering

  \begin{minipage}[t]{0.49\linewidth}
    \centering
    \begin{minipage}[t]{0.49\linewidth}
      \centering
      \includegraphics[width=\linewidth]{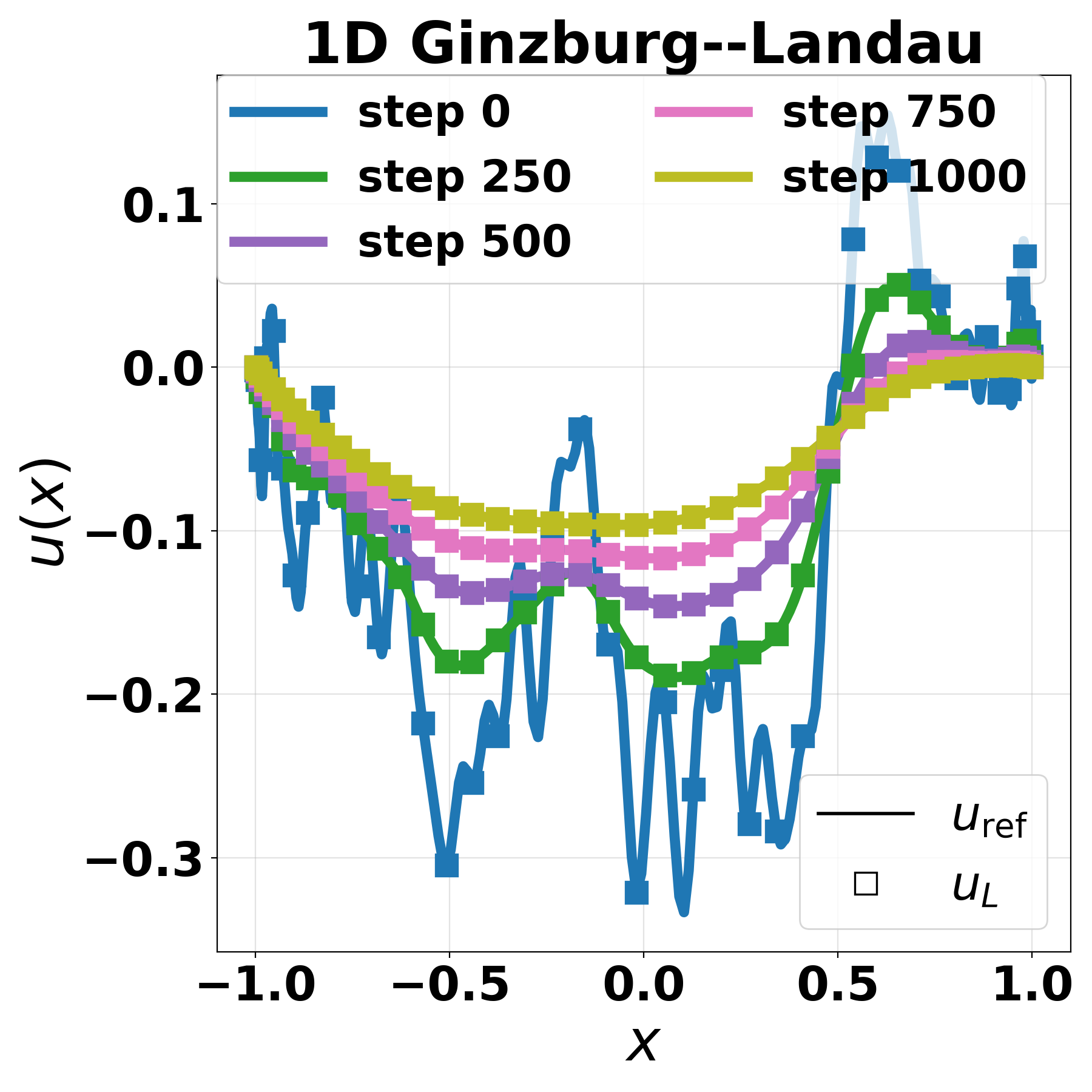}
    \end{minipage}\hfill
    \begin{minipage}[t]{0.49\linewidth}
      \centering
      \includegraphics[width=\linewidth]{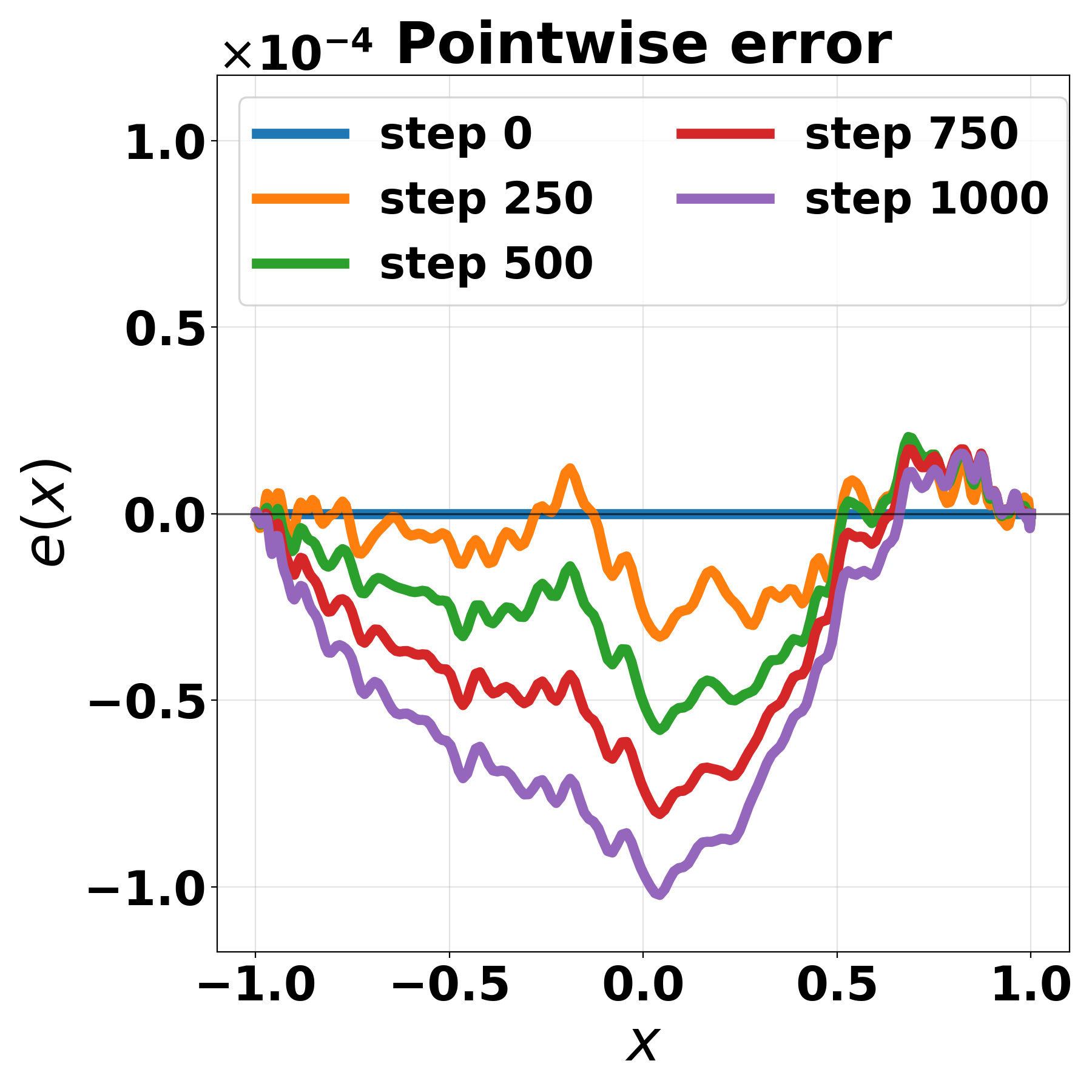}
    \end{minipage}

    \subcaption{1D Ginzburg--Landau
    }
    \label{fig:bench_gl_1d}
  \end{minipage}\hfill
  \begin{minipage}[t]{0.49\linewidth}
    \centering
    \begin{minipage}[t]{0.49\linewidth}
      \centering
      \includegraphics[width=\linewidth]{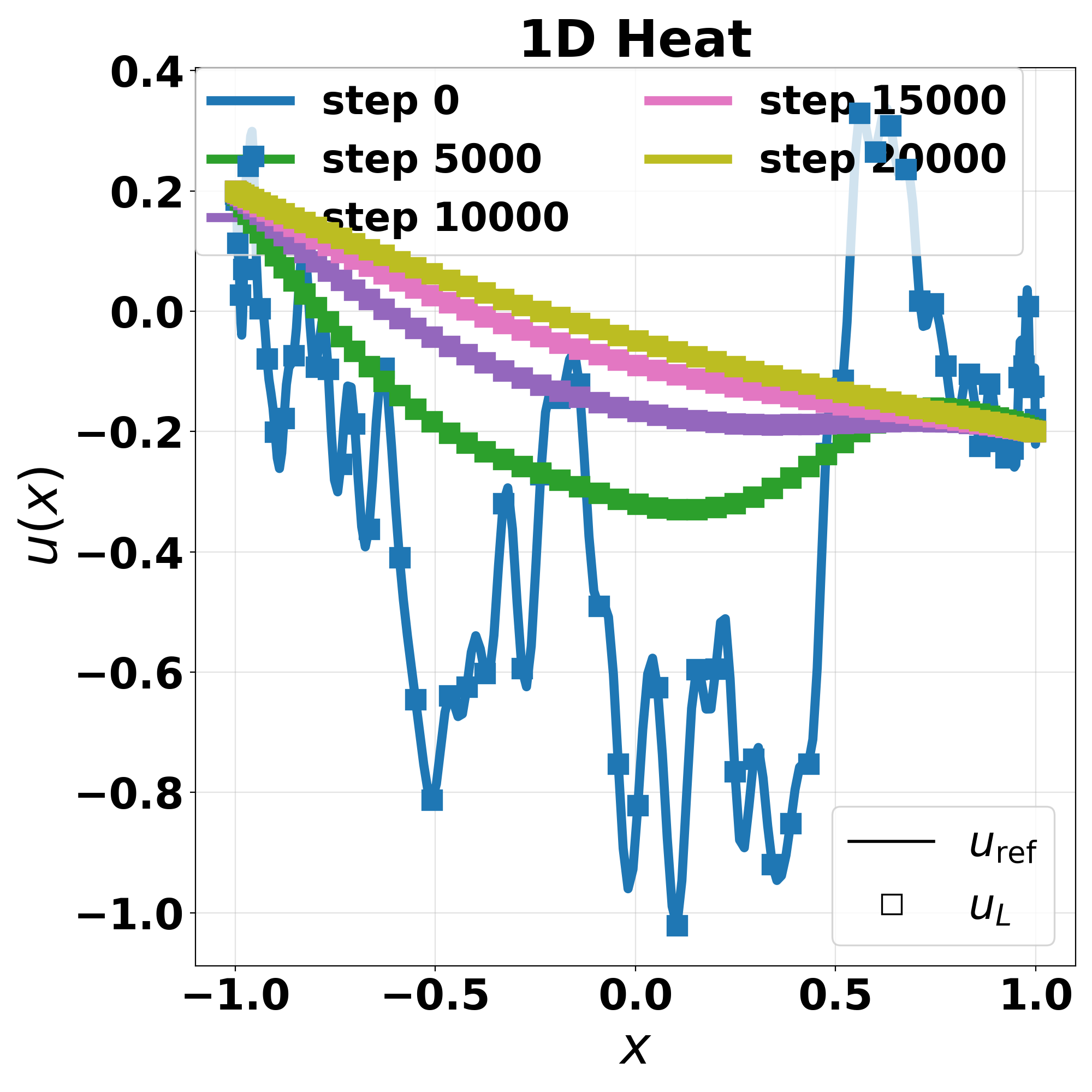}
    \end{minipage}\hfill
    \begin{minipage}[t]{0.49\linewidth}
      \centering
      \includegraphics[width=\linewidth]{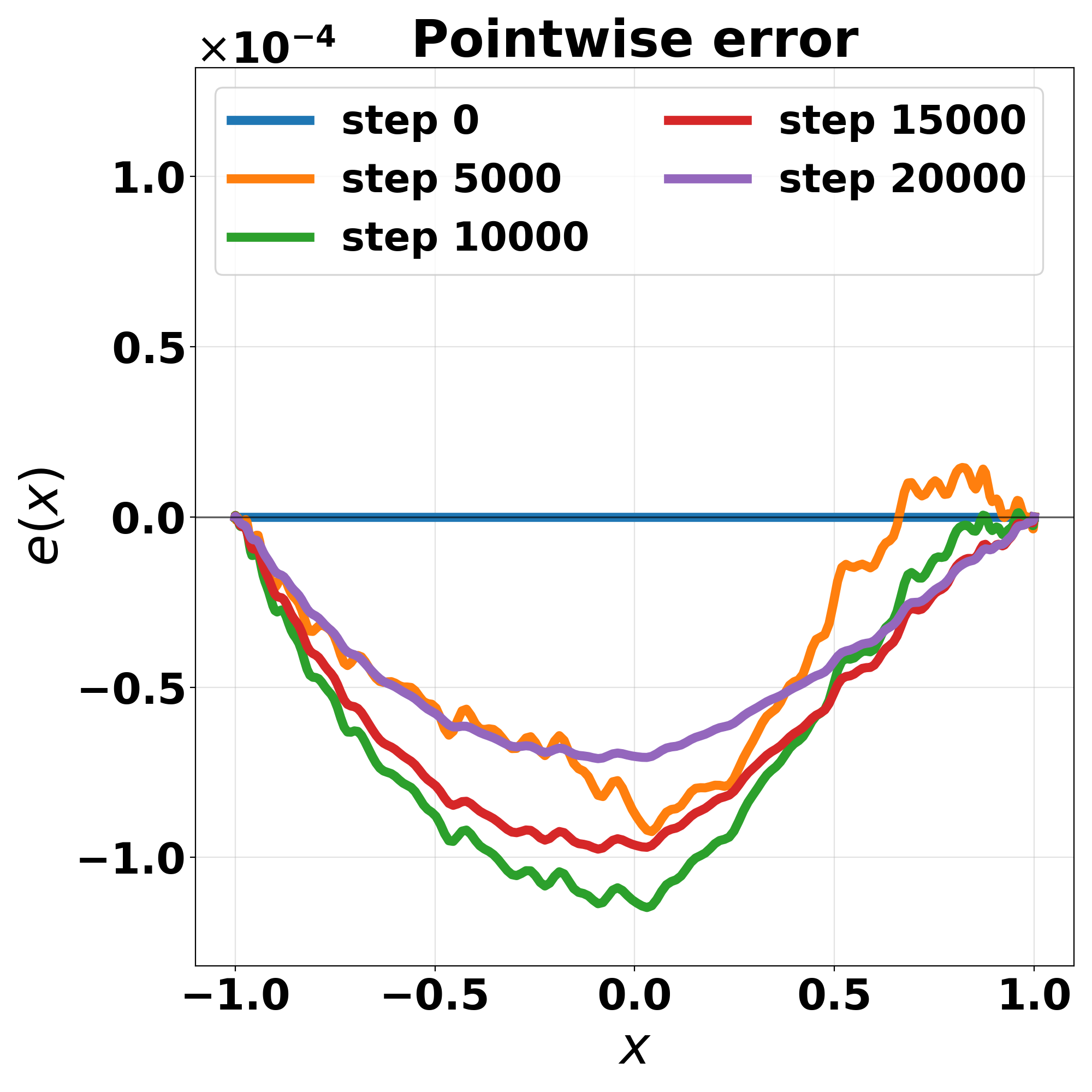}
    \end{minipage}

    \subcaption{1D heat 
    }
    \label{fig:bench_heat_1d}
  \end{minipage}

  \vspace{1.2ex}

  \begin{minipage}[t]{0.49\linewidth}
    \centering
    \includegraphics[width=\linewidth]{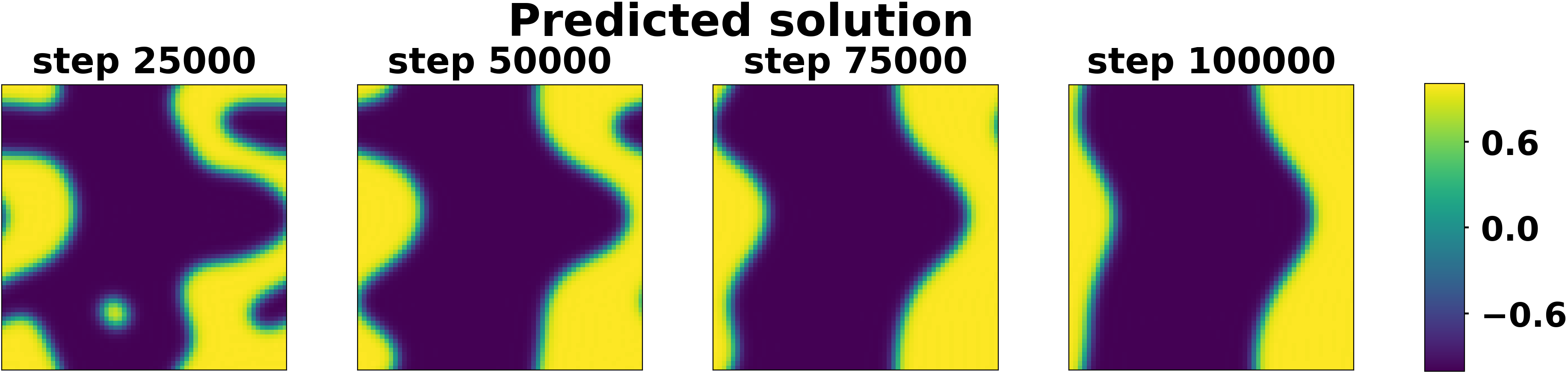}\\
    \includegraphics[width=\linewidth]{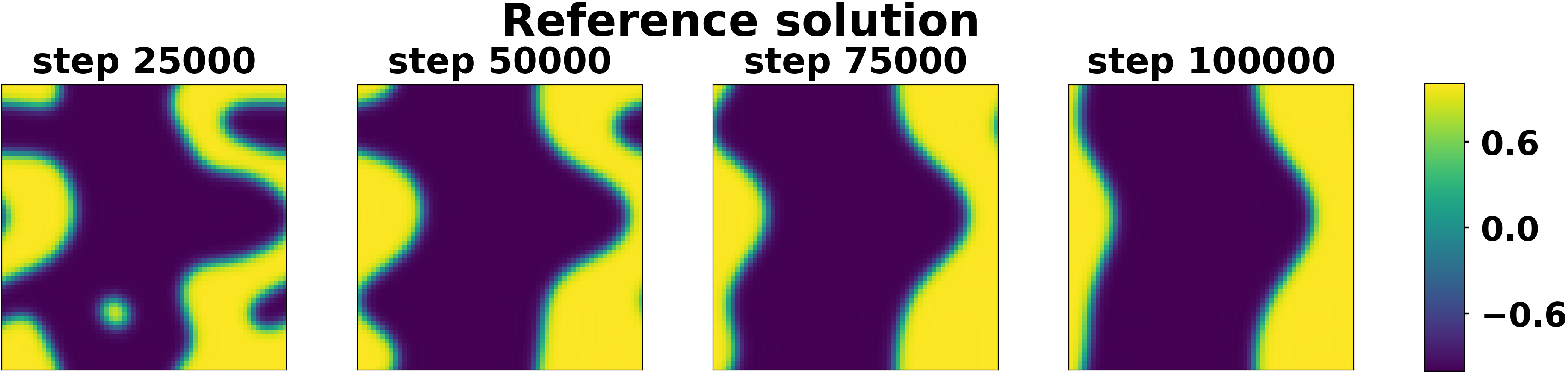}\\
    \includegraphics[width=\linewidth]{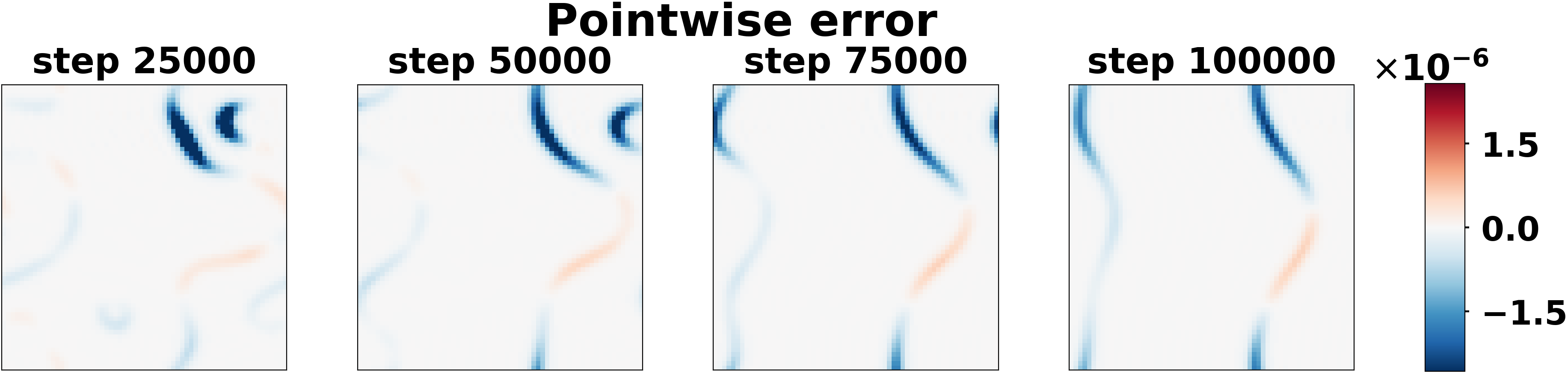}

    \subcaption{2D Allen--Cahn
    }
    \label{fig:bench_ac2d}
  \end{minipage}\hfill
  \begin{minipage}[t]{0.49\linewidth}
    \centering
    \includegraphics[width=\linewidth]{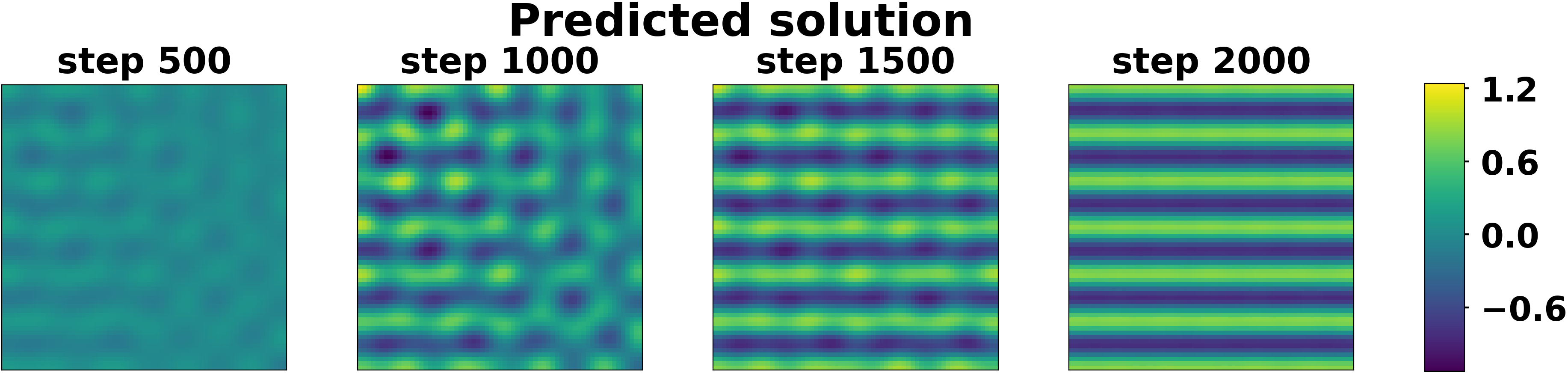}\\
    \includegraphics[width=\linewidth]{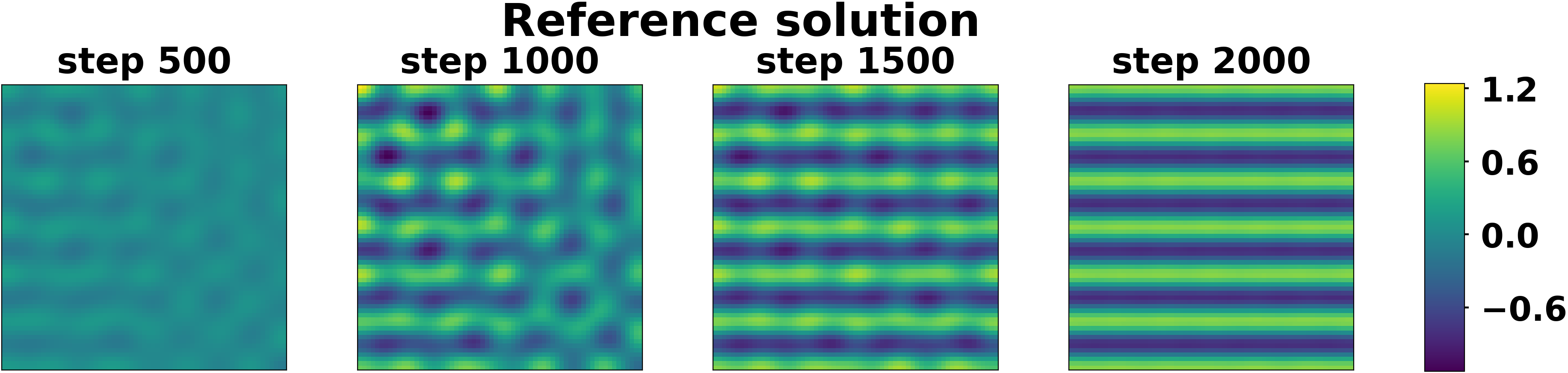}\\
    \includegraphics[width=\linewidth]{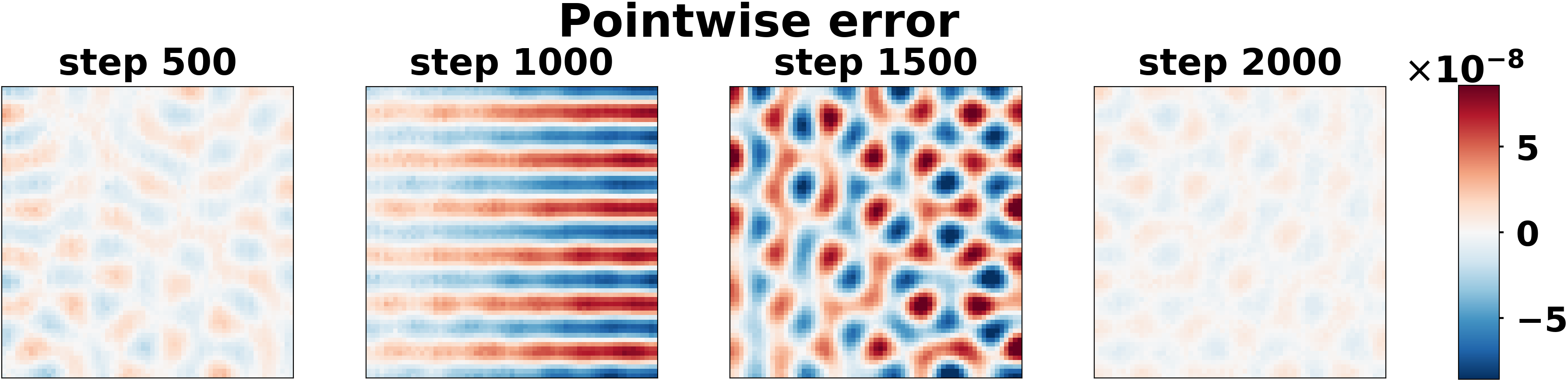}

    \subcaption{2D Swift--Hohenberg
    }
    \label{fig:bench_sh2d}
  \end{minipage}\\[1.0ex]

  \begin{minipage}[t]{0.49\linewidth}
    \centering
    \includegraphics[width=\linewidth]{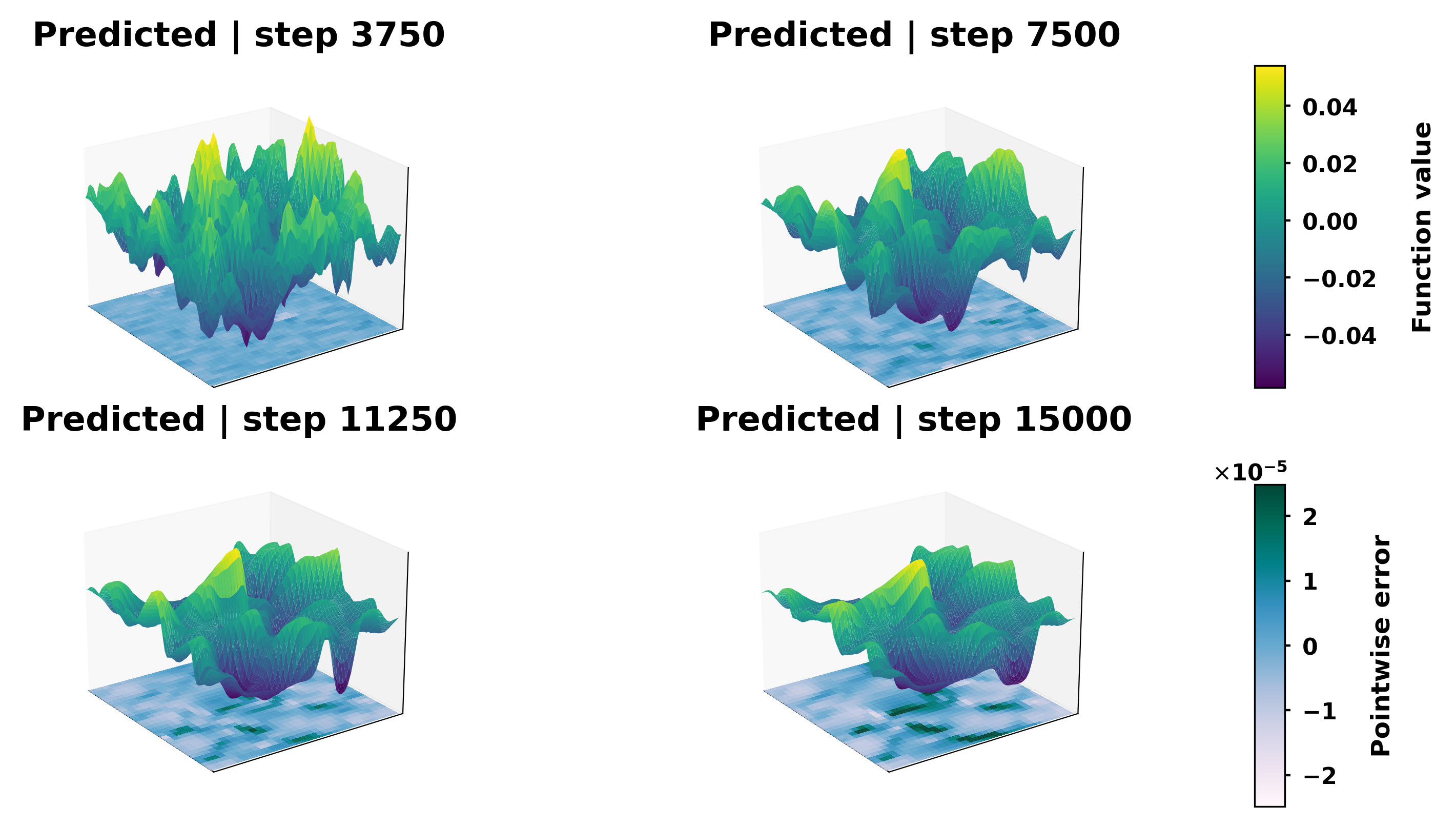}

    \subcaption{2D vector Burgers: $u$ component
    }
    \label{fig:bench_burgers2d_u}
  \end{minipage}\hfill
  \begin{minipage}[t]{0.49\linewidth}
    \centering
    \includegraphics[width=\linewidth]{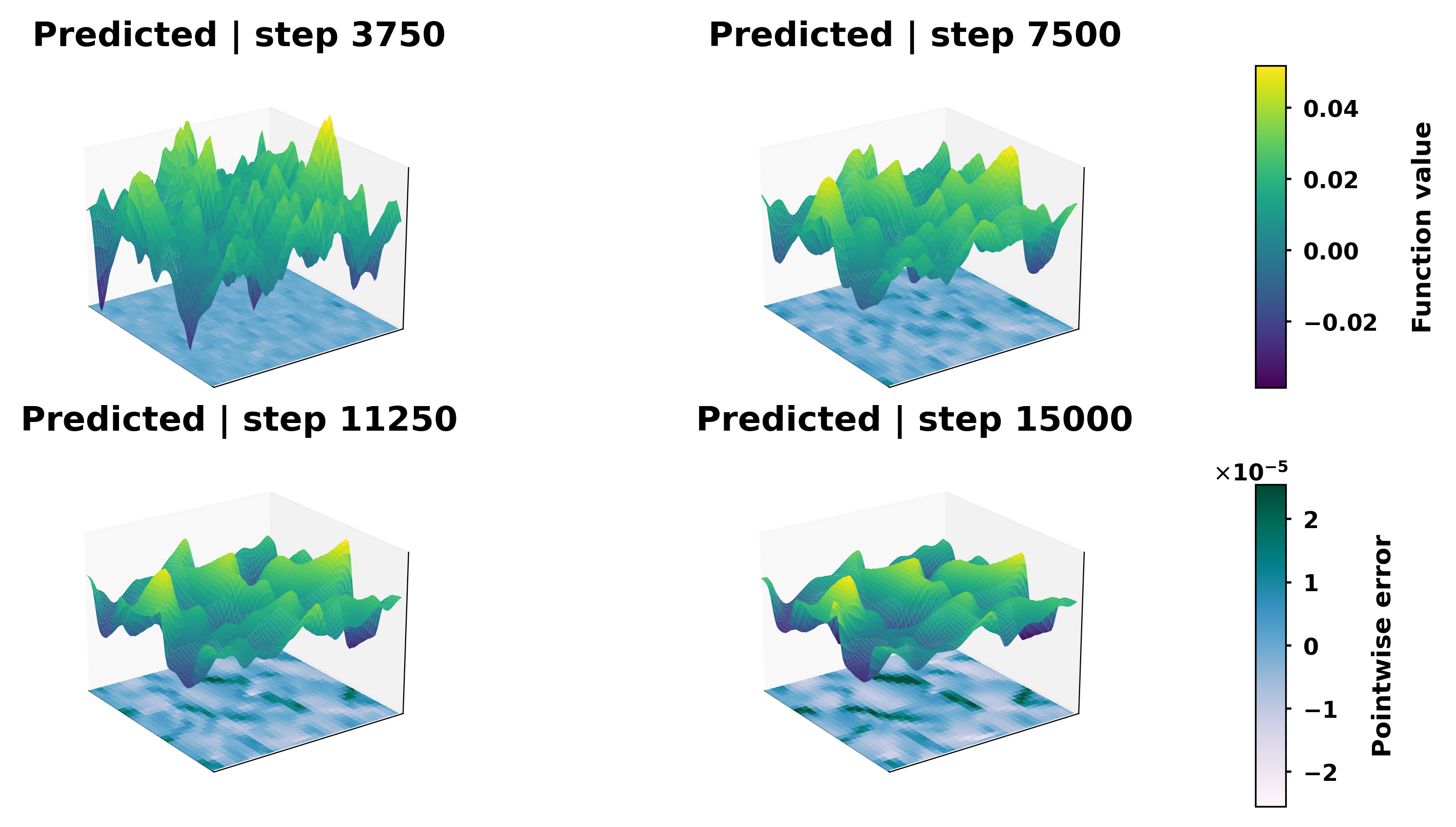}

    \subcaption{2D vector Burgers: $v$ component
    }
    \label{fig:bench_burgers2d_v}
  \end{minipage}\\[1.0ex]

    \begin{minipage}[t]{0.49\linewidth}
    \centering
    \includegraphics[width=\linewidth]{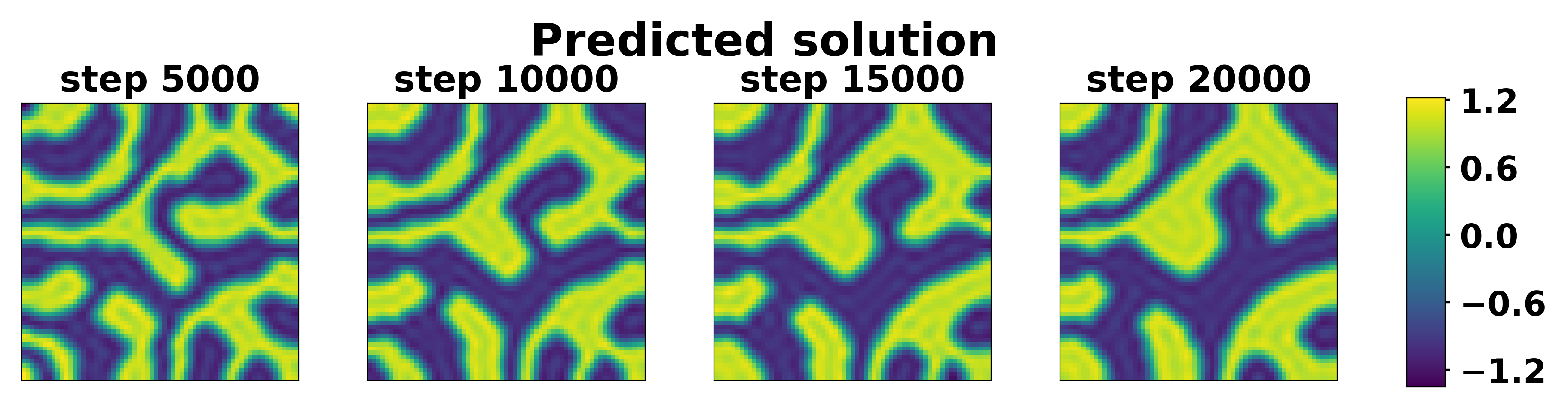}\\
    \includegraphics[width=\linewidth]{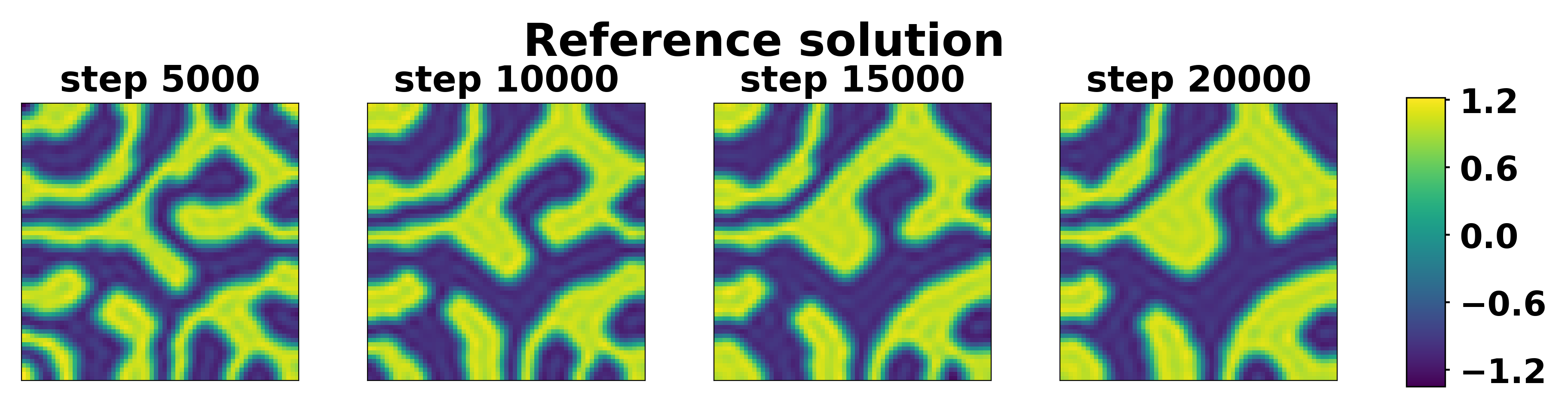}\\
    \includegraphics[width=\linewidth]{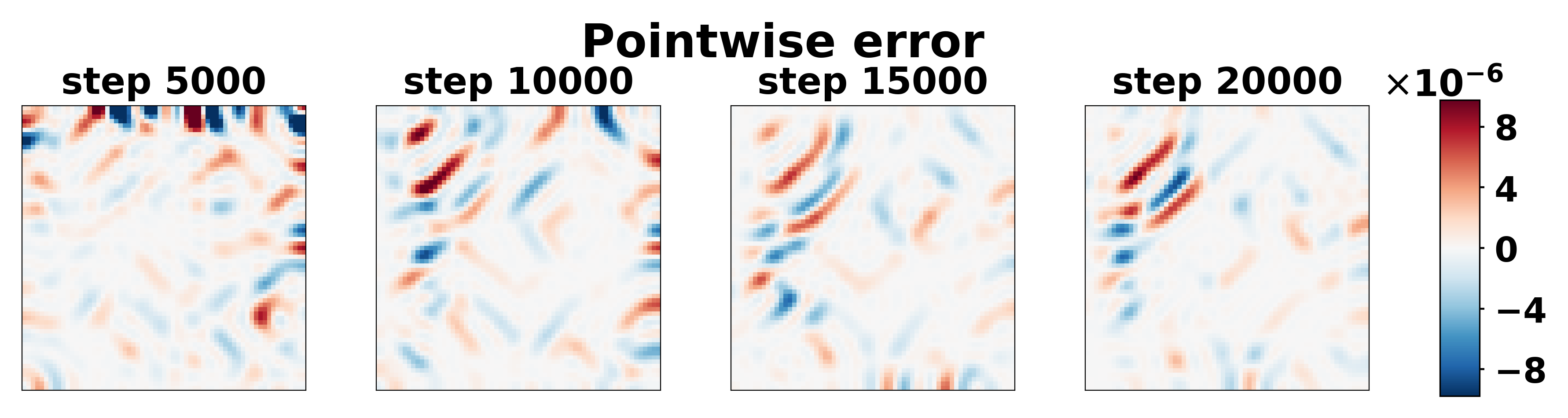}

    \subcaption{2D Cahn--Hilliard
    }
    \label{fig:bench_ch2d}
  \end{minipage}\hfill
  \begin{minipage}[t]{0.49\linewidth}
    \centering
    \includegraphics[width=\linewidth]{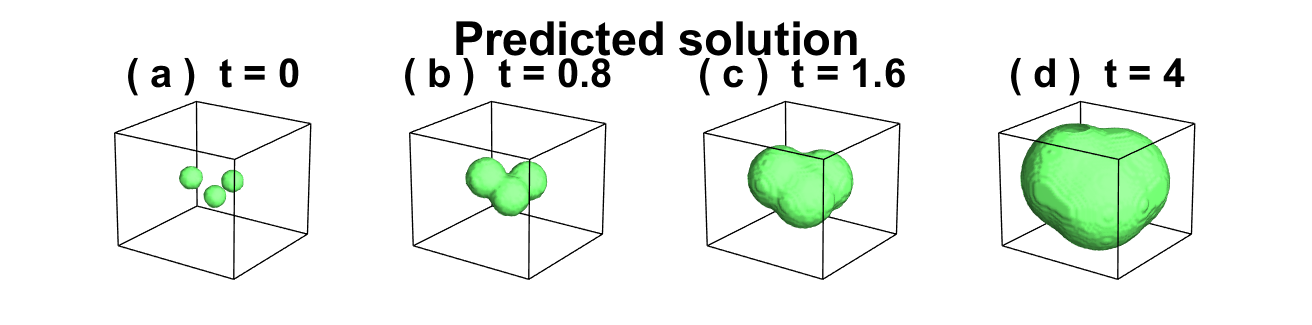}\\
    \includegraphics[width=\linewidth]{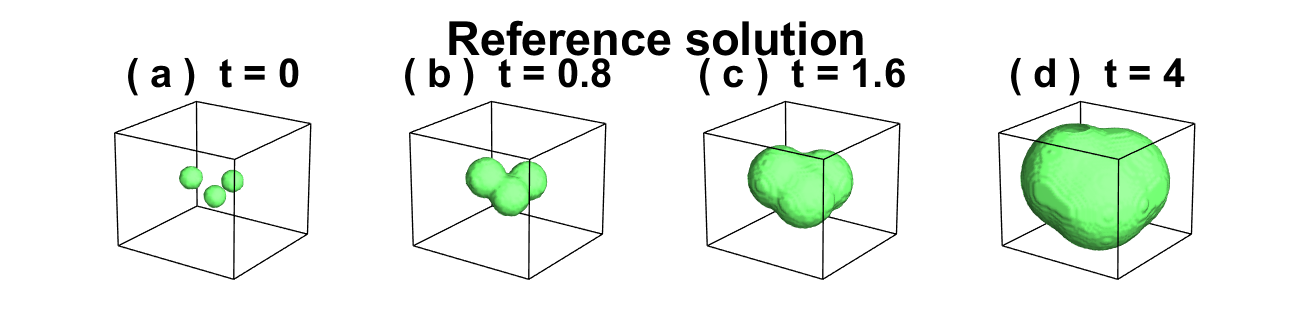}\\
    \includegraphics[width=0.95\linewidth]{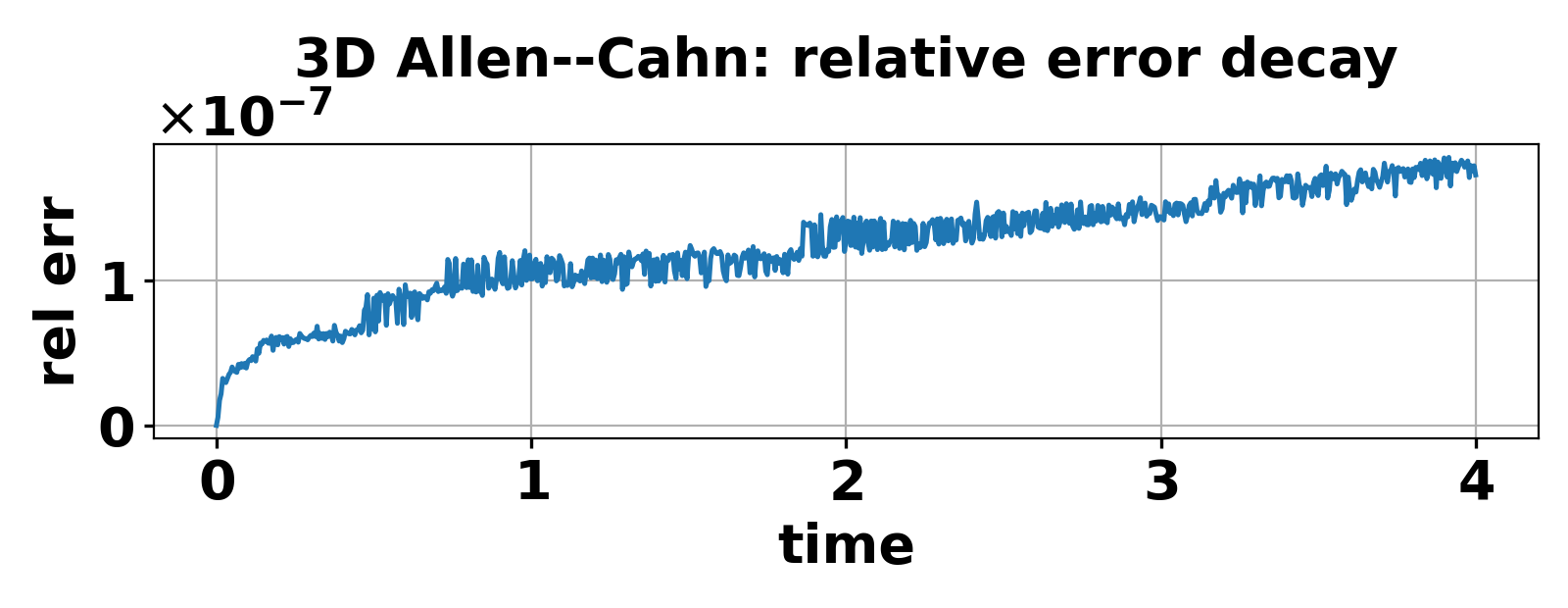}

    \subcaption{3D Allen--Cahn
    }
    \label{fig:bench_ac3d}
  \end{minipage}

  \caption{\textbf{Cross-PDE rollout snapshots.}
  Figs. \ref{fig:bench_gl_1d}--\ref{fig:bench_heat_1d} show 1D Dirichlet results;
  \ref{fig:bench_ac2d}--\ref{fig:bench_burgers2d_v} report 2D periodic Fourier and 2D Neumann cosine rollouts; and
  \ref{fig:bench_ch2d}--\ref{fig:bench_ac3d} summarize the 2D Cahn--Hilliard lifting test and the 3D Allen--Cahn case.
  Each panel reports predicted versus reference states and the associated pointwise error, with the 3D case additionally showing the error curve.}
  \label{fig:benchmark_big_combined}
\end{figure}

LegONet closely tracks the spectral reference throughout the rollout (Fig.~\ref{fig:burgers_overlay}). The normalized pointwise error also remains small over most of the domain (Fig.~\ref{fig:burgers_pointwise}). The largest visible discrepancies appear near the Dirichlet boundaries, where thin boundary layers and spectral truncation concentrate gradients and amplify local interpolation errors. Even in this setting, the predicted solution remains well aligned with the reference trajectory.

The structure diagnostics show that the intended blockwise behavior is preserved after composition. The diffusion block produces monotone energy dissipation, reflected by the nonnegative per-step decay $-\Delta E$ on the vertical axis, while the transport block preserves its Hamiltonian, with the per-half-step drift $\Delta H_x$ remaining at numerical precision across the two half-steps (Figs.~\ref{fig:burgers_dE_substep} and \ref{fig:burgers_dH_substep}).
These results are important because they show that the structure imposed during block pretraining is not lost when the blocks are assembled into a full solver. In other words, the composed dynamics remain interpretable at the level of mechanisms, rather than only at the level of end-to-end rollout.

We next compare LegONet against representative baselines from the two dominant paradigms: supervised operator-learning time-steppers (FNO and DeepONet) and a physics-informed space–time PINN, all trained under comparable parameter and supervision budgets (Supplementary Information).
Closed-loop weighted relative $L^2$ error and kinetic-energy diagnostics, defined by
$E(t):=\tfrac12\|\mathbf u(\cdot,t)\|_{w,2}^2$,
are shown in Figs.~\ref{fig:burgers_wrel} and \ref{fig:burgers_energy}.
%Closed-loop weighted relative $L^2$ error and kinetic-energy diagnostics are shown in Figs.~\ref{fig:burgers_wrel} and \ref{fig:burgers_energy}.
LegONet exhibits the most stable rollout and the smallest energy deviation. By contrast, the supervised operator learners accumulate drift over time, while the PINN baseline exhibits a persistent bias consistent with fitting a global space–time surrogate.

\subsection*{Case study II: turbulence requires structured primitives under long-horizon composition}

We next ask whether structured operator blocks remain important in a regime where long-horizon composition is genuinely challenging. Turbulent dynamics provide a stringent test: diffusion, nonlinear transport and Poisson inversion must interact over many time steps, and small errors can accumulate rapidly. If LegONet is to offer more than modularity alone, then its structure-preserving primitives should matter most in precisely this setting.

\begin{figure}[H]
  \centering

  % ---------------- Row 1 ----------------
  \begin{subfigure}[t]{0.49\linewidth}
    \centering
    \includegraphics[width=\linewidth]{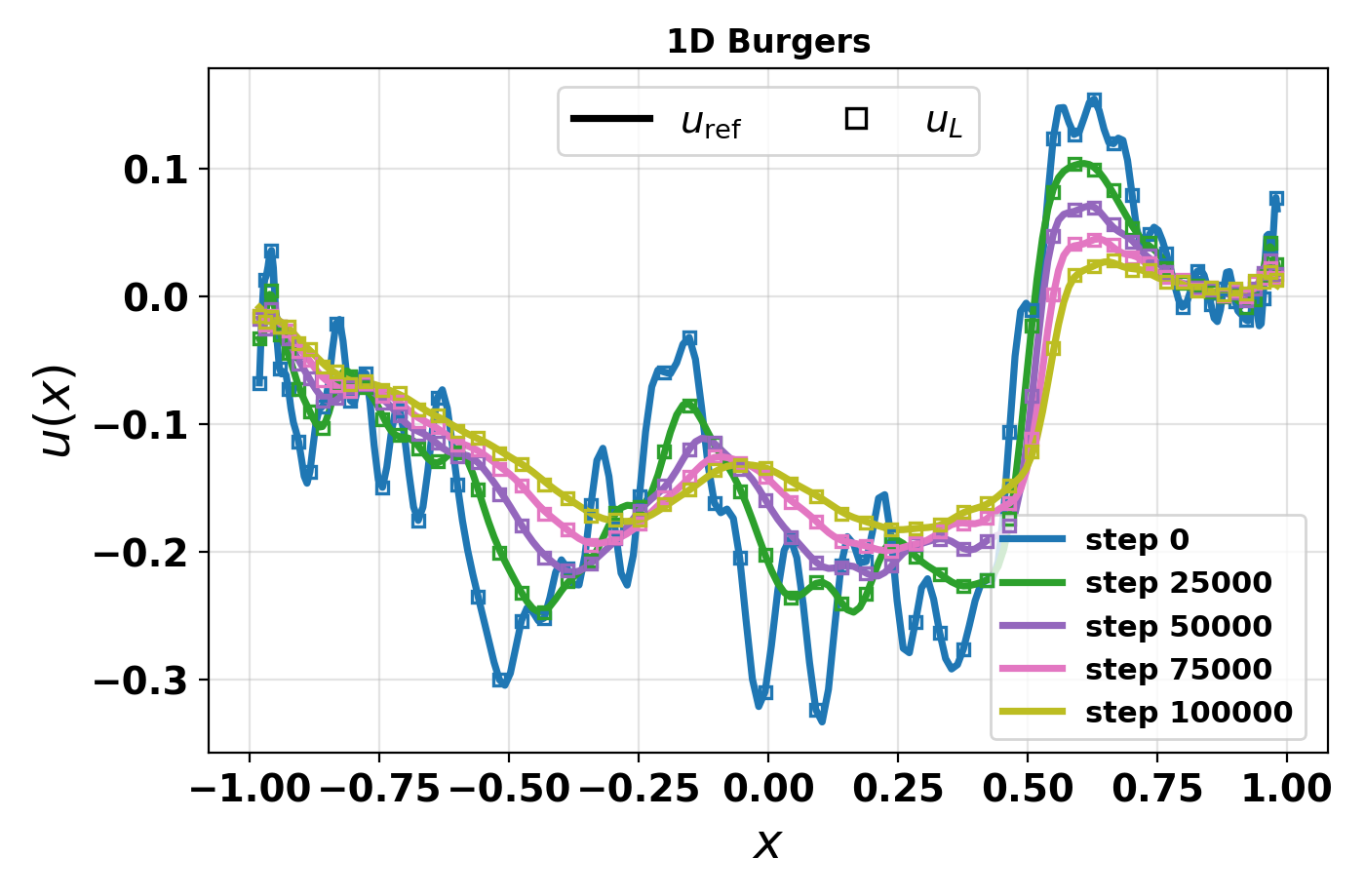}
    \caption{Reference vs.\ LegONet snapshots at selected times.}
    \label{fig:burgers_overlay}
  \end{subfigure}\hfill
  \begin{subfigure}[t]{0.49\linewidth}
    \centering
    \includegraphics[width=\linewidth]{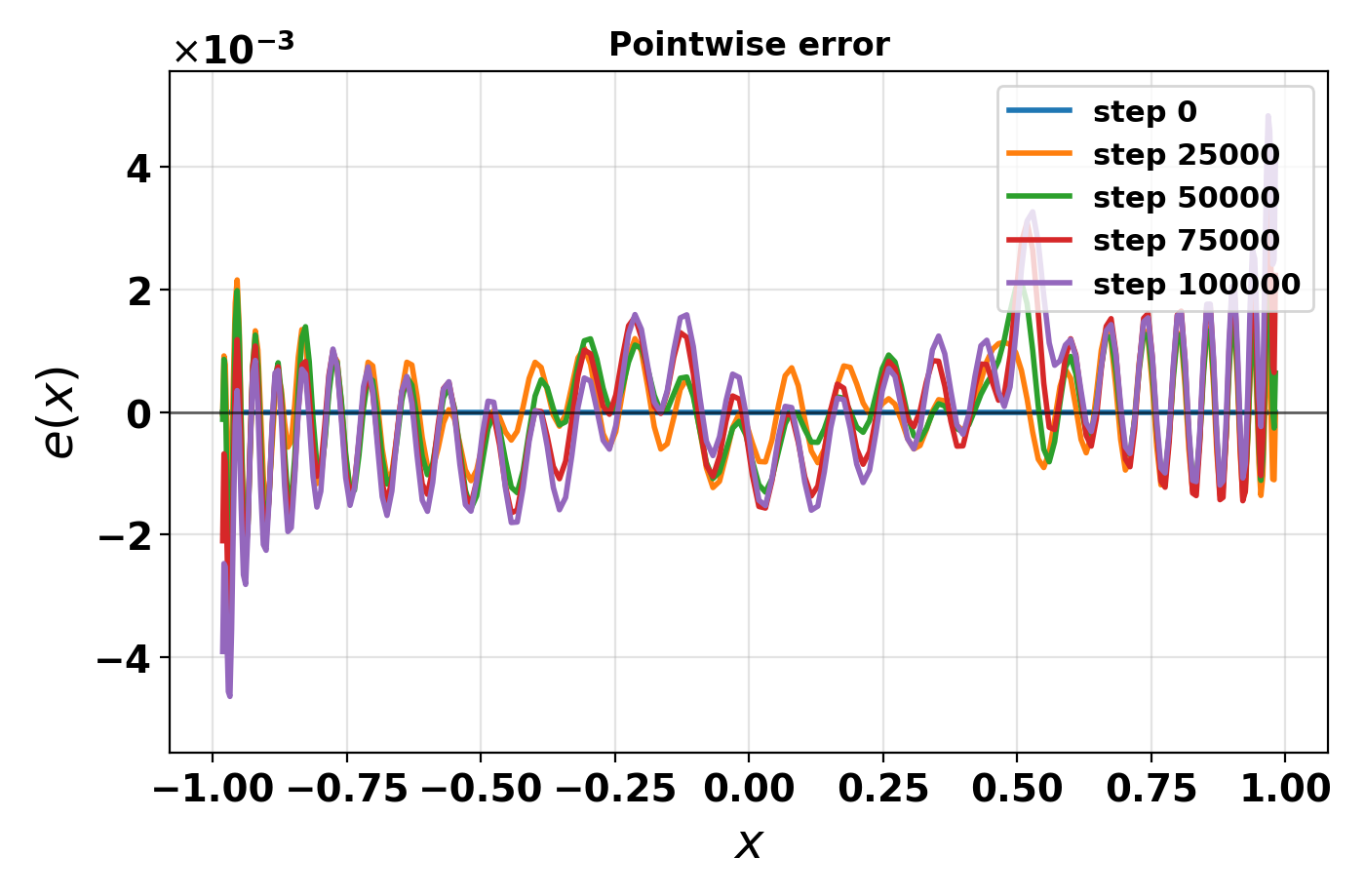}
    \caption{Normalized pointwise error $e(x)$ at selected times.}
    \label{fig:burgers_pointwise}
  \end{subfigure}

  \vspace{0.6em}

  % ---------------- Row 2 ----------------
  \begin{subfigure}[t]{0.49\linewidth}
    \centering
    \includegraphics[width=\linewidth]{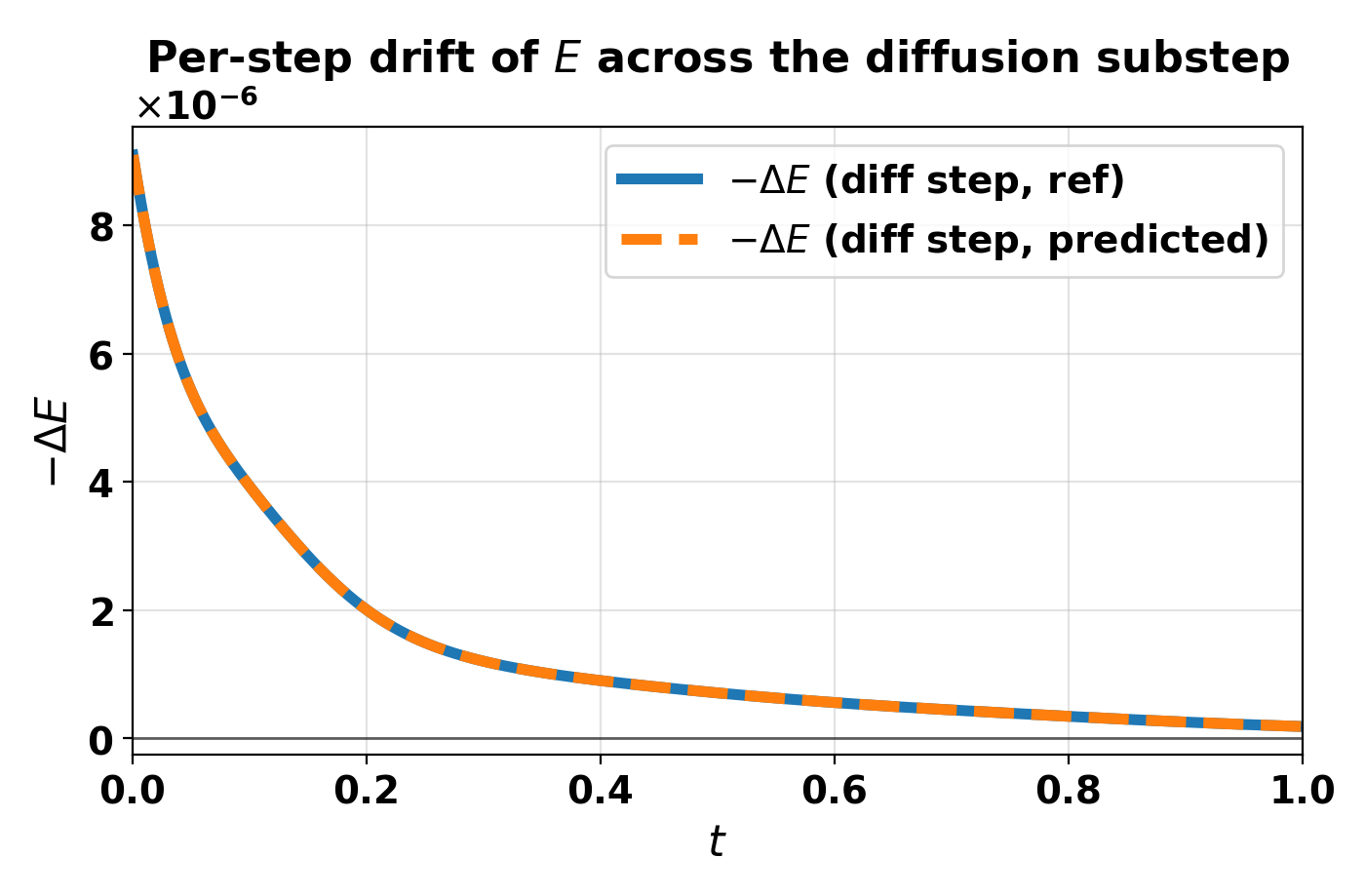}
    \caption{E-block diagnostic: $-\Delta E$ across the diffusion substep.}
    \label{fig:burgers_dE_substep}
  \end{subfigure}\hfill
  \begin{subfigure}[t]{0.49\linewidth}
    \centering
    \includegraphics[width=\linewidth]{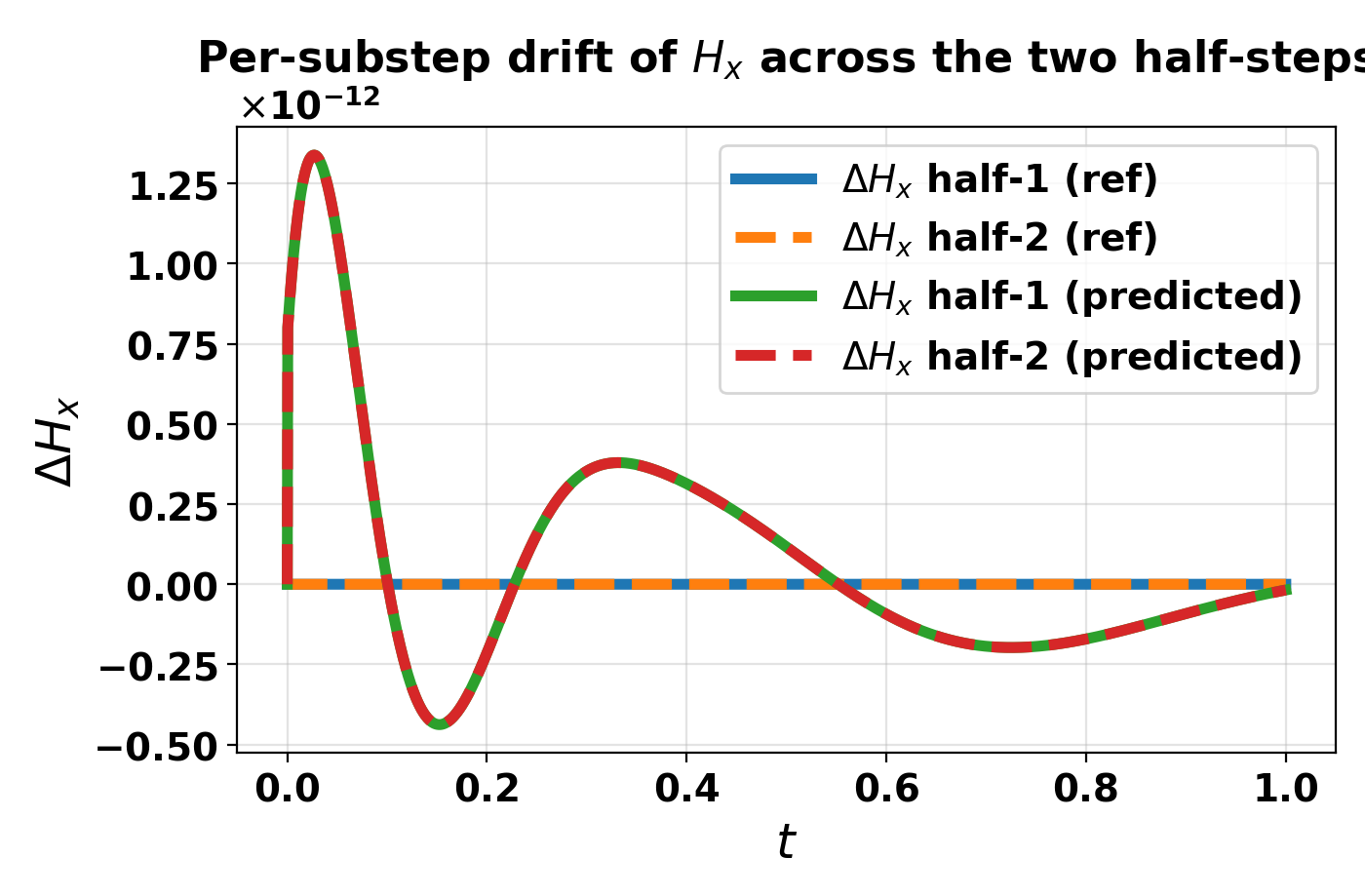}
    \caption{H-block diagnostic: per-substep drift of $H_x$ across the two half-steps.}
    \label{fig:burgers_dH_substep}
  \end{subfigure}

  \vspace{0.6em}

  % ---------------- Row 3 ----------------
  \begin{subfigure}[t]{0.49\linewidth}
    \centering
    \includegraphics[width=\linewidth]{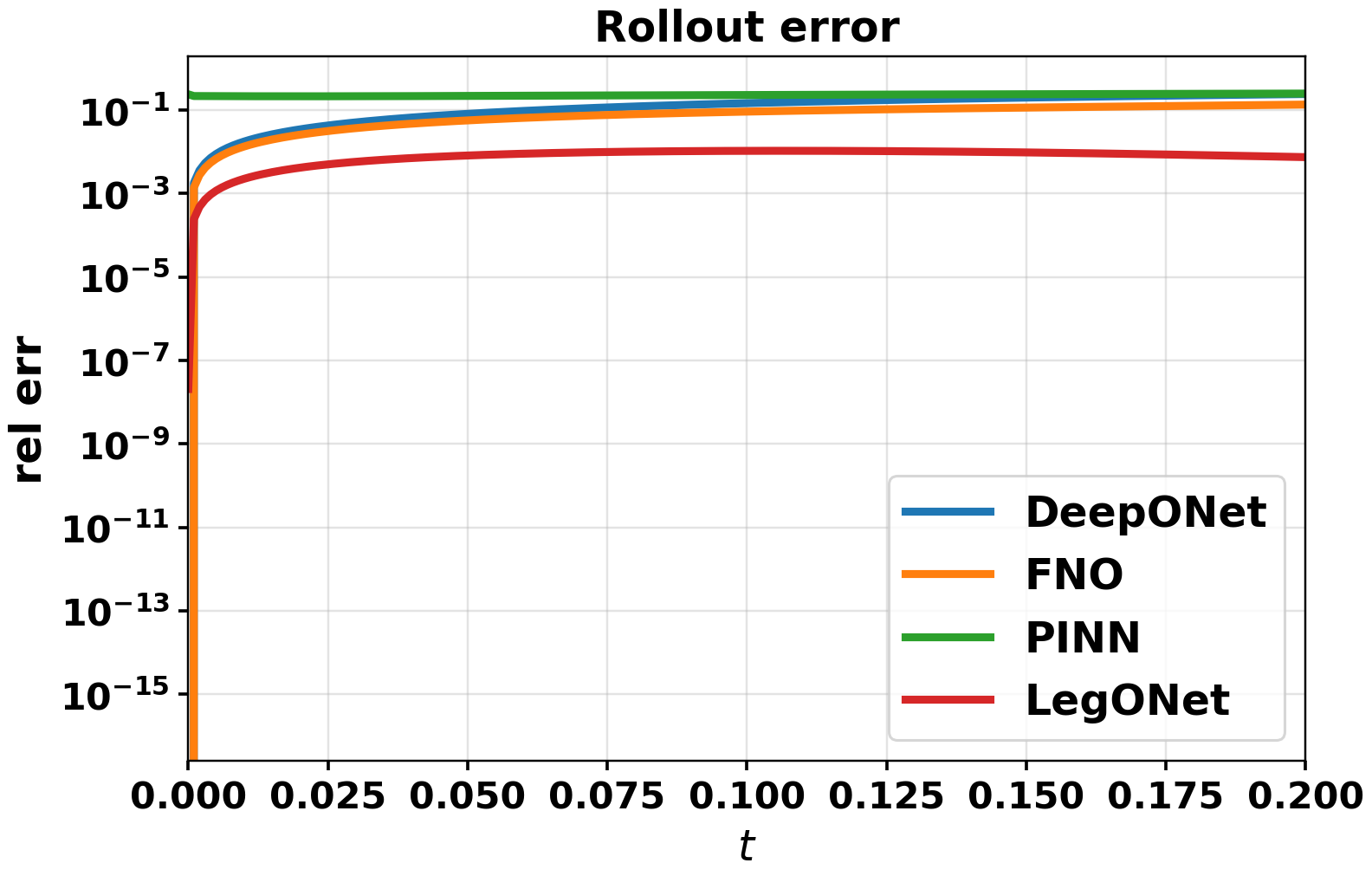}
    \caption{Closed-loop rollout error (weighted relative $L^2$).}
    \label{fig:burgers_wrel}
  \end{subfigure}\hfill
  \begin{subfigure}[t]{0.49\linewidth}
    \centering
    \includegraphics[width=\linewidth]{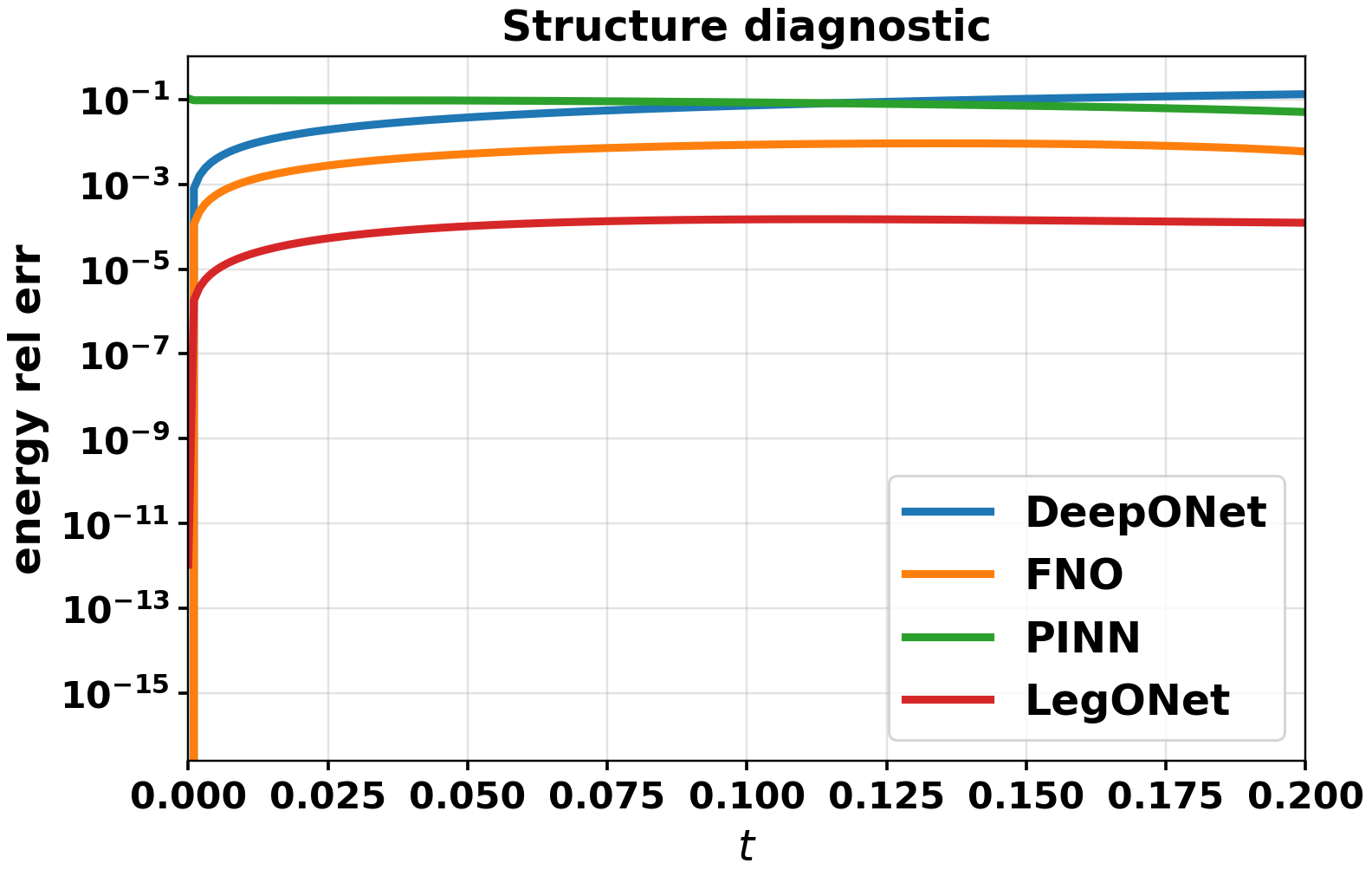}
    \caption{Relative energy error over time.}
    \label{fig:burgers_energy}
  \end{subfigure}

  \caption{\textbf{1D Burgers under Dirichlet boundaries.}
Top: reference and LegONet snapshots with normalized pointwise error.
Middle: substep diagnostics for the assembled E- and H-blocks.
Bottom: closed-loop rollout accuracy and energy diagnostic comparing LegONet with PINN, FNO, and DeepONet under the same evaluation protocol.}
  \label{fig:1d_burgers}
\end{figure}

We consider the 2D incompressible Navier--Stokes equation in vorticity form on the periodic torus $\Omega=[0,2\pi)^2$,
\begin{equation}\label{eq:NS_vort_rewrite_res}
  \omega_t + u\cdot\nabla\omega
  = \nu\,\Delta\omega + f(x,y),\qquad \nabla\cdot u=0,
\end{equation}
with periodic boundary conditions in both $x$ and $y$.
We parameterize the divergence-free velocity by a streamfunction $u=(\psi_y,-\psi_x)$ with
$-\Delta\psi=\omega$ on $\Omega$.
We use Kolmogorov forcing $f(x,y)=0.1(\sin(x+y)+\cos(x+y))$,
viscosity $\nu=10^{-4}$, step size $\Delta t=10^{-3}$, and final time $T=50$.
On the Fourier baseplate, LegONet assembles a Laplacian block $\Delta$, a Poisson-inversion block $\Delta^{-1}$, and transport primitives, and advances coefficients by symmetric splitting (Fig.~\ref{fig:strang}).

For the linear primitives, we use structured quadratic generators on retained modes:
$E_{\Delta}^{a,\boldsymbol{\theta}}(\mathbf a)=\tfrac12 \mathbf a^\top C_{\Delta}^{\boldsymbol{\theta}} \mathbf a$ with $F_{\Delta}^{\boldsymbol{\theta}}(\mathbf a)=-G\nabla_a E_{\Delta}^{a,\boldsymbol{\theta}}$,
and $H_{\Delta^{-1}}^{a,\boldsymbol{\theta}}(\mathbf a)=\tfrac12 \mathbf a^\top C_{\Delta^{-1}}^{\boldsymbol{\theta}} \mathbf a$ with $F_{\Delta^{-1}}^{\boldsymbol{\theta}}(\mathbf a)=\nabla_a H_{\Delta^{-1}}^{a,\boldsymbol{\theta}}$,
where $C_{\Delta}^{\boldsymbol{\theta}}$ and $C_{\Delta^{-1}}^{\boldsymbol{\theta}}$ are diagonal in the Fourier basis. This construction keeps the linear mechanisms aligned with the spectral structure of the baseplate while preserving the shared coefficient representation used for block composition.

Fig.~\ref{fig:ns_steps} compares predicted and reference vorticity fields together with normalized pointwise error at representative times. Despite the turbulent regime and the long rollout horizon, LegONet remains stable and accurate throughout the simulation, with relative error staying below $4\%$ over $T=50$. The predicted vortical structures remain visually well matched to the reference solution, indicating that the composed blocks preserve the dominant flow dynamics even after many thousands of update steps.

To isolate the role of structure, we replace the Laplacian E-block with a parameter-matched but unconstrained neural update, denoted LegONet-unconstrained.
This ablation removes the generator-induced form $-G\nabla E$ while keeping model capacity fixed. The effect is immediate: rollout error rises sharply and the energy diagnostic
$E(t):=\tfrac12\|\boldsymbol{\omega}(\cdot,t)\|_{w,2}^2$
exhibits substantially larger drift (Fig.~\ref{fig:ns_rollout_and_energy}).
Because parameter counts are identical, this degradation cannot be attributed to insufficient capacity. Instead, it shows that in turbulent long-horizon composition, stable behavior depends on mechanism-level structure rather than on expressive function approximation alone.

We also compare against supervised operator-learning baselines under the same closed-loop evaluation protocol. FNO and DeepONet are included, whereas PINN baselines are omitted because optimizing a global space–time surrogate in this chaotic regime does not produce stable rollouts (Supplementary Information). As shown in Fig.~\ref{fig:ns_rollout_and_energy}, LegONet maintains substantially better long-horizon stability than the supervised baselines. This experiment therefore serves as a strong test of the central claim of the paper: in demanding multimechanism settings, plug-and-play composition is not enough by itself. Structured primitives are essential.

\subsection*{Case study III: stiff higher-order operators reuse the same primitives across dimensions}
We finally test whether LegONet remains effective when operator complexity, stiffness and dimensionality increase at the same time. This is a demanding setting for compositional learning. Higher-order operators require repeated block application, stiff dynamics amplify rollout errors, and three-dimensional pattern formation places additional pressure on both stability and expressivity. If LegONet is to support plug-and-play scientific solvers, then its block-based formulation should remain effective under all three challenges.

We consider the 3D Swift--Hohenberg equation on the three-dimensional torus $\Omega=[0,2\pi)^3$, with periodic boundary conditions in each coordinate direction,
\begin{equation}\label{eq:SH3D_PDE_rewrite_res}
  u_t = -(\Delta+k_0^2)^2 u + \mu u - u^3,
\end{equation}
with $\mu=0.6$, $k_0=6$, time step $\Delta t=5\times10^{-4}$, and final time $T=30$.
The stiff linear term requires two successive Laplacian applications.
LegONet reuses the pretrained 3D Laplacian primitive in the same E-block form as in 2D, i.e., a dissipative update $-G\nabla_a E_{\Delta}^{a,\boldsymbol{\theta}}$ on the shared coefficient state,
and composes the repeated Laplacian calls with the cubic reaction under symmetric splitting (Fig.~\ref{fig:strang}).

\begin{figure}[H]
\centering
\begin{subfigure}[b]{0.63\textwidth}
  \centering
  \includegraphics[width=\linewidth,
    trim={0cm 1cm 0cm 0cm},
    clip]{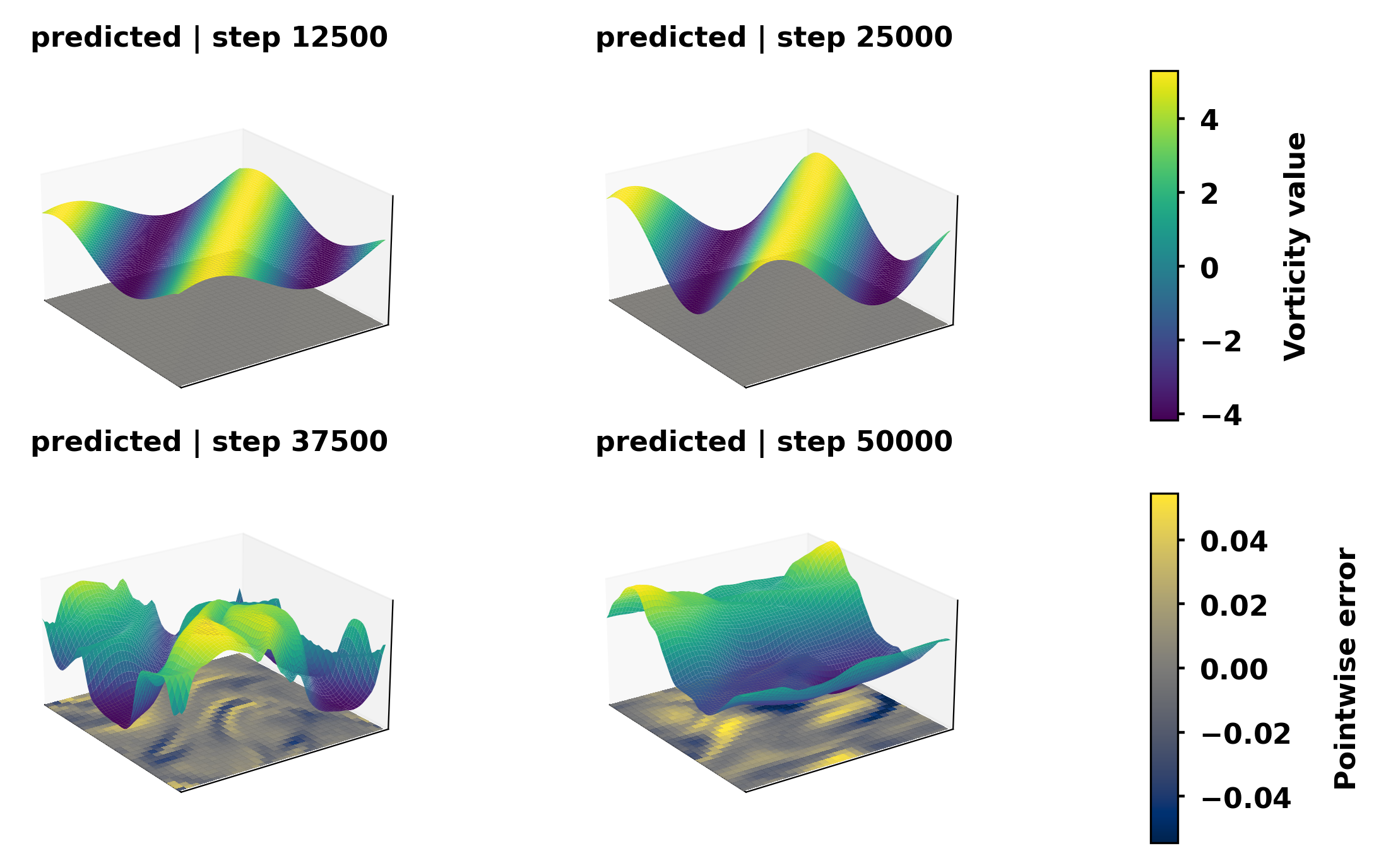}
  \caption{Vorticity snapshots (predicted / reference / normalized pointwise error) at steps $12500$, $25000$, $37500$, and $50000$.}
  \label{fig:ns_steps}
\end{subfigure}\hfill
\begin{subfigure}[b]{0.34\textwidth}
  \centering
  \includegraphics[width=\linewidth]{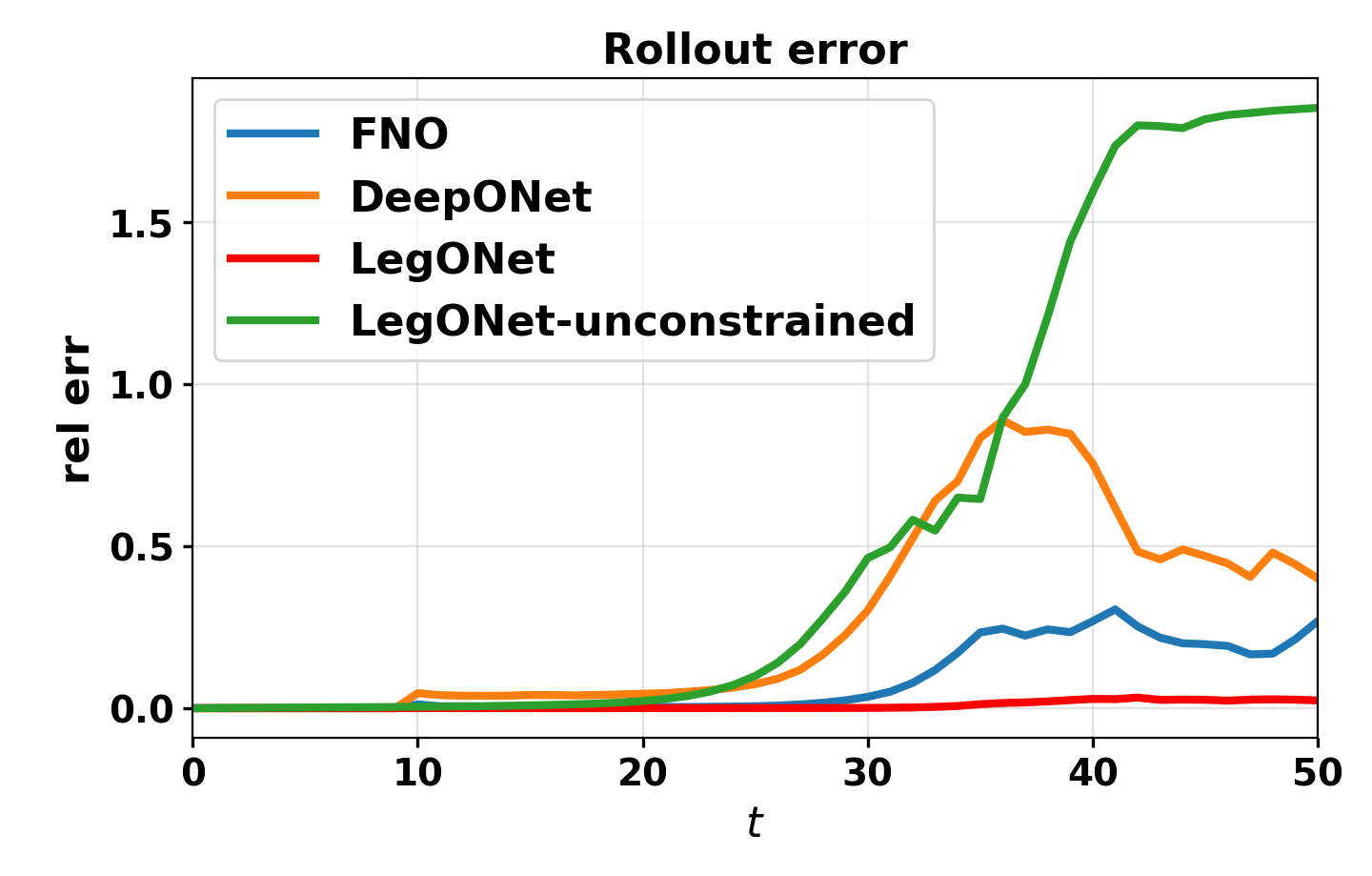}

  \vfill

  \includegraphics[width=\linewidth]{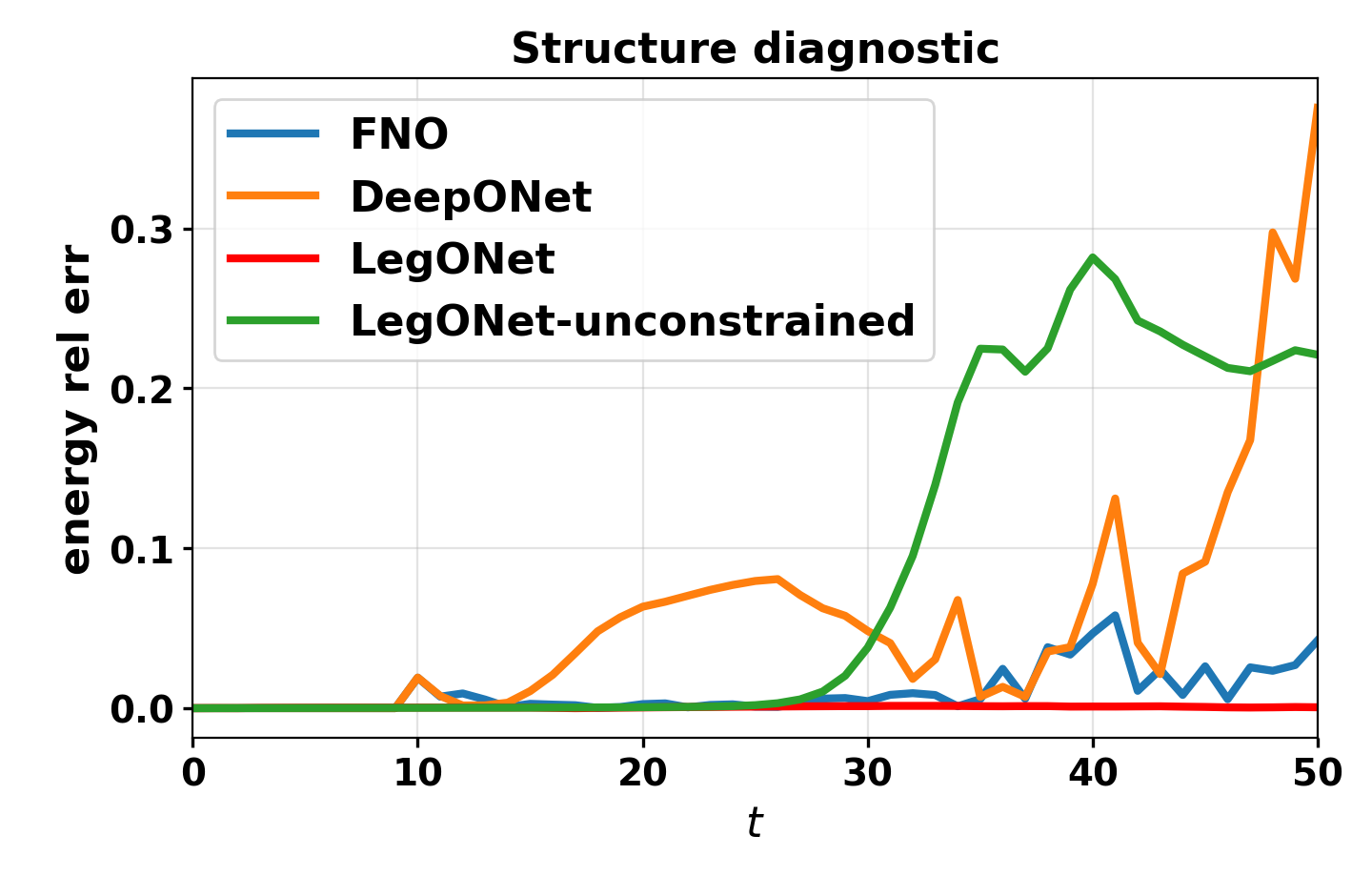}
  \caption{Closed-loop rollout error and energy diagnostic.}
  \label{fig:ns_rollout_and_energy}
\end{subfigure}
\caption{\textbf{2D forced Navier--Stokes: long-horizon turbulent rollouts.}
Left: predicted and reference vorticity snapshots with normalized pointwise error.
Right: closed-loop rollout error and energy diagnostic comparing LegONet, an unstructured diffusion-block ablation, FNO, and DeepONet over $t\in[0,50]$.}
\label{fig:ns_steps_and_relerr}
\end{figure}

Figs.~\ref{fig:sh3d_ref_iso} and \ref{fig:sh3d_learn_iso} compare multi-level isosurfaces of reference and LegONet solutions using identical levels and a shared color scale.
The two trajectories are visually well aligned throughout the rollout. In both cases, the solution evolves from random initial structure toward coherent pattern formation, and the resulting morphology is difficult to distinguish by eye. This agreement is notable because the solver handles a stiff higher-order operator through repeated composition of pretrained mechanism blocks on the shared representation.

We further examine robustness to shifts in the initial-condition distribution. Three settings are considered: in-distribution Gaussian fields, a mild out-of-distribution shift that increases high-frequency content (OOD1), and a stronger shift based on piecewise-constant fields with sharp interfaces (OOD2). 
The full construction of the OOD priors and sampling details are provided in the Supplementary Information.
As shown in Fig.~\ref{fig:sh3d_relerr}, rollout errors remain at the $10^{-5}$ level for the in-distribution setting and OOD1, and remain bounded at the $10^{-4}$ level for OOD2 at $T=30$. These results indicate that the reused primitives are not narrowly tied to a single initialization regime, but retain accuracy under moderate changes in the input prior.

Because this problem combines stiffness, higher-order dynamics and three spatial dimensions, not all baseline paradigms produce reliable closed-loop rollouts under the matched evaluation protocol. PINN and DeepONet baselines are therefore omitted, and we compare only against FNO. As summarized in Fig.~\ref{fig:sh3d_fno_vs_lego_bar}, FNO errors reach approximately $40\%$, whereas LegONet maintains accuracy at the $10^{-4}$ level over long horizons. This final experiment shows that once mechanism blocks are trained on a given baseplate, they can be composed repeatedly to realize more complex operators without retraining a monolithic solver, while preserving the same structured block design across dimensions.

\section*{Discussion}

This work introduces a compositional framework for PDE learning in which plug-and-play operator blocks are trained on a shared coefficient representation and assembled through structure-aware rollout. Rather than learning full trajectories end-to-end, LegONet decomposes a target PDE into mechanism-level components and represents each component by a plug-and-play block with an explicit structural role. Across dimensions, boundary conditions, and PDE families, this formulation supports accurate long-horizon rollout while preserving physically interpretable mechanisms.

A central implication of this formulation is that stability and generalization can be analyzed at the level of blocks rather than only at the level of full trajectories. The finite-horizon analysis separates rollout error into two parts: block mismatch and composition discretization error. The empirical diagnostics support this picture across dissipative, Hamiltonian and mixed systems. This blockwise view offers a degree of transparency that is difficult to obtain in monolithic operator learners, where stability typically emerges only indirectly from training data, regularization and architecture choice.

At the same time, the framework inherits several structural constraints. First, compositional reuse requires a compatible baseplate representation. Changes in geometry, boundary type or trial space therefore require the corresponding baseplate construction. Second, the mechanism library is necessarily finite. New nonlinearities, constraints, or nonlocal interactions require additional block pretraining. Third, although symmetric splitting improves stability, composition error can still accumulate under strong non-commutativity or stiff multi-physics coupling, potentially requiring smaller time steps or higher-order schemes. Finally, block pretraining depends on access to trusted reference operators, so inaccuracies in the underlying discretization can propagate into the learned primitives.

\begin{figure}[H]
\centering

\begin{subfigure}[t]{\textwidth}
  \centering
  \includegraphics[width=\linewidth]{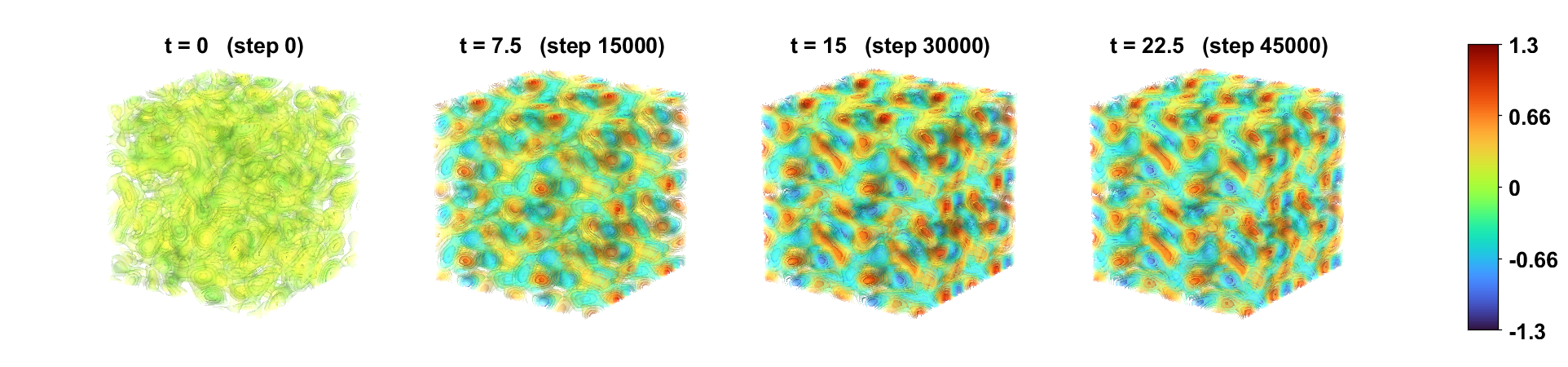}
  \caption{Reference isosurfaces (fixed levels; shared color scale).}
  \label{fig:sh3d_ref_iso}
\end{subfigure}

\vspace{3pt}

\begin{subfigure}[t]{\textwidth}
  \centering
  \includegraphics[width=\linewidth]{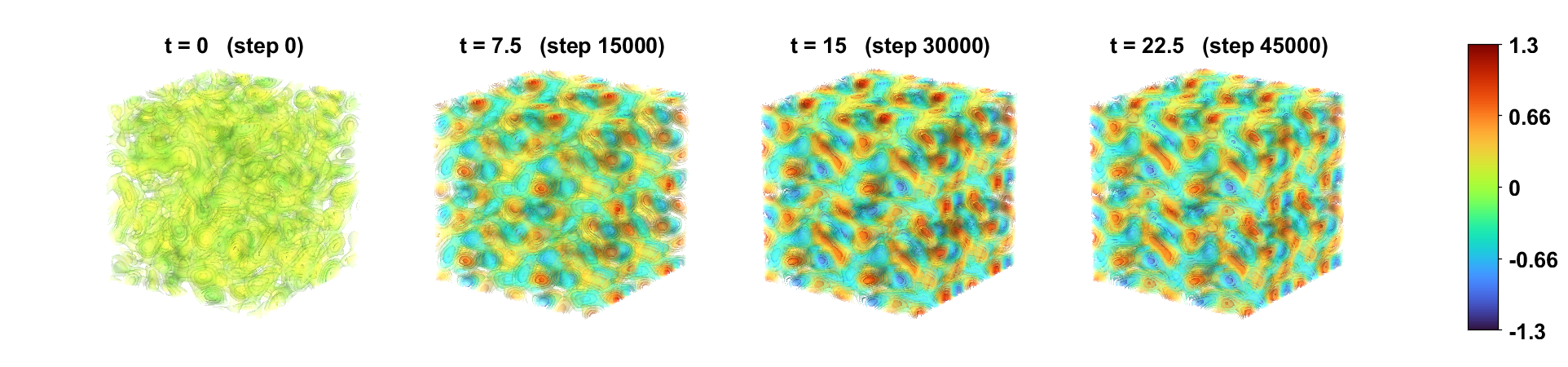}
  \caption{LegONet isosurfaces (same levels and color scale as (a)).}
  \label{fig:sh3d_learn_iso}
\end{subfigure}

\vspace{3pt}

\begin{subfigure}[t]{0.49\textwidth}
  \centering
  \includegraphics[width=\linewidth]{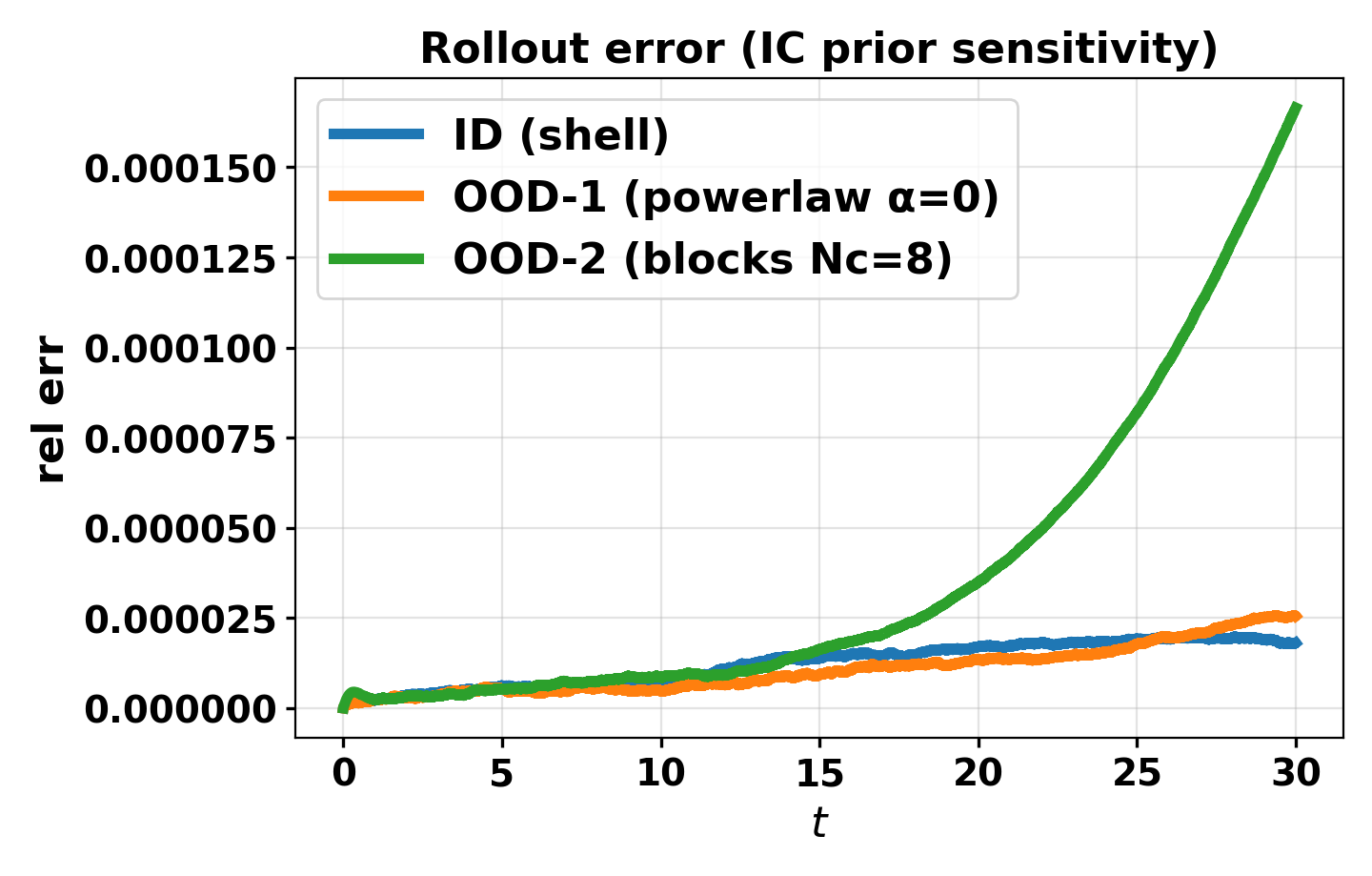}
  \caption{Closed-loop rollout error over time.}
  \label{fig:sh3d_relerr}
\end{subfigure}\hfill
\begin{subfigure}[t]{0.49\textwidth}
  \centering
  \includegraphics[width=\linewidth]{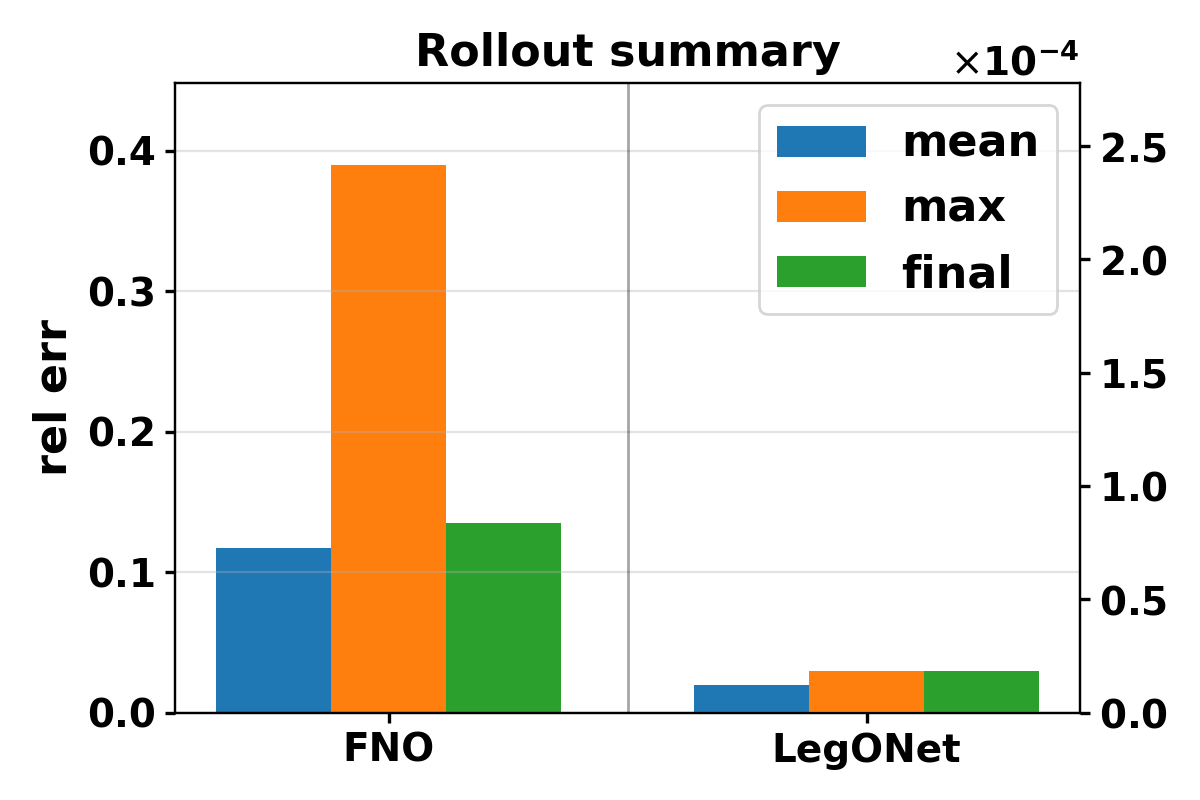}
  \caption{Rollout summary: LegONet vs.\ FNO.}
  \label{fig:sh3d_fno_vs_lego_bar}
\end{subfigure}

\caption{\textbf{3D Swift--Hohenberg: stiff operator reuse and OOD initial conditions.}
Figs. \ref{fig:sh3d_ref_iso}--\ref{fig:sh3d_learn_iso}: reference and LegONet trajectories visualized by multi-level isosurfaces using identical global levels and a shared color scale.
Fig. \ref{fig:sh3d_relerr}: closed-loop rollout error under in-distribution and shifted initial-condition priors (Supplementary Information).
Fig. \ref{fig:sh3d_fno_vs_lego_bar}: rollout summary (mean / max / final-time) comparing LegONet with the supervised FNO baseline on $[0,30]$.}
\label{fig:sh3d}
\end{figure}

These limitations point to several directions for future work. Automated block discovery and selection could reduce reliance on manual mechanism specification. Extending baseplates to more general geometries and adaptive discretizations would broaden applicability. Data-adaptive or state-dependent structure operators may further improve robustness in strongly nonlinear regimes. Expanding the block library toward constrained, stochastic and multiphysics operators would further increase the range of systems that can be assembled within the same framework. More broadly, because LegONet exposes plug-and-play blocks on a common representation, it suggests a path toward community-built libraries of interoperable scientific operators. With open implementations and sustained community contribution, such libraries could eventually mature into practical, immediately usable packages for scientific computing.

Overall, LegONet points to a different way of building learned PDE solvers. Instead of training a monolithic surrogate for each new system, one can train and reuse structure-preserving operator blocks that remain composable, diagnosable and extensible. In that sense, the contribution of this work is not only a new model, but also a step toward plug-and-play mechanism-level infrastructure for scientific machine learning.

\section*{Methods}\label{sec:Methodology}

LegONet constructs neural PDE solvers by composing pretrained operator blocks on a shared coefficient-space representation.
The mechanism decomposition of the PDE, the coefficient representation, and the generator-induced block template for evolution blocks are introduced in
\eqref{eq:intro_pde_sum}--\eqref{eq:intro_block_form}.
Here we specify how these ingredients are instantiated and combined during training and rollout.

A baseplate is the tuple $(u_{\mathrm{lift}},\Phi_b,\mathcal{P}_b)$, which defines a boundary-adapted representation of the solution together with the maps between fields and coefficients.
Given coefficients $\mathbf a\in\mathbb{R}^K$, the operator $\Phi_b$ reconstructs the homogeneous field component in the trial space, while $\mathcal{P}_b$ maps fields back to coefficients.
All learned blocks act only on $\mathbf a$.
Changing the physical-space resolution affects only the evaluation of $\Phi_b$ and $\mathcal{P}_b$ on the grid, while leaving the underlying block representation unchanged.

Throughout, non-bold symbols such as $u(\cdot,t)$ denote continuous fields, whereas bold symbols denote coefficient or nodal vectors.
In particular, $\mathbf a(t)\in\mathbb{R}^K$ denotes the coefficient state, and $\mathbf u(t):=\bigl(u(x_q,t)\bigr)_{q=1}^{Q}$ denotes the corresponding nodal evaluation vector on the common nodes $\{x_q\}_{q=1}^{Q}$.
At discrete times $t_n=n\Delta t$, we write $\mathbf a_n:=\mathbf a(t_n)$ and $\mathbf u_n:=\mathbf u(t_n)$.
When convenient, $u_n$ denotes the continuous field $u(\cdot,t_n)$, while $\mathbf u_n$ always denotes its nodal evaluation vector.
Accordingly, rollout dynamics are analyzed on the coefficient state $\mathbf a_n$, whereas accuracy and structure diagnostics are reported on the reconstructed nodal field $\mathbf u_n$.

\subsection*{Boundary-adapted baseplate}\label{sec:baseplate_eval}

The first role of the framework is to separate boundary handling from mechanism learning. To do this, all computation is organized around the coefficient state $\mathbf{a}\in\mathbb{R}^K$ in \eqref{eq:intro_coeff_rep}, while the baseplate $b$ specifies the boundary-adapted trial space and the transforms associated with it. In this way, boundary compliance is handled by construction, and learned blocks can focus only on plug-and-play mechanism modeling.

In particular, given $\mathbf{a}$, we reconstruct the homogeneous component on evaluation nodes and then recover the full field by adding a lifting term,
\begin{equation}\label{eq:method_reconstruct}
\mathbf{u}_0 \;=\; \Phi_b \mathbf{a}, \qquad
\mathbf{u} \;=\; \mathbf{u}_{\mathrm{lift}} + \mathbf{u}_0 ,
\end{equation}
where $\mathbf{u}_{\mathrm{lift}}$ is the vector of nodal evaluations of the lifting function $u_{\mathrm{lift}}$ at the quadrature/grid points $\{x_q\}_{q=1}^Q$, which encodes the prescribed non-homogeneous boundary data (when present), and $\mathbf{u}_0$ satisfies the corresponding homogeneous boundary conditions by construction of $\{\phi_k^{(b)}\}_{k=1}^K$.
%Let $\{x_q\}_{q=1}^{Q}$ denote the evaluation nodes (grid points or quadrature nodes).
In practice, $\Phi_b\in\mathbb{R}^{Q\times K}$ is the basis evaluation matrix, $(\Phi_b)_{qk}=\phi_k^{(b)}(x_q)$, such as Fourier for periodic boundaries, cosine for Neumann boundaries, and boundary-adapted polynomials for Dirichlet boundaries (Fig. \ref{fig:block}).

Conversely, grid values are mapped back to coefficients through the baseplate projection
\begin{equation}\label{eq:method_project}
\mathbf{a} \;=\; \mathcal{P}_b(\mathbf{u}),
\end{equation}
implemented by the standard modal transform or discrete $L^2$ projection associated with the chosen basis (FFT/DCT in Fourier/cosine bases
and quadrature-based $L^2$ projection in polynomial bases).

Nonlinear mechanisms are evaluated by a reconstruct--evaluate--project step. We first reconstruct $\mathbf{u}_0=\Phi_b\mathbf{a}$, then apply the nonlinearity pointwise on the nodes, and finally project the result back with $\mathcal{P}_b$. In Fourier- and cosine-based settings, nonlinear products are computed pseudospectrally and de-aliased by the $2/3$ rule. In polynomial bases, nonlinearities are evaluated at quadrature nodes and projected by the discrete $L^2$ inner product.

The key consequence is that every learned block shares the same input–output form:
\[
F^{\boldsymbol{\theta}}:\mathbb{R}^K\to\mathbb{R}^K.
\]
The baseplate enforces boundary compatibility through the basis and lifting, while the blocks learn plug-and-play dynamics on a fixed coefficient representation.

\begin{remark}\label{rem:baseplate_transfer}
The boundary-adapted baseplate makes boundary conditions explicit and decouples them from mechanism learning, but it makes blocks baseplate-specific through $\Phi_b$ and $\mathcal{P}_b$. A natural next step is to enable transfer across baseplates by learning or deriving mappings between coefficient representations.
\end{remark}

% ============================================================
\subsection*{Operator block decomposition}\label{sec:template_structure}

The second role of the framework is to separate mechanism modeling from full-solver training. Starting from the mechanism decomposition in \eqref{eq:intro_pde_sum}, the lifted representation $\mathbf u=\mathbf u_{\mathrm{lift}}+\Phi_b\mathbf{a}$ induces coefficient-space reference blocks
\begin{equation}\label{eq:ref_dyn}
\mathbf{a}_t \;=\; \sum_{i=1}^{N_{\mathrm{blk}}} F_i^{\mathrm{ref}}(\mathbf{a}),
\qquad
F_i^{\mathrm{ref}}(\mathbf{a})
\;:=\;
\mathcal{P}_b\!\Big(L_i\big(\mathbf u_{\mathrm{lift}}+\Phi_b\mathbf{a}\big)\Big)\in\mathbb{R}^K,
\end{equation}
so each mechanism contributes an additive vector field on the shared state $\mathbf{a}$.
LegONet replaces these reference updates $F_i^{\mathrm{ref}}$ with learned blocks $F_i^{\boldsymbol{\theta}}:\mathbb{R}^K\to\mathbb{R}^K$ and assembles the reduced dynamics from their sum:
\[
\mathbf{a}_t \;=\; \sum_{i=1}^{N_{\mathrm{blk}}} F_i^{\boldsymbol{\theta}}(\mathbf{a}).
\]

Although all blocks act on the same coefficient state, they play different roles. Evolution blocks contribute directly to the reduced dynamics $\mathbf{a}_t$.
These are instantiated either as dissipative E-blocks of the form $-G\nabla E$ (E-blocks), conservative H-blocks of the form $J\nabla H$ (H-blocks), or residual R-blocks when a scalar-generator form is not natural. Auxiliary operator maps also act in coefficient space, but they are used inside other components rather than appearing as separate evolution terms. Examples include spatial differentiation, projections, and algebraic solves such as Poisson inversion.

Rather than regressing a free vector field $\mathbf{a}\mapsto F_i^{\mathrm{ref}}(\mathbf{a})$,
we parameterize evolution blocks by scalar generators and fixed, baseplate-consistent structure operators.
Concretely, we consider the general coefficient-space form
\[
F_i^{\boldsymbol{\theta}}(\mathbf{a})
=
-\,G_{i}\,\nabla_{\mathbf{a}} E_{i}^{a,\boldsymbol{\theta}}(\mathbf{a})
\;+\;
J_{i}\,\nabla_{\mathbf{a}} H_{i}^{a,\boldsymbol{\theta}}(\mathbf{a})
\;+\;
R_{i}^a(\mathbf{a}).
\]
Here $G_i\in\mathbb{R}^{K\times K}$ is symmetric positive semidefinite and $J_i\in\mathbb{R}^{K\times K}$ is skew-symmetric, so the block inherits dissipative or conservative structure at the discrete level.
The residual term $R_i^a$ captures effects that are not naturally expressed by the chosen scalar-generator form, such as forcing or closure terms.
The corresponding block-level dissipation, conservation, and Strang-composition consequences are summarized in the Supplementary Information.

Note that the expression above is a general template. In practice, to maximize reuse and enable plug-and-play composition,
we typically instantiate each block as a single-purpose component that activates only one term:
an E-block ($-G\nabla E$), an H-block ($J\nabla H$), or an R-block ($R$).
This makes the learned blocks easier to interpret and easier to reuse across PDE assemblies.

When $R_i^a$ is available in closed form, we evaluate it directly on the baseplate.
Otherwise, we parameterize $R_{i}^{a,\boldsymbol{\theta}}:\mathbb{R}^K\to\mathbb{R}^K$
and train it by the same operator-matching objective used for structured blocks.

\begin{remark}
    Although LegONet is designed for assembling time-evolution mechanisms, the same block principle also applies to non-evolution operator maps that act purely in space.
    In our experiments, for example, the Poisson inversion block approximates $(-\Delta)^{-1}$ on the retained modes.
    The operator-matching view is therefore not limited to temporal generators.
\end{remark}

% ============================================================
\subsection*{Trajectory-free block pretraining}\label{sec:block_pretraining}

Once the block roles are fixed, training is carried out independently at the mechanism level rather than on full trajectories. Each learned evolution block is pretrained offline by instantaneous operator matching on the chosen baseplate. This makes training modular and allows the same block to be reused across multiple PDE assemblies supported by the same coefficient representation.

We sample admissible coefficient states $\mathbf{a}\sim\mu_b$ and compute the corresponding reference targets
$F_i^{\mathrm{ref}}(\mathbf{a})$ from a trusted discretization consistent with $(u_{\mathrm{lift}},\Phi_b,\mathcal{P}_b)$.
The block parameters are learned by
\begin{equation}\label{eq:method_loss}
\min_{\boldsymbol{\theta}}\;
\mathbb{E}_{\mathbf{a}\sim\mu_b}\!\left[
\big\|F_i^{\boldsymbol{\theta}}(\mathbf{a})-F_i^{\mathrm{ref}}(\mathbf{a})\big\|_{2}^2
\right].
\end{equation}
The sampling distribution $\mu_b$ specifies a training protocol on the boundary-compatible coefficient space.
In our experiments, the default choice is a spectral-decay Gaussian prior, which covers smooth states while controlling modal amplitudes. Auxiliary operator maps on the same retained modes are either evaluated in closed form or trained by the same objective \eqref{eq:method_loss}.

In practice, we approximate the expectation in \eqref{eq:method_loss} with an empirical risk over $M$ i.i.d.\ coefficient samples
$\{\mathbf{a}^{(m)}\}_{m=1}^{M}$ and optimize it with mini-batch stochastic gradient methods:
\begin{equation}\label{eq:method_loss_erm}
\min_{\boldsymbol{\theta}}\;
\frac{1}{M}\sum_{m=1}^{M}
\big\|F_i^{\boldsymbol{\theta}}\!\big(\mathbf{a}^{(m)}\big)-F_i^{\mathrm{ref}}\!\big(\mathbf{a}^{(m)}\big)\big\|_{2}^2.
%\;\approx\;
%\min_{\boldsymbol{\theta}}\;
%\frac{1}{B}\sum_{b=1}^{B}
%\big\|F_i^{\boldsymbol{\theta}}\!\big(\mathbf{a}_{b}\big)-F_i^{\mathrm{ref}}\!\big(\mathbf{a}_{b}\big)\big\|_{2}^2,
\end{equation}
%where the right-hand objective is the mini-batch loss with batch size $B$ used during training.

This training cost is paid upfront, but it is naturally amortized when the same mechanisms are reused across many equations, parameter settings or solver reconfigurations. In the present study, each block is trained with $20{,}000$ independent coefficient samples and then reused without additional trajectory-level training. A practical consequence of Eq.~\eqref{eq:method_loss} is that block pretraining assumes access to reliable instantaneous operator labels. This is well matched to settings where accurate spectral or high-order discretizations are available, but it may be restrictive in regimes where trusted operator evaluations are difficult to obtain.

\subsection*{Plug-and-play inference}\label{sec:method_inference}

At deployment, LegONet does not retrain a solver end-to-end. Instead, it selects pretrained blocks and composes them on the shared coefficient state associated with the chosen baseplate. In this sense, changing the PDE becomes a selection-and-assembly problem rather than a retraining problem.

Given a target PDE decomposition in \eqref{eq:intro_pde_sum}, inference evolves the reduced state $\mathbf{a}(t)\in\mathbb{R}^K$ from \eqref{eq:intro_coeff_rep}
over a uniform time grid $t_n=n\Delta t$, $n=0,\dots,N_{\mathrm{steps}}$, with $T=N_{\mathrm{steps}}\Delta t$. Each selected mechanism $i\in\{1,\dots,N_{\mathrm{blk}}\}$ contributes a learned coefficient-space vector field
$F_i^{\boldsymbol{\theta}}:\mathbb{R}^K\to\mathbb{R}^K$ (cf.\ \eqref{eq:intro_block_form}), defining the isolated sub-dynamics
\begin{equation}\label{eq:method_subdyn}
%\frac{d\mathbf{a}}{dt} 
\mathbf a_t\;=\; F_i^{\boldsymbol{\theta}}(\mathbf{a}).
\end{equation}
Let $S_{i,\tau}^{\boldsymbol{\theta}}:\mathbb{R}^K\to\mathbb{R}^K$ denote the block update that advances \eqref{eq:method_subdyn} over a substep $\tau\in\{\Delta t/2,\Delta t\}$.
The implementation of $S_{i,\tau}^{\boldsymbol{\theta}}$ is chosen to respect the intended structure of the block, such as dissipative or conservative behavior. The assembled dynamics are then advanced by the symmetric Strang macro-step
\begin{equation}\label{eq:method_rollout}
\mathbf{a}_{n+1}^{\boldsymbol{\theta}}
\;:=\;
S_{\Delta t}^{\boldsymbol{\theta}}\!\big(\mathbf{a}_n^{\boldsymbol{\theta}}\big),
\qquad
S_{\Delta t}^{\boldsymbol{\theta}}
:=
S_{1,\Delta t/2}^{\boldsymbol{\theta}}\circ
S_{2,\Delta t/2}^{\boldsymbol{\theta}}\circ\cdots\circ
S_{N_{\mathrm{blk}},\Delta t}^{\boldsymbol{\theta}}\circ\cdots\circ
S_{2,\Delta t/2}^{\boldsymbol{\theta}}\circ
S_{1,\Delta t/2}^{\boldsymbol{\theta}}.
\end{equation}
The physical field is reconstructed through
\begin{equation}\label{eq:method_u_recon}
\mathbf{u}_n^{\boldsymbol{\theta}}
\;=\;
\mathbf{u}_{\mathrm{lift}}(\cdot,t_n) \;+\; \Phi_b\,\mathbf{a}_n^{\boldsymbol{\theta}}.
\end{equation}
Because every update acts on the coefficient state inside the fixed trial space, boundary conditions are satisfied by construction throughout rollout.

Algorithm~\ref{alg:inference_strang} and Fig.~\ref{fig:workflow} summarize this procedure. When a reference discretization is used, we apply the same composition schedule with $F_i^{\mathrm{ref}}$ and the associated maps
$S_{i,\tau}^{\mathrm{ref}}$. The only difference is the underlying block vector fields. In practice, primitive blocks are often trained for canonical operators (e.g. $u_{xx}$, $uu_x$, $\Delta u$), while the deployment may require scaled, shifted, or composed variants such as
$\nu \Delta u$, $\alpha uu_x$, or higher-order operators
$-(\Delta + k_0^2)^2 u$.
These are implemented through linear combinations
or repeated compositions of primitive blocks. Since scaling preserves Lipschitz continuity and smoothness,
and composition is handled explicitly by splitting,
the error decomposition in Theorem~\ref{thm:main_structure_error} extends directly to
these deployed forms.

\begin{algorithm}[H]
\caption{Rollout by composing pretrained blocks}
\label{alg:inference_strang}
\begin{algorithmic}[1]
\State \textbf{Inputs:} baseplate $(u_{\mathrm{lift}}(\cdot,t),\,\Phi_b,\,\mathcal{P}_b)$; step size $\Delta t$; blocks
$\{F_i^{\boldsymbol{\theta}}\}_{i=1}^{N_{\mathrm{blk}}}$ with within-block maps $\{S_{i,\tau}^{\boldsymbol{\theta}}\}_{i=1}^{N_{\mathrm{blk}}}$.
\State \textbf{Initialize:} $\mathbf{a}_0=\mathcal{P}_b\!\big(\mathbf{u}(\cdot,0)-\mathbf{u}_{\mathrm{lift}}(\cdot,0)\big)$.
\For{$n=0,1,\ldots,N_{\mathrm{steps}}-1$}
    \State \textbf{Strang macro-step:}
    \[
    \mathbf{a}_{n+1}^{\boldsymbol{\theta}}
    =
    S_{\Delta t}^{\boldsymbol{\theta}}\!\left(\mathbf{a}_n^{\boldsymbol{\theta}}\right),
    \qquad
    S_{\Delta t}^{\boldsymbol{\theta}}
    =
    S_{1,\Delta t/2}^{\boldsymbol{\theta}}\circ\cdots\circ S_{N_{\mathrm{blk}},\Delta t}^{\boldsymbol{\theta}}\circ\cdots\circ S_{1,\Delta t/2}^{\boldsymbol{\theta}}.
    \]
    \State \textbf{Reconstruct (optional):}
    $\mathbf{u}^{\boldsymbol{\theta}}(\cdot,t_{n+1})=\mathbf{u}_{\mathrm{lift}}(\cdot,t_{n+1})+\Phi_b\,\mathbf{a}_{n+1}^{\boldsymbol{\theta}}$.
    %with $t_{n+1}=t_n+\Delta t$.
\EndFor
\State \textbf{Output:} $\{\mathbf{a}_n^{\boldsymbol{\theta}}\}_{n=0}^{N_{\mathrm{steps}}}$ and $\{u^{\boldsymbol{\theta}}(\cdot,t_n)\}_{n=0}^{N_{\mathrm{steps}}}$.
\end{algorithmic}
\end{algorithm}

%===================================================
\subsection*{Implementation details}\label{sec:Implementation}

The remaining ingredients specify how the block templates are parameterized on a fixed baseplate $(u_{\mathrm{lift}},\,\Phi_b,\,\mathcal{P}_b)$. Across all settings, the shared coefficient representation is fixed, and learning enters only through the hypothesis class chosen for each block.

Each E-block or H-block follows the generator-induced template in Eq.~\eqref{eq:intro_block_form}, where the learnable objects are scalar
generators $E_{i}^{a,\boldsymbol{\theta}}$ and/or $H_{i}^{a,\boldsymbol{\theta}}$.
Their gradients with respect to $\mathbf{a}$ are obtained by automatic differentiation.
By default, each scalar generator is implemented as a compact map $\mathbb{R}^K\to\mathbb{R}$, typically a multilayer perceptron(MLP), unless additional structure is known.

When the mechanism and the chosen baseplate suggest a lower-complexity form, we restrict the hypothesis class to improve identifiability and data efficiency. For mechanisms close to mode-decoupled responses in the chosen basis, such as linear diffusion or dispersion on aligned spectral bases, we use
a structured quadratic generator
\[
E_{i}^{a,\boldsymbol{\theta}}(\mathbf{a})
=
\tfrac12\,\mathbf{a}^\top K_{i}^{\boldsymbol{\theta}}\,\mathbf{a},
\qquad
K_{i}^{\boldsymbol{\theta}}=\mathrm{diag}(\mathbf{k}_{i}^{\boldsymbol{\theta}})+U_{i}^{\boldsymbol{\theta}}(U_{i}^{\boldsymbol{\theta}})^\top,
\]
which enforces symmetry by construction and combines a dominant diagonal response with a low-rank coupling correction.

When the generator admits an integral form, we learn a pointwise density $\rho_{i}^{\boldsymbol{\theta}}$ and define
\[
H_{i}^{a,\boldsymbol{\theta}}(\mathbf{a})
=
\int_{\Omega}\rho_{i}^{\boldsymbol{\theta}}\!\big(u_0(x)\big)\,dx,
\qquad
u_0(x)=\sum_{k=1}^{K} a_k\,\phi_k^{(b)}(x).
\]
In practice, the integral is evaluated on nodes consistent with the chosen baseplate and projection operator.

The same principle also supports other structure-aware simplifications, depending on the block being learned and the information available on the baseplate.
Examples include restricting generators to diagonal or banded forms in bases where the target operator is nearly diagonal, enforcing
symmetries implied by the domain or boundary conditions, sharing parameters across mode groups, or using low-rank or separable
parameterizations when the mechanism exhibits approximate separability.
These choices are optional, but they help align the learned block with known operator structure while keeping the coefficient representation unchanged.

\begin{remark}\label{rem:learn_GJ}
Fixing $G_i$ and $J_i$ isolates learning to the generators and ensures that each evolution block inherits the intended discrete structure
while remaining reusable across resolutions and PDE configurations supported by the baseplate.
Learning $G_i$ or $J_i$ is possible, but would require constrained parameterizations and additional stability control.
\end{remark}

%====================================================
\subsection*{Structural properties and error bounds}
\label{sec:structural-properties}

The theory in this section clarifies two points that are central to the LegONet design. First, the blockwise structure declared at the level of individual updates is preserved inside the Strang composition. Second, finite-horizon rollout error separates cleanly into a modeling term from block mismatch and a numerical term from time discretization. All standing assumptions and proofs are given in Supplementary Information.

We use $\|\cdot\|_2$ for the Euclidean norm on $\mathbb{R}^K$.
For a physical field $\mathbf v$, we use the baseplate-induced norm  $\|\mathbf v\|_{w,2} \;=\; \left(\sum_{q=1}^{Q} w_q\,|v(x_q)|^2\right)^{1/2}$, where $\{x_q,w_q\}_{q=1}^Q$ are the grid or quadrature nodes and weights used by the baseplate projection.
The field--coefficient relation follows the baseplate reconstruction \eqref{eq:method_reconstruct}:
\[
\mathbf u(t) \;=\; \mathbf u_{\mathrm{lift}}(\cdot,t)+\Phi_b\,\mathbf{a}(t),
\qquad
\mathbf u_n^{\boldsymbol{\theta}} \;=\; \mathbf u_{\mathrm{lift}}(\cdot,t_n)+\Phi_b\,\mathbf{a}_n^{\boldsymbol{\theta}},
\]
where $t_n=n\Delta t$.
The learned rollout $\{\mathbf{a}_n^{\boldsymbol{\theta}}\}_{n=0}^{N_{\mathrm{steps}}}$ is generated by the Strang macro-step \eqref{eq:method_rollout} with within-block maps $S_{i,\tau}^{\boldsymbol{\theta}}$.

\begin{theorem}[Structure-preserving rollout and total error bound]
\label{thm:main_structure_error}
Fix $T>0$ and an integer $N_{\mathrm{steps}}\ge 1$, and set $\Delta t:=T/N_{\mathrm{steps}}$ and $t_n:=n\Delta t$.
Let $\mathcal{K}\subset\mathbb{R}^K$ be compact and assume the standing conditions in Supplementary Information.
Let $\mathbf{a}(t)$ solve the reference reduced dynamics \eqref{eq:ref_dyn} with $\mathbf{a}(0)=\mathbf{a}_0$, and let
$\{\mathbf{a}_n^{\boldsymbol{\theta}}\}_{n=0}^{N_{\mathrm{steps}}}$ be generated by the learned Strang macro-step \eqref{eq:method_rollout}.

\smallskip
\noindent
(i) \emph{Structure preservation.}
If a dissipative block $i$ uses a within-block update satisfying
\[
E_{i}^{a,\boldsymbol{\theta}}\!\big(S_{i,\tau}^{\boldsymbol{\theta}}(\mathbf{a})\big)\le E_{i}^{a,\boldsymbol{\theta}}(\mathbf{a}),
\qquad \forall\,\mathbf{a}\in\mathcal{K},\ \tau\in\{\Delta t/2,\Delta t\},
\]
and a conservative block $i$ uses a within-block update satisfying
\[
H_{i}^{a,\boldsymbol{\theta}}\!\big(S_{i,\tau}^{\boldsymbol{\theta}}(\mathbf{a})\big)= H_{i}^{a,\boldsymbol{\theta}}(\mathbf{a}),
\qquad \forall\,\mathbf{a}\in\mathcal{K},\ \tau\in\{\Delta t/2,\Delta t\},
\]
then the same inequalities and equalities hold at the corresponding substeps inside the Strang schedule \eqref{eq:method_rollout} for every macro-step.

\smallskip
\noindent
(ii) \emph{Total error bound.}
Define the uniform block mismatch on $\mathcal{K}$ by
\[
\varepsilon_i
\;:=\;
\sup_{\mathbf{a}\in\mathcal{K}}
\big\|\mathbf{F}_i^{\boldsymbol{\theta}}(\mathbf{a})-\mathbf{F}_i^{\mathrm{ref}}(\mathbf{a})\big\|_2.
\]
Then there exists $C_T>0$, independent of $\Delta t$, such that
\begin{equation}\label{eq:main_error_bound}
\big\|\mathbf u_{N_{\mathrm{steps}}}^{\boldsymbol{\theta}}-\mathbf u(T)\big\|_{w,2}
\;\le\;
C_T\Big(
T\sum_{i=1}^{N_{\mathrm{blk}}}\varepsilon_i
+
\Delta t^2
\Big).
\end{equation}
\end{theorem}

The error bound in Eq.~\eqref{eq:main_error_bound} separates modeling and numerics in a form that is directly aligned with the LegONet design. For a fixed horizon $T$, the term $T\sum_i \varepsilon_i$ accumulates mismatch between the learned and reference blocks, whereas the $\Delta t^2$ term is the global splitting error of the Strang macro-step. The quantity $\varepsilon_i$ is a worst-case block mismatch over a compact set $\mathcal K\subset\mathbb R^K$ containing all coefficient states visited up to time $T$. In practice, we do not compute this supremum directly. Instead, we estimate it empirically on a held-out collection of admissible coefficient states $\{\mathbf a^{(m)}\}_{m=1}^M\subset\mathcal K$ through both a maximum and a mean statistic,
\begin{comment}
The bound~\eqref{eq:main_error_bound} separates modeling and numerics.
For a fixed horizon $T$, the term $T\sum_i \varepsilon_i$ accumulates the block approximation errors, whereas the $\Delta t^2$ term is the global splitting error of the Strang macro-step.
The quantity $\varepsilon_i$ is a worst-case mismatch between the learned and reference coefficient-space vector fields, measured uniformly over a compact set $\mathcal K\subset\mathbb R^K$ that contains all coefficient states visited by the reference and learned substeps up to time $T$.
We do not compute the supremum directly.
Instead, we estimate an empirical proxy on a held-out collection of admissible coefficient states $\{a^{(m)}\}_{m=1}^M\subset\mathcal K$ and report both a max and a mean statistic,
\end{comment}
\[
\hat{\varepsilon}_i^{\max}:=\max_{m}\big\|F_i^{\boldsymbol{\theta}}(\mathbf a^{(m)})-F_i^{\mathrm{ref}}(\mathbf a^{(m)})\big\|_2,
\qquad
\hat{\varepsilon}_i^{\mathrm{mean}}:=\frac1M\sum_{m=1}^M\big\|F_i^{\boldsymbol{\theta}}(\mathbf a^{(m)})-F_i^{\mathrm{ref}}(\mathbf a^{(m)})\big\|_2.
\]
Improved block pretraining reduces these empirical mismatches and is typically accompanied by smaller rollout errors.

%===============================================

% ==============================
% Section: Acknowledgments
% ==============================
\section*{Acknowledgments}
% This section was empty in the original file. A placeholder is added.
We would like to thank the support of National Science Foundation (DMS-2533878, DMS-2053746, DMS-2134209, ECCS-2328241, CBET-2347401 and OAC-2311848), and U.S.~Department of Energy (DOE) Office of Science Advanced Scientific Computing Research program DE-SC0023161, the SciDAC LEADS Institute, and DOE–Fusion Energy Science, under grant number: DE-SC0024583.

\section*{Code availability}
The code used in this study, including Python scripts for data analysis, is available at \href{https://github.com/Yooki-YueqiWang/LegONet}{https://github.com/Yooki-YueqiWang/LegONet}.

\bibliography{sn-bibliography}% common bib file
%% if required, the content of .bbl file can be included here once bbl is generated
%%\input sn-article.bbl

\clearpage
\section*{Extended Data}

% Reset counters within Extended Data
\setcounter{equation}{0}
\setcounter{figure}{0}
\setcounter{table}{0}

% Use plain arabic numbering inside ED
\renewcommand{\theequation}{\arabic{equation}}
\renewcommand{\thefigure}{\arabic{figure}}
\renewcommand{\thetable}{\arabic{table}}

% Caption name prefix
\renewcommand{\figurename}{Extended Data Fig.}
\renewcommand{\tablename}{Extended Data Table}

% Caption style (Nature-like)
\captionsetup{
  labelfont={bf,color=black},
  textfont=normalfont,
  labelsep=space
}

% =========================
% Figure 1: 1D Dirichlet experiments
% =========================
\begin{figure}[H]
  \centering

  % Row 1: operator comparisons
  \begin{subfigure}[t]{0.49\linewidth}
    \centering
    \includegraphics[width=\linewidth]{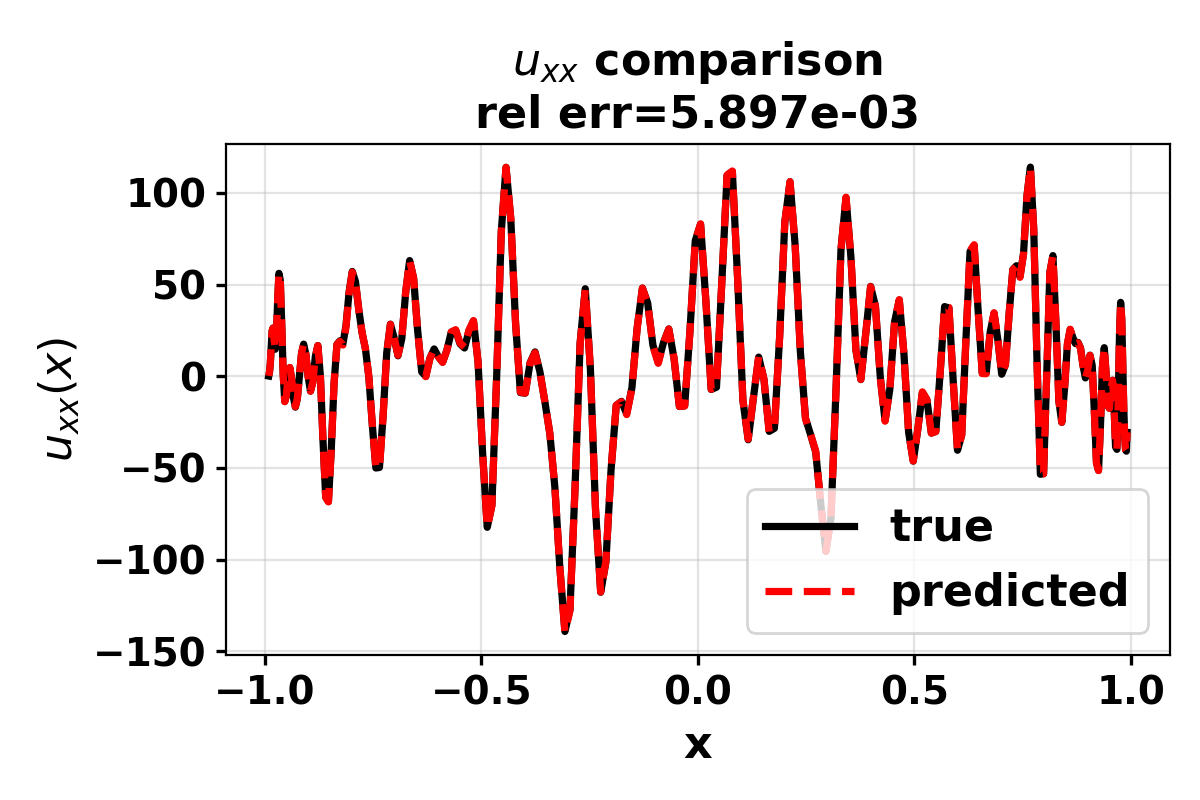}
    \caption{$u_{xx}$ block.}
    \label{fig:1d_uxx_op}
  \end{subfigure}\hfill
  \begin{subfigure}[t]{0.49\linewidth}
    \centering
    \includegraphics[width=\linewidth]{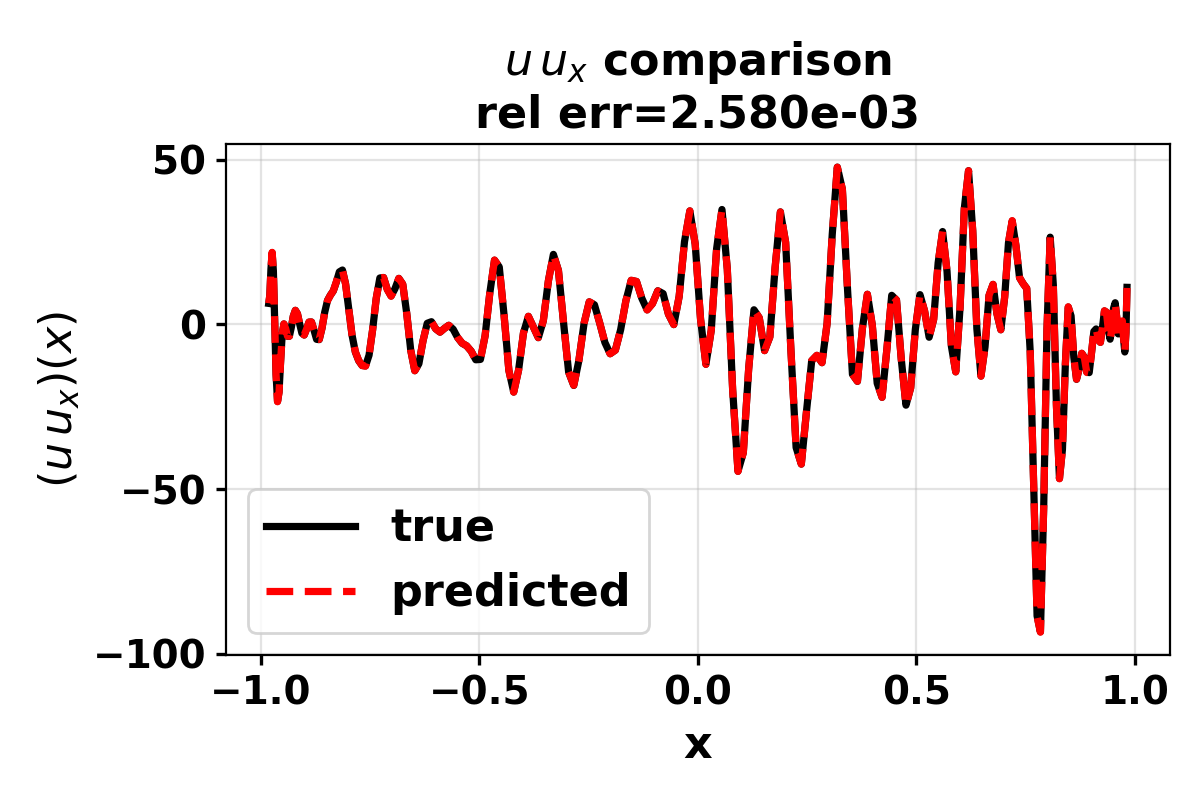}
    \caption{$u\,u_x$ block.}
    \label{fig:1d_uux_op}
  \end{subfigure}

  \vspace{0.8ex}

  % Row 2: energy diagnostics
  \begin{subfigure}[t]{0.49\linewidth}
    \centering
    \includegraphics[width=\linewidth]{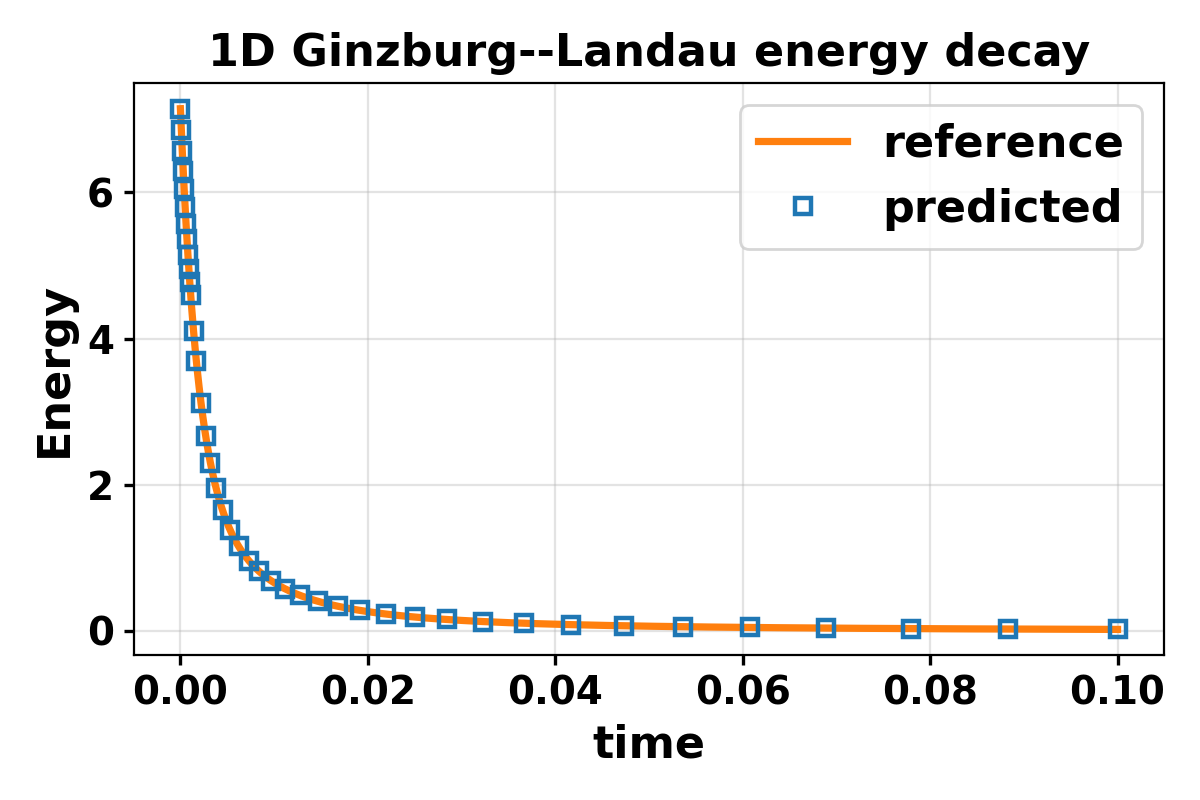}
    \caption{1D Ginzburg--Landau: energy decay.}
    \label{fig:1d_gl_energy}
  \end{subfigure}\hfill
  \begin{subfigure}[t]{0.49\linewidth}
    \centering
    \includegraphics[width=\linewidth]{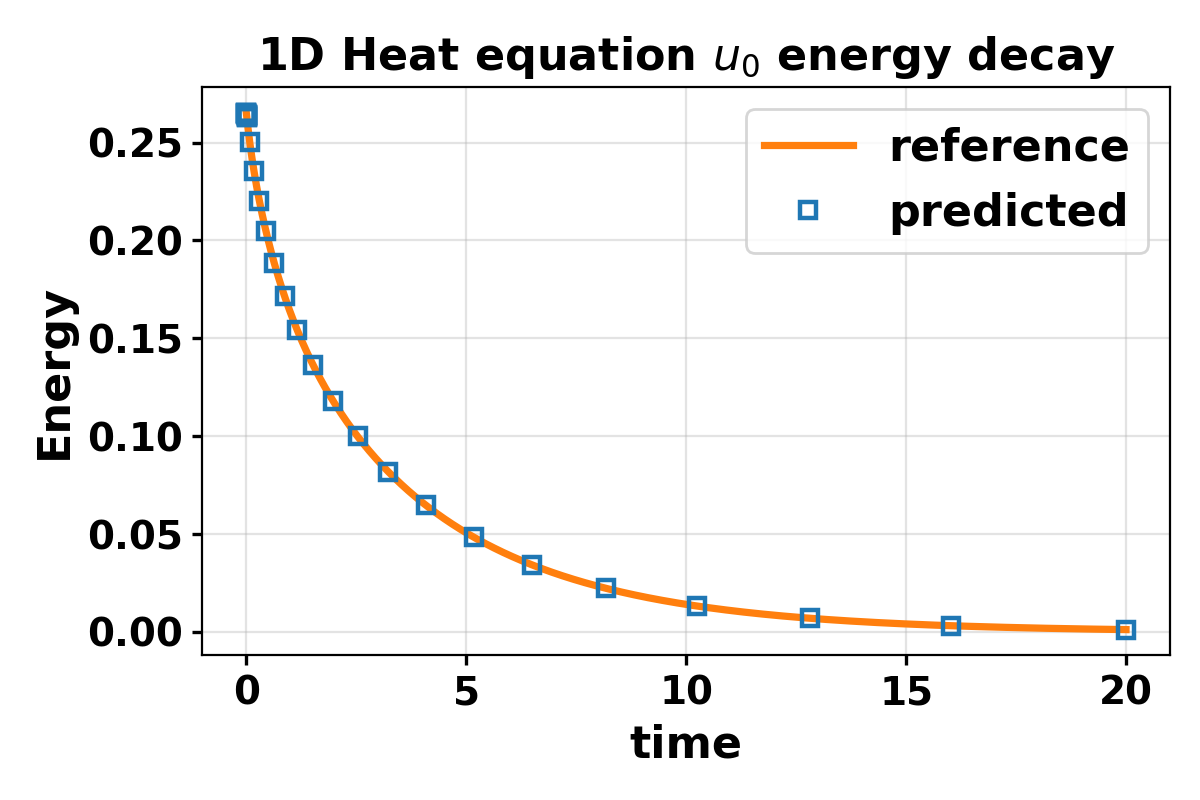}
    \caption{1D heat with time-dependent Dirichlet data: energy of the lifted interior state $u_0$.}
    \label{fig:1d_heat_u0_energy}
  \end{subfigure}

  \caption{
  \textbf{1D Dirichlet benchmarks: block-level operator matching and structure diagnostics.}
  Figs. \ref{fig:1d_uxx_op}--\ref{fig:1d_uux_op}: physical-space comparisons for pretrained blocks on a held-out sample; errors use a weighted $L^2$ norm on Gauss--Legendre nodes.
  Figs. \ref{fig:1d_gl_energy}--\ref{fig:1d_heat_u0_energy}: energy diagnostics for 1D Ginzburg--Landau and for the interior component $u_0$ in the heat equation after time-dependent boundary lifting.
  }
  \label{fig:1d_summary}
\end{figure}

% =========================
% Figure 2: 2D experiments (periodic + Neumann)
% =========================
\begin{figure}[H]
  \centering

  % Row 1: 2D Allen--Cahn
  \begin{subfigure}[t]{0.49\linewidth}
    \centering
    \includegraphics[width=\linewidth]{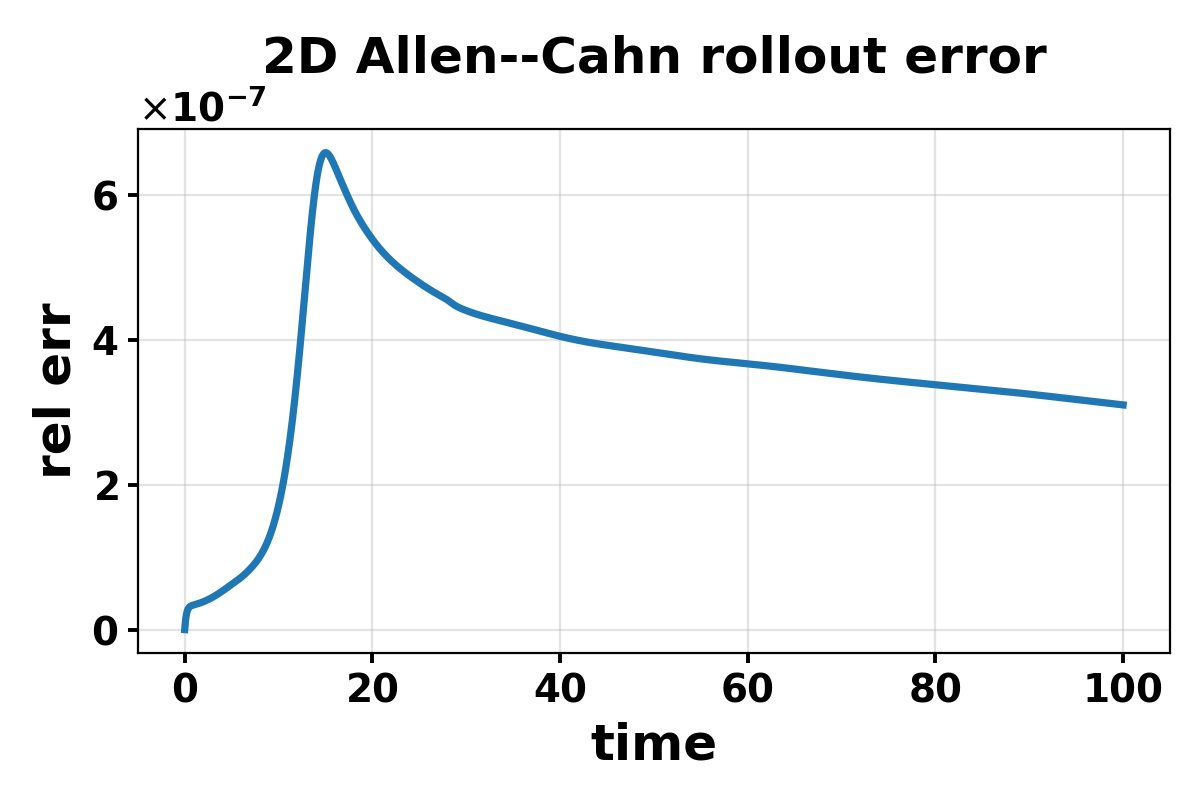}
    \caption{2D Allen--Cahn: rollout relative $L^2$ error.}
    \label{fig:ac2d_relerr}
  \end{subfigure}\hfill
  \begin{subfigure}[t]{0.49\linewidth}
    \centering
    \includegraphics[width=\linewidth]{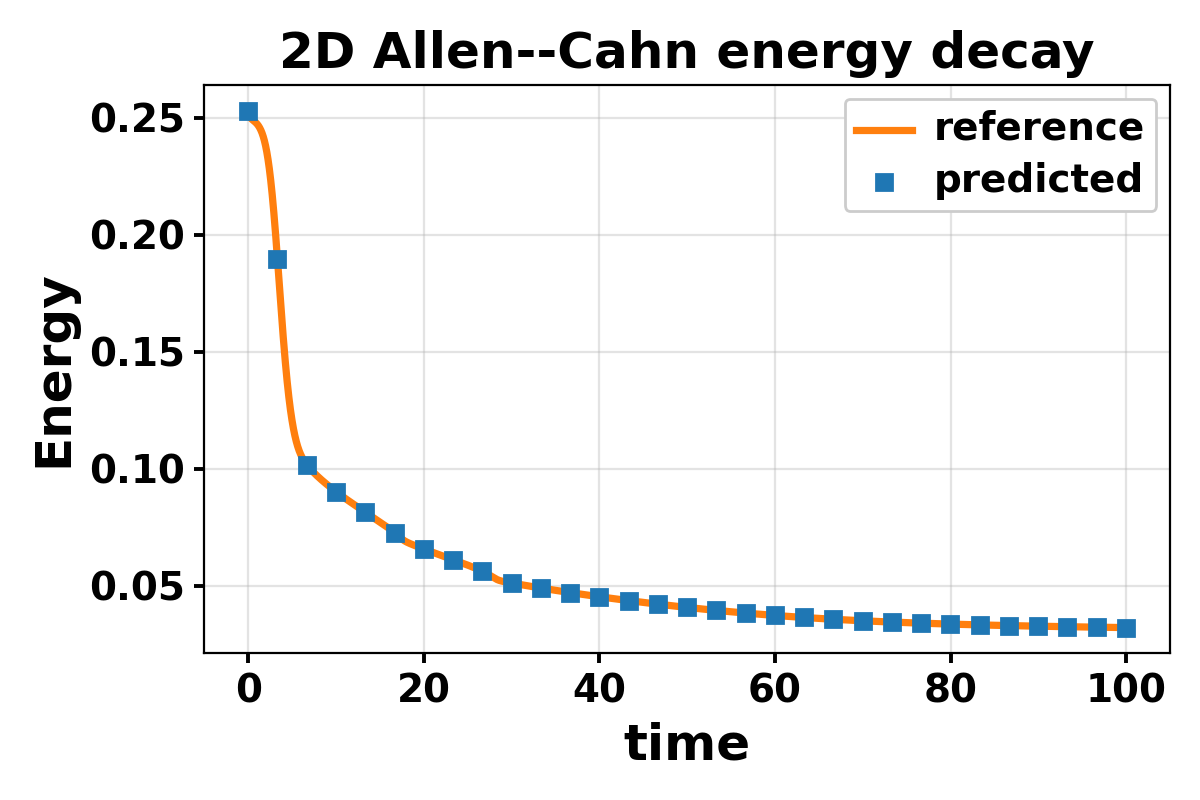}
    \caption{2D Allen--Cahn: energy decay.}
    \label{fig:ac2d_energy}
  \end{subfigure}

  \vspace{0.8ex}

  % Row 2: 2D vector Burgers
  \begin{subfigure}[t]{0.49\linewidth}
    \centering
    \includegraphics[width=\linewidth]{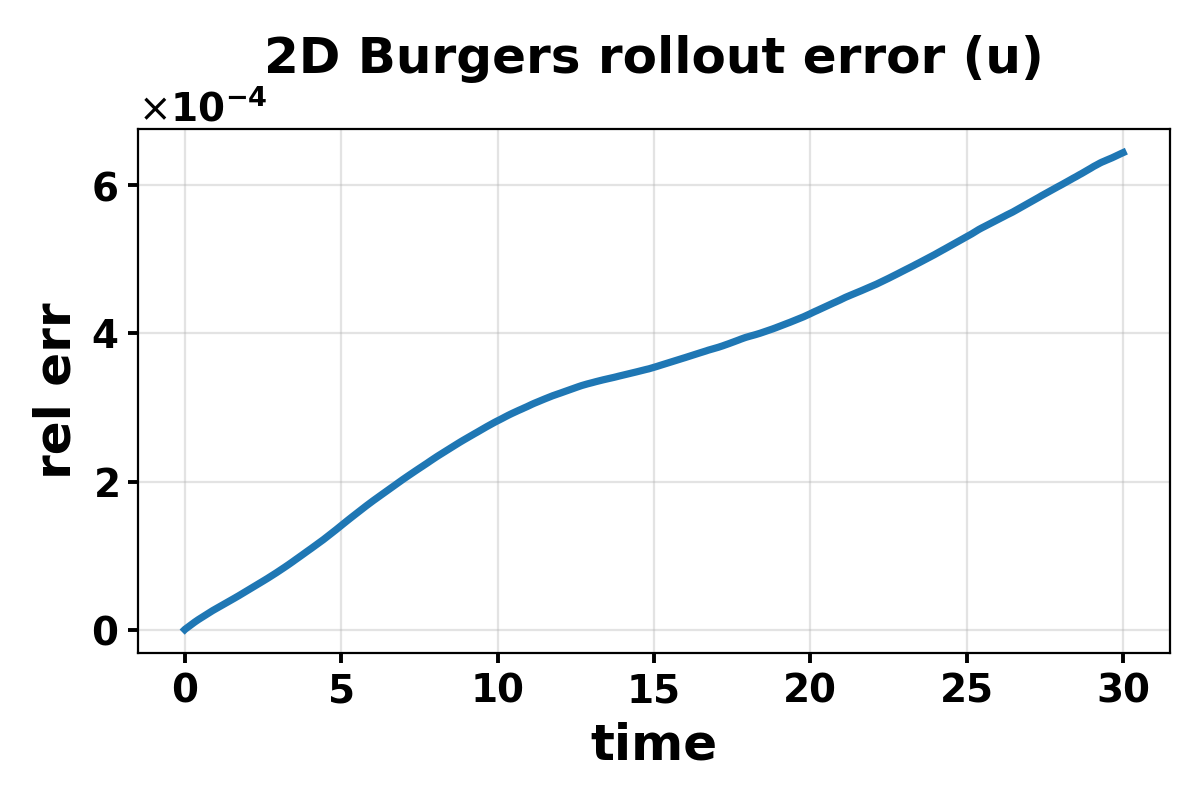}
    \caption{2D vector Burgers: rollout error for $u$.}
    \label{fig:burgers2d_u_relerr}
  \end{subfigure}\hfill
  \begin{subfigure}[t]{0.49\linewidth}
    \centering
    \includegraphics[width=\linewidth]{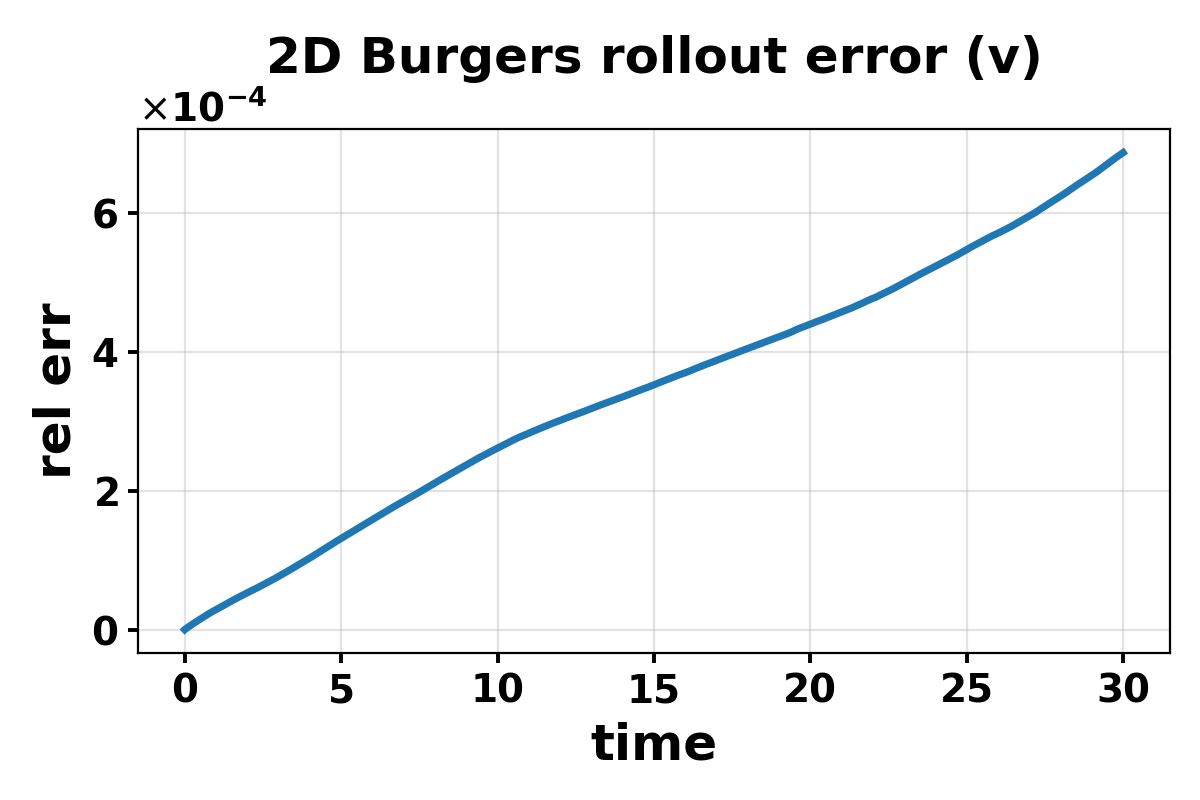}
    \caption{2D vector Burgers: rollout error for $v$.}
    \label{fig:burgers2d_v_relerr}
  \end{subfigure}

  \vspace{0.8ex}

  % Row 3: 2D Neumann (cosine baseplate)
  \begin{subfigure}[t]{0.49\linewidth}
    \centering
    \includegraphics[width=\linewidth]{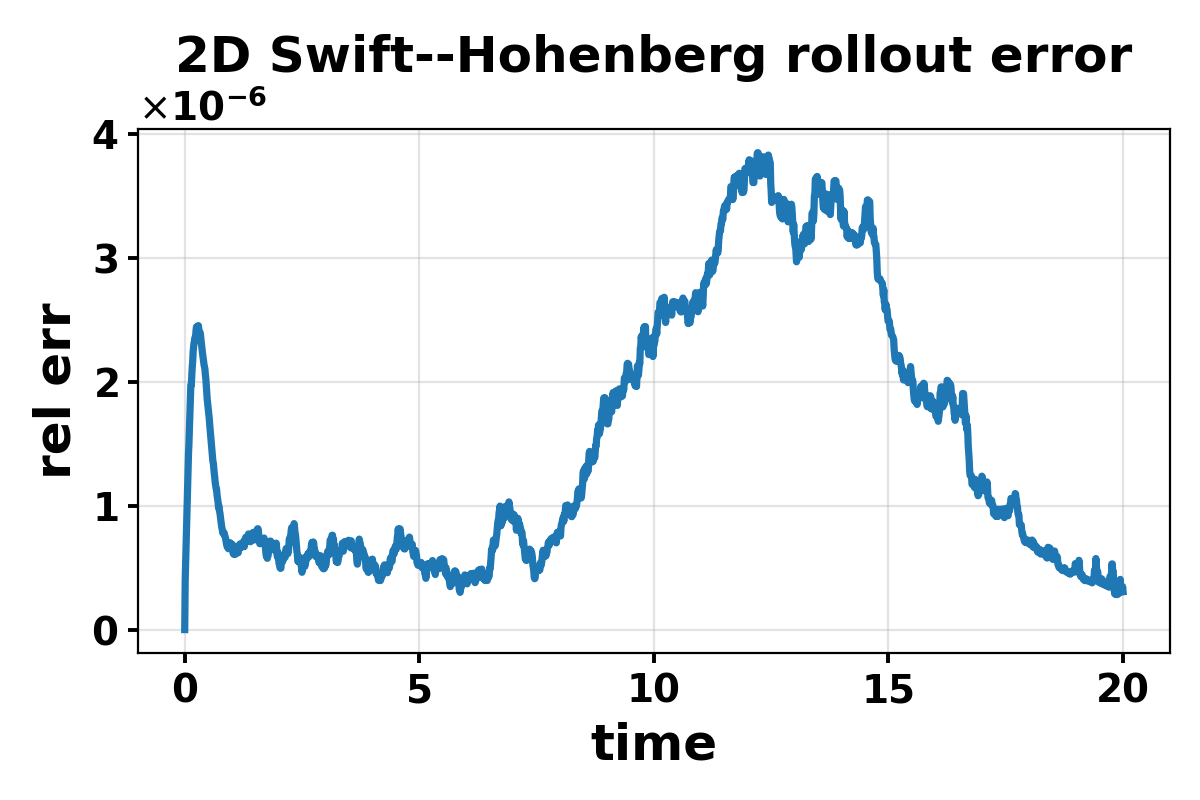}
    \caption{2D Swift--Hohenberg (Neumann): rollout relative $L^2$ error.}
    \label{fig:sh2d_relerr}
  \end{subfigure}\hfill
  \begin{subfigure}[t]{0.49\linewidth}
    \centering
    \includegraphics[width=\linewidth]{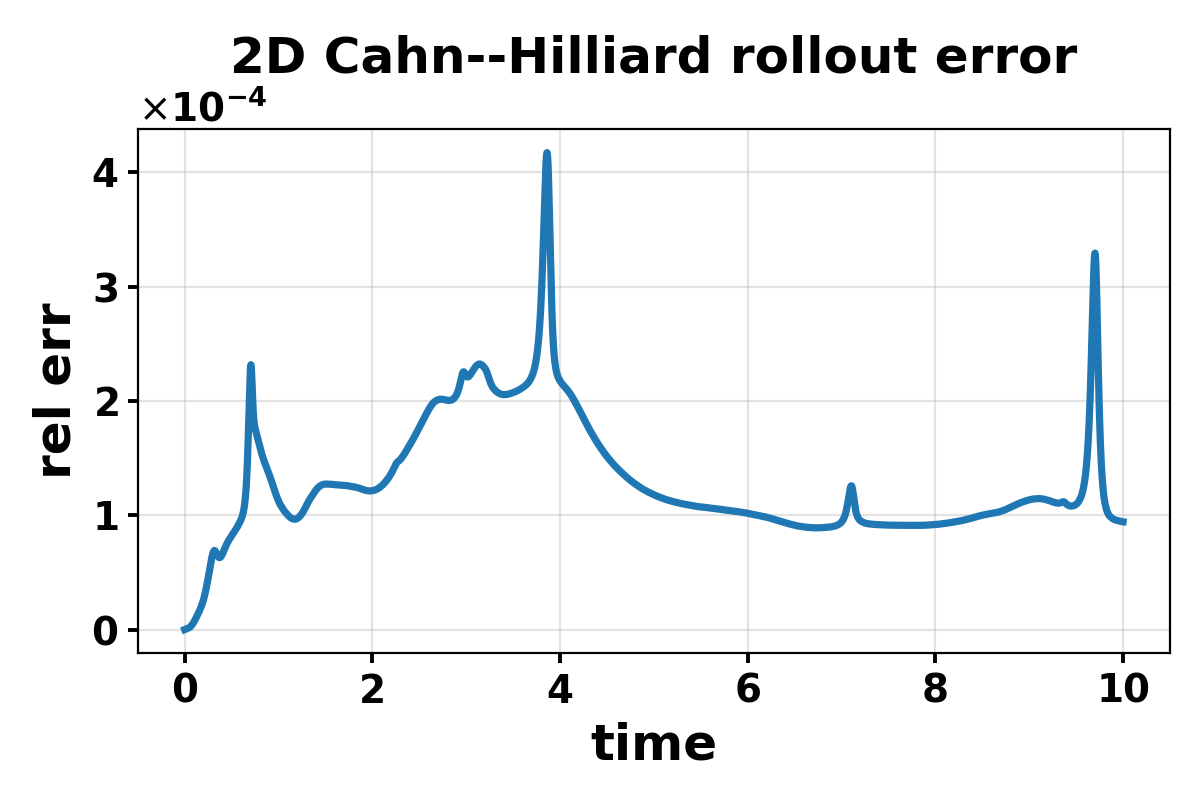}
    \caption{2D Cahn--Hilliard with lifted Neumann data: rollout relative $L^2$ error.}
    \label{fig:ch2d_relerr}
  \end{subfigure}

  \caption{
  \textbf{2D benchmarks: rollout accuracy and energy diagnostics across periodic and Neumann baseplates.}
  Figs. \ref{fig:ac2d_relerr}--\ref{fig:ac2d_energy}: 2D Allen--Cahn: relative $L^2$ trajectory error and Allen--Cahn energy decay.
  Figs. \ref{fig:burgers2d_u_relerr}--\ref{fig:burgers2d_v_relerr}: 2D vector Burgers: relative $L^2$ trajectory error for both velocity components.
  Figs. \ref{fig:sh2d_relerr}--\ref{fig:ch2d_relerr}: 2D Neumann benchmarks in a cosine trial space: Swift--Hohenberg and Cahn--Hilliard rollout error curves.
  }
  \label{fig:2d_summary}
\end{figure}

\clearpage
\section*{Supplementary Information}

\setcounter{equation}{0}
\setcounter{figure}{0}
\setcounter{table}{0}
\renewcommand{\theequation}{S\arabic{equation}}
\renewcommand{\thefigure}{S\arabic{figure}}
\renewcommand{\thetable}{S\arabic{table}}

\section{Proof of Theorem~\ref{thm:main_structure_error}}
\label{sec:appendix_error}

We collect the standing conditions and supporting estimates used in the proof of
Theorem~\ref{thm:main_structure_error}.
Throughout, $\|\cdot\|_2$ denotes the Euclidean norm on $\mathbb{R}^K$, and $\|\cdot\|_{w,2}$ is the fixed physical norm induced by the baseplate.

\begin{assumption}[Standing conditions]\label{ass:standing_app}
Fix $T>0$ and let $\mathcal{K}\subset\mathbb{R}^K$ be compact.
Let $\mathbf{a}(t)$ solve the reference reduced dynamics \eqref{eq:ref_dyn} with $\mathbf{a}(0)=\mathbf{a}_0$.
Set $t_n=n\Delta t$ and $N_{\mathrm{steps}}=T/\Delta t$.
\begin{enumerate}\setlength{\itemsep}{2pt}
\item \textbf{Containment on $\mathcal K$.}
The reference trajectory satisfies $\mathbf{a}(t)\in\mathcal K$ for all $t\in[0,T]$.
Moreover, there exists $\Delta t_0>0$ such that for $0<\Delta t\le \Delta t_0$,
all intermediate states produced by the reference and learned Strang schedules up to time $T$
remain in $\mathcal K$.

\item \textbf{Regularity of block fields.}
For each block $i$, the vector fields $F_i^{\mathrm{ref}}$ and $F_i^{\boldsymbol\theta}$
are Lipschitz on $\mathcal K$ with constants $L_i>0$.
In particular, $F^{\mathrm{ref}}=\sum_{i=1}^{N_{\mathrm{blk}}}F_i^{\mathrm{ref}}$ is Lipschitz on $\mathcal K$.

\item \textbf{Within-block accuracy (second order).}
For each $i$ and each $\tau\in\{\Delta t/2,\Delta t\}$ used by the Strang schedule, there exist one-step maps
$S_{i,\tau}^{\mathrm{ref}}$ and $S_{i,\tau}^{\boldsymbol\theta}$ for the isolated sub-dynamics
$\mathbf a_t=F_i^{\mathrm{ref}}(\mathbf a)$ and $\mathbf a_t=F_i^{\boldsymbol\theta}(\mathbf a)$ such that,
for all $\mathbf a\in\mathcal K$,
\[
\big\|S_{i,\tau}^{\mathrm{ref}}(\mathbf a)-\varphi_{i,\tau}^{\mathrm{ref}}(\mathbf a)\big\|_2\le C_i^{\mathrm{ref}}\tau^{3},
\qquad
\big\|S_{i,\tau}^{\boldsymbol\theta}(\mathbf a)-\varphi_{i,\tau}^{\boldsymbol\theta}(\mathbf a)\big\|_2\le C_i^{\boldsymbol\theta}\tau^{3},
\]
where $\varphi_{i,t}^{\mathrm{ref}}$ and $\varphi_{i,t}^{\boldsymbol\theta}$ denote the exact subflows of
$\mathbf a_t=F_i^{\mathrm{ref}}(\mathbf a)$ and $\mathbf a_t=F_i^{\boldsymbol\theta}(\mathbf a)$, and
the constants $C_i^{\mathrm{ref}},C_i^{\boldsymbol\theta}>0$ are
independent of $\tau$ and $\Delta t$.

\item \textbf{Reference macro-step stability.}
Let $S_{\Delta t}^{\mathrm{ref}}$ be the symmetric Strang composition built from $\{S_{i,\tau}^{\mathrm{ref}}\}_{i=1}^{N_{\mathrm{blk}}}$.
There exist $\Delta t_0>0$ and $L\ge 0$ such that for all $\mathbf a,\mathbf b\in\mathcal K$ and $0<\Delta t\le \Delta t_0$,
\[
\big\|S_{\Delta t}^{\mathrm{ref}}(\mathbf a)-S_{\Delta t}^{\mathrm{ref}}(\mathbf b)\big\|_2
\le (1+L\Delta t)\,\|\mathbf a-\mathbf b\|_2.
\]

\item \textbf{Uniform block mismatch on $\mathcal K$.}
For each $i$,
\[
\varepsilon_i
:=
\sup_{\mathbf a\in\mathcal K}
\big\|F_i^{\boldsymbol\theta}(\mathbf a)-F_i^{\mathrm{ref}}(\mathbf a)\big\|_2
<\infty.
\]

\item \textbf{Bounded reconstruction.}
There exists $C_\Phi>0$ such that $\|\Phi_b \mathbf v\|_{w,2}\le C_\Phi\|\mathbf v\|_2$ for all $\mathbf v\in\mathbb{R}^K$.
\end{enumerate}
\end{assumption}

Assumption~\ref{ass:standing_app}\,(3) is standard in practice.
Each block evolves on the same finite-dimensional coefficient interface, so its isolated sub-dynamics can be advanced by a block-adapted second-order one-step scheme, such as an exact subflow when available or a standard second-order integrator; see \cite{hairer2006structure,sanz2018numerical,strang1968construction}.
Assumption~\ref{ass:standing_app}\,(5) quantifies the uniform block mismatch on $\mathcal{K}$.
Its size depends on the approximation class and on training quality, including the expressiveness of the block model and the coverage and accuracy of the operator-matching samples.
In principle, $\varepsilon_i$ can be made small with richer model classes, more representative training data, and sufficiently effective optimization, provided the target block map is well approximated within the chosen class; see \cite{cybenko1989approximation,hornik1989multilayer}.

After Assumption~\ref{ass:standing_app}, we summarize the block-level structural consequences of the parameterization
\[
F_i^{\boldsymbol{\theta}}(\mathbf a)
=
-\,G_i\nabla_{\mathbf a}E_i^{a,\boldsymbol{\theta}}(\mathbf a)
+
J_i\nabla_{\mathbf a}H_i^{a,\boldsymbol{\theta}}(\mathbf a)
+
R_i^a(\mathbf a).
\]
For clarity, we first state the dissipative and conservative cases in isolation, where the residual term is absent; see
\cite{ambrosio2005gradient,hairer2006structure}.

\begin{property}[Energy dissipation for learned dissipative blocks]
\label{prop:dissipation}
Let $E_i^{a,\boldsymbol{\theta}}:\mathbb{R}^{K}\to\mathbb{R}$ be continuously differentiable and consider the isolated dissipative coefficient dynamics
\begin{equation}\label{eq:dissip_block_theta}
\mathbf a_t
=
-\,G_i\,\nabla_{\mathbf a} E_i^{a,\boldsymbol{\theta}}(\mathbf a),
\qquad
G_i^\top=G_i,\quad G_i\succeq 0 .
\end{equation}
Here, $G_i^\top=G_i$ and $G_i\succeq 0$ mean that $G_i$ is symmetric positive semidefinite.
Then along any trajectory $\mathbf a(t)$ of \eqref{eq:dissip_block_theta}, the learned scalar generator is non-increasing:
\[
\frac{d}{dt}E_i^{a,\boldsymbol{\theta}}\bigl(\mathbf a(t)\bigr)\le 0.
\]
\end{property}

\begin{proof}
By the chain rule,
\[
\frac{d}{dt}E_i^{a,\boldsymbol{\theta}}\bigl(\mathbf a(t)\bigr)
=
\nabla_{\mathbf a}E_i^{a,\boldsymbol{\theta}}(\mathbf a)^\top \mathbf a_t
=
-\nabla_{\mathbf a}E_i^{a,\boldsymbol{\theta}}(\mathbf a)^\top
G_i
\nabla_{\mathbf a}E_i^{a,\boldsymbol{\theta}}(\mathbf a)
\le 0,
\]
because $G_i\succeq 0$ implies $\xi^\top G_i \xi\ge 0$ for all $\xi\in\mathbb R^K$.
\end{proof}

\begin{property}[Hamiltonian conservation for learned conservative blocks]
\label{prop:hamiltonian}
Let $H_i^{a,\boldsymbol{\theta}}:\mathbb{R}^{K}\to\mathbb{R}$ be continuously differentiable and consider the isolated conservative coefficient dynamics
\begin{equation}\label{eq:cons_block_theta}
\mathbf a_t
=
J_i\,\nabla_{\mathbf a} H_i^{a,\boldsymbol{\theta}}(\mathbf a),
\qquad
J_i^\top=-J_i .
\end{equation}
Here, $J_i^\top=-J_i$ means that $J_i$ is skew-symmetric.
Then along any trajectory $\mathbf a(t)$ of \eqref{eq:cons_block_theta}, the learned Hamiltonian is conserved:
\[
\frac{d}{dt}H_i^{a,\boldsymbol{\theta}}\bigl(\mathbf a(t)\bigr)=0.
\]
\end{property}

\begin{proof}
By the chain rule,
\[
\frac{d}{dt}H_i^{a,\boldsymbol{\theta}}\bigl(\mathbf a(t)\bigr)
=
\nabla_{\mathbf a}H_i^{a,\boldsymbol{\theta}}(\mathbf a)^\top \mathbf a_t
=
\nabla_{\mathbf a}H_i^{a,\boldsymbol{\theta}}(\mathbf a)^\top
J_i
\nabla_{\mathbf a}H_i^{a,\boldsymbol{\theta}}(\mathbf a).
\]
For any skew-symmetric matrix $J_i$, one has $\xi^\top J_i \xi=0$ for all $\xi\in\mathbb R^K$.
Hence the derivative vanishes.
\end{proof}

Under Assumption~\ref{ass:standing_app}, the following is a standard consequence of symmetric Strang splitting; see
\cite{strang1968construction,hairer2006structure}.

\begin{lemma}[Strang composition {\normalfont\cite{strang1968construction,hairer2006structure}}]
\label{lem:strang_global_app}
%Let $F^{\mathrm{ref}}(\mathbf a)=\sum_{i=1}^{N_{\mathrm{blk}}}F_i^{\mathrm{ref}}(\mathbf a)$ and let $\varphi^{\mathrm{ref}}_t$ denote the exact flow of $\mathbf a_t=F^{\mathrm{ref}}(\mathbf a)$.
%Let $S_{\Delta t}^{\mathrm{ref}}$ be the symmetric Strang macro-step built from the isolated within-block maps
%$\{S_{i,\tau}^{\mathrm{ref}}\}$ with $\tau\in\{\Delta t/2,\Delta t\}$.
Under Assumption~\ref{ass:standing_app}, for sufficiently small $\Delta t$, there exist constant $C_T^{\mathrm{spl}}>0$, independent of $\Delta t$, such that
\[
\bigl\|\mathbf a_{N_{\mathrm{steps}}}^{\mathrm{ref}}-\mathbf a(T)\bigr\|_2
\le
C_T^{\mathrm{spl}}\,\Delta t^2,
\]
where $\mathbf a_{n+1}^{\mathrm{ref}}=S_{\Delta t}^{\mathrm{ref}}(\mathbf a_n^{\mathrm{ref}})$ and $\mathbf a_0^{\mathrm{ref}}=\mathbf a_0$.
Moreover, suppose that, for a dissipative block $i$, its isolated within-block update satisfies
\[
E_i^{a,\bullet}\!\left(S_{i,\tau}^{\bullet}(\mathbf a)\right)\le E_i^{a,\bullet}(\mathbf a),
\qquad
\forall\,\mathbf a\in\mathcal K,\ \tau\in\{\Delta t/2,\Delta t\},
\]
or, for a conservative block $i$, its isolated within-block update satisfies
\[
H_i^{a,\bullet}\!\left(S_{i,\tau}^{\bullet}(\mathbf a)\right)= H_i^{a,\bullet}(\mathbf a),
\qquad
\forall\,\mathbf a\in\mathcal K,\ \tau\in\{\Delta t/2,\Delta t\},
\]
where $\bullet\in\{\mathrm{ref},\boldsymbol{\theta}\}$.
Then the same inequality or equality holds at the corresponding substeps inside the symmetric Strang schedule.
\end{lemma}

\begin{lemma}[Exact subflow perturbation on $\mathcal K$]\label{lem:subflow_perturb}
Assume Assumption~\ref{ass:standing_app}. Fix a block $i$.
For any $\mathbf a\in\mathcal K$ and any $t\ge 0$ such that both subflows stay in $\mathcal K$ on $[0,t]$,
\[
\big\|\varphi_{i,t}^{\boldsymbol\theta}(\mathbf a)-\varphi_{i,t}^{\mathrm{ref}}(\mathbf a)\big\|_2
\le
\frac{e^{L_i t}-1}{L_i}\,\varepsilon_i
\le
t\,e^{L_i t}\,\varepsilon_i.
\]
\end{lemma}

\begin{proof}
Let $\mathbf a_\theta(s)=\varphi_{i,s}^{\boldsymbol\theta}(\mathbf a)$ and $\mathbf a_{\mathrm{ref}}(s)=\varphi_{i,s}^{\mathrm{ref}}(\mathbf a)$.
Then
\[
\frac{d}{ds}\big(\mathbf a_\theta-\mathbf a_{\mathrm{ref}}\big)
=
F_i^{\boldsymbol\theta}(\mathbf a_\theta)-F_i^{\mathrm{ref}}(\mathbf a_{\mathrm{ref}})
=
\underbrace{F_i^{\boldsymbol\theta}(\mathbf a_\theta)-F_i^{\boldsymbol\theta}(\mathbf a_{\mathrm{ref}})}_{\text{Lipschitz}}
+
\underbrace{F_i^{\boldsymbol\theta}(\mathbf a_{\mathrm{ref}})-F_i^{\mathrm{ref}}(\mathbf a_{\mathrm{ref}})}_{\le \varepsilon_i}.
\]
Taking norms and using Lipschitz continuity on $\mathcal K$ yields
\[
\frac{d}{ds}\|\mathbf a_\theta-\mathbf a_{\mathrm{ref}}\|_2
\le
L_i\|\mathbf a_\theta-\mathbf a_{\mathrm{ref}}\|_2+\varepsilon_i.
\]
By Gr\"onwall \cite{hartman2002ordinary},
\(
\|\mathbf a_\theta(t)-\mathbf a_{\mathrm{ref}}(t)\|_2
\le
\int_0^t e^{L_i(t-s)}\varepsilon_i\,ds
=
\frac{e^{L_i t}-1}{L_i}\varepsilon_i
\),
and the second inequality follows from $\frac{e^{x}-1}{x}\le e^{x}$ for $x\ge 0$.
\end{proof}

\begin{lemma}[Substep defect]\label{lem:substep_defect}
Assume Assumption~\ref{ass:standing_app}. Fix a block $i$ and let $\tau\in\{\Delta t/2,\Delta t\}$.
Then for all $\mathbf a\in\mathcal K$,
\[
\big\|S_{i,\tau}^{\boldsymbol\theta}(\mathbf a)-S_{i,\tau}^{\mathrm{ref}}(\mathbf a)\big\|_2
\le
\tau e^{L_i\tau}\,\varepsilon_i
+
\big(C_i^{\boldsymbol\theta}+C_i^{\mathrm{ref}}\big)\tau^{3}.
\]
\end{lemma}

\begin{proof}
Insert and subtract the exact subflows:
\[
\big\|S_{i,\tau}^{\boldsymbol\theta}(\mathbf a)-S_{i,\tau}^{\mathrm{ref}}(\mathbf a)\big\|_2
\le
\big\|S_{i,\tau}^{\boldsymbol\theta}(\mathbf a)-\varphi_{i,\tau}^{\boldsymbol\theta}(\mathbf a)\big\|_2
+
\big\|\varphi_{i,\tau}^{\boldsymbol\theta}(\mathbf a)-\varphi_{i,\tau}^{\mathrm{ref}}(\mathbf a)\big\|_2
+
\big\|\varphi_{i,\tau}^{\mathrm{ref}}(\mathbf a)-S_{i,\tau}^{\mathrm{ref}}(\mathbf a)\big\|_2.
\]
Apply Assumption~\ref{ass:standing_app}(3) to the first and third terms, and Lemma~\ref{lem:subflow_perturb} to the middle term.
\end{proof}

\begin{lemma}[One-step macro defect]\label{lem:macro_defect_app_revised}
Assume Assumption~\ref{ass:standing_app}. Let $S_{\Delta t}^{\boldsymbol\theta}$ and $S_{\Delta t}^{\mathrm{ref}}$
be the learned and reference symmetric Strang macro-steps built from the same schedule and substep durations
$\tau\in\{\Delta t/2,\Delta t\}$. Then there exist constants $C_{\mathrm{blk}},C_{\mathrm{num}}>0$,
independent of $\Delta t$, such that for all $\mathbf a\in\mathcal K$ and sufficiently small $\Delta t$,
\[
\big\|S_{\Delta t}^{\boldsymbol\theta}(\mathbf a)-S_{\Delta t}^{\mathrm{ref}}(\mathbf a)\big\|_2
\le
C_{\mathrm{blk}}\,\Delta t\sum_{i=1}^{N_{\mathrm{blk}}}\varepsilon_i
\;+\;
C_{\mathrm{num}}\,\Delta t^{3}.
\]
\end{lemma}

\begin{proof}
Write the Strang schedule as a fixed finite composition of substep maps,
\[
S_{\Delta t}^{\mathrm{ref}}=T_m^{\mathrm{ref}}\circ\cdots\circ T_1^{\mathrm{ref}},
\qquad
S_{\Delta t}^{\boldsymbol\theta}=T_m^{\boldsymbol\theta}\circ\cdots\circ T_1^{\boldsymbol\theta},
\]
where each $T_j$ is some $S_{i,\tau}$ with $\tau\in\{\Delta t/2,\Delta t\}$.
Define intermediate states
$\mathbf z_0=\mathbf a$, $\mathbf z_j=T_j^{\mathrm{ref}}(\mathbf z_{j-1})$
and
$\mathbf w_0=\mathbf a$, $\mathbf w_j=T_j^{\boldsymbol\theta}(\mathbf w_{j-1})$.
By containment, $\mathbf z_j,\mathbf w_j\in\mathcal K$.

A telescoping bound gives
\[
\|\mathbf w_m-\mathbf z_m\|_2
\le
\sum_{j=1}^m
\Big\|
T_m^{\boldsymbol\theta}\circ\cdots\circ T_{j+1}^{\boldsymbol\theta}(\mathbf w_j)
-
T_m^{\boldsymbol\theta}\circ\cdots\circ T_{j+1}^{\boldsymbol\theta}(\mathbf z_j)
\Big\|_2
+
\sum_{j=1}^m
\|T_j^{\boldsymbol\theta}(\mathbf z_{j-1})-T_j^{\mathrm{ref}}(\mathbf z_{j-1})\|_2.
\]
Using Lipschitz continuity of each substep map on $\mathcal K$ (implied by Assumption~\ref{ass:standing_app}(2) for sufficiently small $\tau$)
bounds the first sum by a constant multiple of $\max_j\|\mathbf w_j-\mathbf z_j\|_2$ and is absorbed into the second sum for small $\Delta t$.
For the second sum, apply Lemma~\ref{lem:substep_defect} to each occurrence of block $i$:
each contributes $\mathcal O(\tau\,\varepsilon_i)+\mathcal O(\tau^3)$.
Since the schedule contains a fixed finite number of substeps and $\sum_j \tau = \mathcal O(\Delta t)$, we obtain
\[
\|S_{\Delta t}^{\boldsymbol\theta}(\mathbf a)-S_{\Delta t}^{\mathrm{ref}}(\mathbf a)\|_2
\le
C_{\mathrm{blk}}\,\Delta t\sum_i \varepsilon_i
+
C_{\mathrm{num}}\,\Delta t^3,
\]
with constants depending only on the Lipschitz bounds on $\mathcal K$ and the fixed schedule.
\end{proof}

\begin{lemma}[Discrete Gr\"onwall {\normalfont\cite{hairer1993solving}}]\label{lem:gronwall_app}
Let $\Delta t>0$, $N\in\mathbb{N}$, and set $T_N:=N\Delta t$.
Assume $L\ge 0$, $\alpha\ge 0$, and a nonnegative sequence $\{e_n\}_{n=0}^N$ satisfies
\[
e_{n+1}\le (1+L\Delta t)e_n + \alpha \Delta t,
\qquad n=0,\ldots,N-1,
\qquad e_0=0.
\]
Then
\[
e_N \le \alpha\,T_N\,\exp(LT_N)
\le \alpha\,T\,\exp(LT),
\qquad \text{for any }T\ge T_N.
\]
\end{lemma}

\begin{proof}[Proof of Theorem~\ref{thm:main_structure_error}]
The structure claim is immediate: as shown in Lemma \ref{lem:strang_global_app}, the Strang macro-step is a fixed composition of the within-block maps
$S_{i,\tau}^{\boldsymbol{\theta}}$, so any per-substep monotonicity (resp.\ invariance) of
$E_i^{a,\boldsymbol{\theta}}$ (resp.\ $H_i^{a,\boldsymbol{\theta}}$) is inherited at the corresponding locations in the schedule.

For the error bound, decompose
\[
\|\mathbf{a}_{N_{\mathrm{steps}}}^{\boldsymbol{\theta}}-\mathbf{a}(T)\|_2
\le
\|\mathbf{a}_{N_{\mathrm{steps}}}^{\boldsymbol{\theta}}-\mathbf{a}_{N_{\mathrm{steps}}}^{\mathrm{ref}}\|_2
+
\|\mathbf{a}_{N_{\mathrm{steps}}}^{\mathrm{ref}}-\mathbf{a}(T)\|_2.
\]
The second term is controlled by Lemma~\ref{lem:strang_global_app}.
For the first term, set $\mathbf{e}_n:=\mathbf{a}_n^{\boldsymbol{\theta}}-\mathbf{a}_n^{\mathrm{ref}}$.
Using triangle inequality, for each step
\begin{align*}
\|\mathbf{e}_{n+1}\|_2
&=
\big\|S_{\Delta t}^{\boldsymbol{\theta}}(\mathbf{a}_n^{\boldsymbol{\theta}})-S_{\Delta t}^{\mathrm{ref}}(\mathbf{a}_n^{\mathrm{ref}})\big\|_2 \\
&\le
\underbrace{\big\|S_{\Delta t}^{\boldsymbol{\theta}}(\mathbf{a}_n^{\boldsymbol{\theta}})
      -S_{\Delta t}^{\mathrm{ref}}(\mathbf{a}_n^{\boldsymbol{\theta}})\big\|_2}_{\text{macro defect at the same input}}
+
\underbrace{\big\|S_{\Delta t}^{\mathrm{ref}}(\mathbf{a}_n^{\boldsymbol{\theta}})
      -S_{\Delta t}^{\mathrm{ref}}(\mathbf{a}_n^{\mathrm{ref}})\big\|_2}_{\text{stability of }S_{\Delta t}^{\mathrm{ref}}}.
\end{align*}
By Lemma~\ref{lem:macro_defect_app_revised},
the first term is bounded by
$C_{\mathrm{blk}}\Delta t\sum_{i=1}^{N_{\mathrm{blk}}}\varepsilon_i + C_{\mathrm{num}}\Delta t^{3}$.
By Assumption~\ref{ass:standing_app}\,(4), the second term is bounded by $(1+L\Delta t)\|\mathbf{e}_n\|_2$.
Hence,
\[
\|\mathbf{e}_{n+1}\|_2
\le
(1+L\Delta t)\|\mathbf{e}_n\|_2
+
C_{\mathrm{blk}}\Delta t\sum_{i=1}^{N_{\mathrm{blk}}}\varepsilon_i
+
C_{\mathrm{num}}\Delta t^{3}.
\]
Applying Lemma~\ref{lem:gronwall_app} gives
\[
\|\mathbf{e}_{N_{\mathrm{steps}}}\|_2
\le
C_T^{\mathrm{blk}}\,T\sum_{i=1}^{N_{\mathrm{blk}}}\varepsilon_i
+
C_T^{\mathrm{num}}\,T\Delta t^{2}.
\]
Here we set constants $C_T^{\mathrm{blk}}:=C_{\mathrm{blk}}e^{LT}$ (and similarly for $C_T^{\mathrm{num}}$).
Combining with Lemma~\ref{lem:strang_global_app} gives
\[
\|\mathbf{a}_{N_{\mathrm{steps}}}^{\boldsymbol{\theta}}-\mathbf{a}(T)\|_2
\le
C_T^{\mathrm{blk}}\,T\sum_{i=1}^{N_{\mathrm{blk}}}\varepsilon_i
+
C_T^{\mathrm{spl}}\Delta t^2.
\]
Finally, by Assumption~\ref{ass:standing_app}\,(6),
$\|u_{N_{\mathrm{steps}}}^{\boldsymbol{\theta}}-u(T)\|_{w,2}
=\|\Phi_b(\mathbf{a}_{N_{\mathrm{steps}}}^{\boldsymbol{\theta}}-\mathbf{a}(T))\|_{w,2}
\le C_\Phi\|\mathbf{a}_{N_{\mathrm{steps}}}^{\boldsymbol{\theta}}-\mathbf{a}(T)\|_2$,
which yields \eqref{eq:main_error_bound} after absorbing constants into $C_T$.
\end{proof}

\section{Block architectures and training details}
\label{sec:appendix_trainblocks}

This section records implementation details for block parameterizations and pretraining used in the numerical experiments.
Each block is pretrained by the unified operator-matching objective~\eqref{eq:method_loss} on coefficient samples
$\mathbf a\sim \mu_b$, where $\mu_b$ is a baseplate-dependent spectral-decay Gaussian prior on the retained coefficient interface.

\subsection{Coefficient prior $\mu_b$.}
In 1D, we sample independent coordinates
\begin{equation}\label{eq:app_gaussian_prior_1d}
a_k \sim \mathcal{N}(0,\sigma_k^2),
\qquad
\sigma_k=\frac{\mathrm{amp}}{(1+k)^\alpha},
\qquad k=1,\ldots,K .
\end{equation}
In 2D and 3D bases, we retain modes indexed by multi-indices $(j,\ell)$ (2D) or $(j,\ell,m)$ (3D).
For notational simplicity, we fix a bijection between the retained multi-indices and a scalar index
$k\in\{1,\dots,K\}$, and denote the corresponding coefficient by $a_k$.
We then set
\begin{equation}\label{eq:app_gaussian_prior_k}
\sigma_k=\frac{\mathrm{amp}}{\bigl(1+\|k\|_2\bigr)^{\alpha}}.
\end{equation}
where $k$ denotes the underlying multi-index.
Here, we take $\mathrm{amp}=1$, $\alpha=0.5$.
For periodic Fourier baseplates, we draw coefficients in the real-valued Hermitian-packed representation:
for each non self-conjugate mode, we sample
$\Re(a_{k}),\Im(a_{k})\sim\mathcal{N}\!\bigl(0,\sigma_k^2\bigr)$ independently,
while for self-conjugate modes only the real part is sampled; unpacking enforces Hermitian symmetry.
For cosine (Neumann) baseplates, coefficients are real, and we sample
$a_{k}\sim\mathcal{N}\!\bigl(0,\sigma_k^2\bigr)$ using \eqref{eq:app_gaussian_prior_k}.
Unless otherwise stated, we generate $20{,}000$ coefficient samples for training.

Targets in~\eqref{eq:method_loss} are generated by the corresponding exact Galerkin/spectral operators restricted to the retained modes.
Optimization uses AdamW with StepLR schedules, and all pretrained blocks are reused unchanged at inference time.

\subsection{1D Dirichlet baseplate (Shen--Legendre)}
\label{sec:app_train_1d}

We consider $\Omega=(-1,1)$ with $u(\pm 1)=0$ and represent fields in Shen's Legendre basis
\cite{shen1994efficient}
\[
\phi_k(x)=L_{k-1}(x)-L_{k+1}(x),\qquad k=1,\ldots,K,
\]
so that $\phi_k(\pm 1)=0$ and $u(x)=\sum_{k=1}^{K} a_k\,\phi_k(x)$.
Grid evaluations use a Gauss--Legendre quadrature grid $\{x_q\}_{q=1}^Q$ with basis matrix
$\Phi_b\in\mathbb{R}^{Q\times K}$, $(\Phi_b)_{qk}=\phi_k(x_q)$.
We use $Q=256$ and $K=96$.
The baseplate projection $\mathcal{P}_b$ is the discrete $L^2$ projection induced by the mass matrix
$M_{ij}=\langle \phi_i,\phi_j\rangle_{L^2}$.

\subsubsection*{Diffusion block ($u\mapsto u_{xx}$).}
We parameterize the energy generator $E_{u_{xx}}^{a,\boldsymbol{\theta}}:\mathbb{R}^K\to\mathbb{R}$ by an MLP
(4 hidden layers, width 128, GELU activations) and define the learned vector field by the coefficient-space gradient-flow form,
\begin{equation}
F_{u_{xx}}^{\boldsymbol{\theta}}(\mathbf a) \;=\; -\,G\,\nabla_a E_{u_{xx}}^{a,\boldsymbol{\theta}}(\mathbf a),
\end{equation}
where $G=M^{-1}$ is the fixed mobility induced by the discrete $L^2$ metric on the Shen space.
We train with AdamW (learning rate $10^{-3}$) and StepLR (step size 50, decay factor 0.3) for 200 epochs with batch size 128.

\subsubsection*{Transport block ($u\mapsto u\,u_x$).}
We learn a pointwise density $h_{\boldsymbol{\theta}}:\mathbb{R}\to\mathbb{R}$ (depth 4, width 128, GELU)
and induce a Hamiltonian generator via the density construction described in Section~\ref{sec:Methodology}.
The learned transport vector field takes the form
\begin{equation}
F_{uu_x}^{\boldsymbol{\theta}}(\mathbf a) \;=\; J\,\nabla_a H^{a,\boldsymbol{\theta}}_{uu_{x}}(\mathbf a),
\end{equation}
where $J=M^{-1}S$ is fixed and represents $\partial_x$ on the Shen space, with
$S_{ij}=\langle \partial_x\phi_i,\phi_j\rangle_{L^2}$.
We train $h_{\boldsymbol{\theta}}$ with AdamW (learning rate $10^{-4}$, weight decay $10^{-4}$) for 100 epochs with batch size 128.

Extended Data Figs.~\ref{fig:1d_uxx_op}-\ref{fig:1d_uux_op} provide block-level sanity checks on a held-out
coefficient state, comparing the learned operators against their
Galerkin-projected reference counterparts in physical space. The reported
discrepancies are on the order of $10^{-3}$ in the weighted $L^2$ norm
evaluated at Gauss--Legendre nodes.

\subsection{2D periodic baseplate (Fourier)}
\label{sec:app_train_2d_fourier}

We consider $\Omega=[0,2\pi)^2$ and represent real fields by a band-limited Fourier expansion
\[
u(x,y)=\sum_{|j|\le K_{\rm cut}}\sum_{|\ell|\le K_{\rm cut}}
a_{j,\ell}\,e^{i(jx+\ell y)},\qquad a_{-\!j,-\!\ell}=\overline{a_{j,\ell}} .
\]
Coefficients are stored in a real Hermitian-packed coordinate vector $\mathbf a\in\mathbb{R}^K$ that uniquely represents a real band-limited field.
Grid evaluations use an $N\times N$ uniform grid with FFT/iFFT transforms for $\Phi_b$ and $\mathcal{P}_b$.
Unless stated otherwise, $N=64$ and we retain modes up to $K_{\rm cut}=21$, so the retained complex mode set is
$(2K_{\rm cut}+1)^2$ and the packed real dimension is denoted by $K$.

\subsubsection*{Laplacian diffusion block ($u\mapsto \Delta u$).}
Since the Laplacian is mode-decoupled in the Fourier baseplate, we restrict the energy generator to a structured quadratic form,
\begin{equation}
E_{\Delta}^{a,\boldsymbol{\theta}}(\mathbf a)
\;=\;
\frac12\, \mathbf a^\top \!\operatorname{diag}(\mathbf c^{\boldsymbol{\theta}})\, \mathbf a,
\qquad \mathbf c^{\boldsymbol{\theta}}\in\mathbb{R}^K.
\end{equation}
We then define the diffusion vector field by the gradient-flow template
\begin{equation}\label{eq:fourier_delta}
F_{\Delta}^{\boldsymbol{\theta}}(\mathbf a)
\;=\;
-\,G\,\nabla_a E_{\Delta}^{a,\boldsymbol{\theta}}(\mathbf a),
\end{equation}
with $G=I$ for the orthonormal Fourier coefficient metric.
Training uses AdamW (learning rate $10^{-3}$) with StepLR (step size 40, decay factor 0.3) for 80 epochs (batch size 128).

\subsubsection*{Hamiltonian transport blocks ($u\mapsto uu_x$ and $u\mapsto uu_y$).}
We train two pointwise density networks $\rho_{\boldsymbol{\theta}_x}$ and $\rho_{\boldsymbol{\theta}_y}$
with identical architecture (depth 4, width 128, GELU), and assemble directional transport via fixed derivative operators
$J_x$ and $J_y$ that represent $\partial_x$ and $\partial_y$ on the retained modes:
\begin{equation}
F_{uu_x}^{\boldsymbol{\theta}}(\mathbf a)=J_x\,\nabla_a H^{a,\boldsymbol{\theta}_x}_{uu_x}(\mathbf a),
\qquad
F_{uu_y}^{\boldsymbol{\theta}}(\mathbf a)=J_y\,\nabla_a H^{a,\boldsymbol{\theta}_y}_{uu_y}(\mathbf a).
\end{equation}
Training uses AdamW (learning rate $5\times 10^{-4}$, weight decay $10^{-6}$) for 500 epochs (batch size 16),
with StepLR (step size 150, decay factor 0.5).

\subsubsection*{Poisson inversion block ($\Delta\psi=\omega$).}
Given $\omega$, we predict $\psi$ on the same retained Fourier modes.
We parameterize a diagonal quadratic generator with softplus-constrained weights and fit it using exact Fourier inversion targets.
Training uses AdamW for 200 epochs (batch size 128) with StepLR (step size 80, decay factor 0.3).

\subsection{2D Neumann baseplate (cosine/DCT)}
\label{sec:app_train_2d_neumann}

We consider $\Omega=[0,1]^2$ with homogeneous Neumann boundary conditions and represent fields in a tensor-product cosine basis
\[
u(x,y)=\sum_{j=0}^{K_{\rm cut}}\sum_{\ell=0}^{K_{\rm cut}} a_{j,\ell}\,
\cos(\pi j x)\cos(\pi \ell y).
\]
Grid evaluations use an endpoint $N\times N$ grid with IDCT-I/DCT-I transforms for $\Phi_b$ and $\mathcal{P}_b$.
Unless stated otherwise, $N=65$ and $K_{\rm cut}=\min\!\bigl(K_{\max},\lfloor (N-1)/3\rfloor\bigr)=21$, so $K=(K_{\rm cut}+1)^2$.
Let $\mathbf a\in\mathbb{R}^K$ be the vector
obtained by stacking the coefficients $\{a_{k,\ell}\}$ in a fixed order.

%\paragraph{Neumann Laplacian diffusion block ($u\mapsto \Delta u$).}
%We learn a Neumann Laplacian diffusion block on retained cosine modes by parameterizing
%$E_{\Delta}^{\boldsymbol{\theta}}:\mathbb{R}^{K}\to\mathbb{R}$ with an MLP (4 hidden layers, width 256, GELU) and setting
%\begin{equation}
%F_{\Delta}^{\boldsymbol{\theta}}(a)\;=\;-\nabla_a E_{\Delta}^{\boldsymbol{\theta}}(a).
%\end{equation}
%Training uses AdamW (learning rate $10^{-3}$) with StepLR (step size 40, decay factor 0.3) for 80 epochs (batch size 128).

\subsubsection*{Neumann Laplacian diffusion block ($u\mapsto \Delta u$).}
This block follows the same Laplacian design and training recipe as the Fourier-baseplate diffusion block \eqref{eq:fourier_delta}, with the only change being the baseplate.

\subsection{3D periodic baseplate (Fourier/FFT)}
\label{sec:app_train_3d_fourier}

We consider $\Omega=[0,2\pi)^3$ and represent real fields by a band-limited 3D Fourier expansion
\[
u(x,y,z)=\sum_{|j|\le K_{\rm cut}}\sum_{|\ell|\le K_{\rm cut}}\sum_{|m|\le K_{\rm cut}}
a_{j,\ell,m}\,e^{i(jx+\ell y+m z)},\qquad
a_{-\!j,-\!\ell,-\!m}=\overline{a_{j,\ell,m}} .
\]
Coefficients are stored in a real Hermitian-packed vector $\mathbf a\in\mathbb{R}^K$.
Grid evaluations use an $N\times N\times N$ uniform grid with 3D FFT/iFFT transforms for $\Phi_b$ and $\mathcal{P}_b$.
We retain modes up to $K_{\rm cut}$, following the same truncation convention as in the 2D periodic baseplate.

\subsubsection*{3D Laplacian diffusion block ($u\mapsto \Delta u$).}
The 3D Laplacian block is trained in exactly the same way as the 2D periodic Fourier Laplacian block \eqref{eq:fourier_delta}, except that the baseplate is the 3D Fourier expansion and the retained index set is three-dimensional.

%\paragraph{3D Laplacian diffusion block ($u\mapsto \Delta u$).}
%We pretrain a 3D Laplacian diffusion block by learning an energy generator $E_{\Delta}^{\boldsymbol{\theta}}:\mathbb{R}^K\to\mathbb{R}$
%and defining
%\begin{equation}
%F_{\Delta}^{\boldsymbol{\theta}}(a)\;=\;-\nabla_a E_{\Delta}^{\boldsymbol{\theta}}(a).
%\end{equation}
%Targets are generated by the exact 3D spectral Laplacian restricted to the retained band-limited modes.
%The architecture, sampling protocol, and optimizer settings follow the 2D periodic configuration in
%Section~\ref{sec:app_train_2d_fourier}.

\section{Additional numerical experiments}
\label{sec:app_additional_experiments}

This section reports supporting experiments that follow the same baseplate interface, pretrained-block reuse, and Strang block-composition protocol as in the main text, and collectively reinforce the advantages of LegONet in accuracy, stability, and plug-and-play reuse across PDE settings.
Each experiment is evaluated by the weighted relative error $\mathrm{rel}$ and the normalized pointwise profile $e(x)$ defined in \eqref{eq:err_metrics}.
Unless otherwise stated, linear or stiff coefficient-space substeps use exact or Crank--Nicolson-type updates, whereas nonlinear substeps, including reaction, forcing, and transport terms, use second-order explicit updates such as Heun schemes, or exact pointwise maps when available. Thus the within-block updates are consistent with the second-order assumption used in the analysis.

% -------------------------
\subsection{1D Dirichlet domains}
\label{sec:app_1d_dirichlet_exps}

\subsubsection*{Experiment A1: 1D Ginzburg--Landau.}
We consider
\begin{equation}
  u_t = u_{xx} - (u+u^3),
  \qquad x\in(-1,1),\qquad u(\pm1,t)=0,
\end{equation}
advanced by a diffusion--reaction Strang composition
$S^{\boldsymbol{\theta}}_{u_{xx},\Delta t/2}\circ S_{-(u+u^3),\Delta t}\circ S^{\boldsymbol{\theta}}_{u_{xx},\Delta t/2}$.
The diffusion substep $u_{xx}$ uses a Crank--Nicolson update in the Shen--Legendre coefficient space, while the reaction substep applies a second-order Heun update to the projected local term $-(u+u^3)$ on quadrature nodes.
We take $\Delta t=10^{-4}$ and $N_{\mathrm{steps}}=10^3$ ($T=0.1$).
Fig. \ref{fig:bench_gl_1d} shows snapshots and the corresponding $e(x)$.
Extended Data Fig.~\ref{fig:1d_gl_energy} reports the energy diagnostics.

\subsubsection*{Experiment A2: 1D heat equation with time-dependent Dirichlet data.}
We consider
\begin{equation}
  u_t=\nu u_{xx},\qquad x\in(-1,1),\qquad u(-1,t)=A(t),\ \ u(1,t)=B(t),
\end{equation}
with $\nu=2\times 10^{-2}$ and impose the non-homogeneous boundary data by lifting,
$u=u_0+u_{\mathrm{lift}}$, where
$u_{\mathrm{lift}}(x,t)=\frac{1-x}{2}A(t)+\frac{1+x}{2}B(t)$ so that $u_0(\pm1,t)=0$.
The interior component satisfies
\[
(u_0)_t=\nu (u_0)_{xx}+g(x,t),
\qquad
g(x,t)=-(u_{\mathrm{lift}})_t,
\]
and is advanced by a diffusion--forcing Strang composition
$S_{g(x,t),\Delta t/2}\circ S^{\boldsymbol{\theta}}_{\nu (u_0)_{xx},\Delta t}\circ S_{g(x,t),\Delta t/2}$,
with Crank--Nicolson diffusion and Heun forcing.
We take $A(t)=A_0+\alpha_A\sin(2\pi\omega t)$, $B(t)=B_0+\alpha_B\cos(2\pi\omega t)$ with $(A_0,B_0)=(0.2,-0.2)$, $\alpha_A=\alpha_B=5.6$, and $\omega=1$.
We use $\Delta t=10^{-3}$ and $N_{\mathrm{steps}}=2\times10^4$ ($T=20$).
Fig.~\ref{fig:bench_heat_1d} reports snapshots and $e(x)$.
Energy diagnostics are collected in Extended Data Fig.~\ref{fig:1d_heat_u0_energy}.

% -------------------------
\subsection{2D periodic domains}
\label{sec:app_2d_periodic_exps}

\subsubsection*{Experiment A3: 2D Allen--Cahn.}
We consider
\begin{equation}
  u_t=\varepsilon\,\Delta u + u-u^3,\qquad (x,y)\in\Omega=[0,2\pi)^2,
\end{equation}
with $\varepsilon=10^{-2}$, advanced by a diffusion--reaction Strang composition.
The diffusion substep applies a coefficient-space update for $\varepsilon\Delta$ on the retained Fourier modes, and the reaction substep applies a Heun update to $u-u^3$ pointwise on the grid followed by projection and de-aliasing.
We take $\Delta t=10^{-3}$ and $N_{\mathrm{steps}}=10^5$ ($T=100$).
Fig.~\ref{fig:bench_ac2d} and Extended Data Figs. \ref{fig:ac2d_relerr}, \ref{fig:ac2d_energy} report snapshots, $\mathrm{rel}$, and energy decay.

\subsubsection*{Experiment A4: 2D vector Burgers.}
We consider
\begin{equation}
  \begin{aligned}
    u_t + u\,u_x + v\,u_y &= \nu\,\Delta u,\\
    v_t + u\,v_x + v\,v_y &= \nu\,\Delta v,
  \end{aligned}
  \qquad (x,y)\in\Omega=[0,2\pi)^2,
\end{equation}
with $\nu=10^{-3}$, advanced by a symmetric Strang composition on the retained Fourier modes as shown in Fig. \ref{fig:strang}.
Diffusion uses an implicit coefficient-space update, while transport is evaluated pseudo-spectrally: the self-advection terms are provided by the pretrained density blocks for $uu_x$ and $uu_y$, and the cross terms are computed on the grid with Fourier differentiation before projection.
We take $\Delta t=2\times10^{-3}$ and $N_{\mathrm{steps}}=1.5\times10^4$ ($T=30$).
Figs.~\ref{fig:bench_burgers2d_u}, \ref{fig:bench_burgers2d_v} and Extended Data Figs. \ref{fig:burgers2d_u_relerr}, \ref{fig:burgers2d_v_relerr} report snapshots and $\mathrm{rel}$ for both components.

% -------------------------
\subsection{2D Neumann domains}
\label{sec:app_2d_neumann_exps}

\subsubsection*{Experiment A5: 2D Swift--Hohenberg.}
We consider
\begin{equation}
\begin{cases}
  u_t = -(\Delta+k_0^2)^2 u + \mu u-u^3, & (x,y)\in\Omega=[0,1]^2,\\
  \partial_n u=0,\quad \partial_n(\Delta u)=0, & (x,y)\in\partial\Omega,
\end{cases}
\end{equation}
with $\mu=0.5$ and $k_0=6$, advanced by a linear--reaction Strang composition.
The stiff linear substep advances $u_t=-(\Delta+k_0^2)^2u$ in cosine coefficient space by mode-wise diagonal updates, and the reaction substep applies a pointwise update to $\mu u-u^3$ on the physical grid followed by projection to the retained cosine modes.
We take $\Delta t=10^{-2}$ and $N_{\mathrm{steps}}=2\times10^3$ ($T=20$).
Fig.~\ref{fig:bench_sh2d} and Extended Data Fig. \ref{fig:sh2d_relerr} report snapshots and $\mathrm{rel}$.

\subsubsection*{Experiment A6: 2D Cahn--Hilliard with non-homogeneous Neumann flux.}
We consider
\begin{equation}
  u_t=\Delta\mu,\qquad \mu=-\varepsilon^2\Delta u+(u^3-u),\qquad (x,y)\in\Omega=[0,1]^2,
\end{equation}
with boundary flux
$\partial_y u(x,0,t)=g(x)$ and $\partial_y u(x,1,t)=-g(x)$, where $g(x)=g_{\mathrm{amp}}\cos(\pi x)$, and homogeneous Neumann conditions in $x$.
We set $\varepsilon=5\times10^{-2}$ and $g_{\mathrm{amp}}=5\times10^{-2}$.
We impose the flux by a harmonic lifting $u=u_0+u_{\mathrm{lift}}$ with $\Delta u_{\mathrm{lift}}=0$, so that $u_0$ satisfies homogeneous Neumann conditions and is represented in the cosine basis.
Time stepping advances $u_0$ by a Strang splitting between the stiff linear operator $-\varepsilon^2\Delta^2 u_0$ and the remaining nonlinear term evaluated on the physical grid.
We take $\Delta t=5\times10^{-4}$ and $N_{\mathrm{steps}}=2\times10^4$ ($T=10$).
Fig.~\ref{fig:bench_ch2d} and Extended Data Fig. \ref{fig:ch2d_relerr} report snapshots and $\mathrm{rel}$.

% -------------------------

\subsection{3D periodic domains}
\label{sec:app_3d_periodic_exps}

\subsubsection*{Experiment A7: 3D Allen--Cahn with volume constraint.}

We consider
\begin{equation}
  u_t=\varepsilon\Delta u + u-u^3-\lambda(t),
  \qquad \lambda(t)=\langle u-u^3\rangle,
  \qquad (x,y,z)\in\Omega=[0,2\pi)^3,
\end{equation}
with $\varepsilon=10^{-2}$, advanced by the same diffusion--reaction Strang template as in 2D, with the spatial-average correction applied in the reaction step.
Here $\langle f\rangle:=|\Omega|^{-1}\int_\Omega f(x)\,dx$ denotes the spatial average.
We take $\Delta t=5\times10^{-3}$ and $N_{\mathrm{steps}}=800$ ($T=4$).
This experiment deliberately departs from the spectral-decay Gaussian prior \eqref{eq:app_gaussian_prior_k} used in block pretraining and instead employs a two-phase voxel initialization with sharp interfaces as shown in Fig.~\ref{fig:bench_ac3d}.
Such non-smooth morphology induces a substantially different coefficient distribution, activating higher-frequency modes absent from the training prior.
The test therefore probes robustness under coefficient-distribution shift.
Fig.~\ref{fig:bench_ac3d} reports phase renderings and $\mathrm{rel}$.

\subsubsection*{OOD initial-condition priors for 3D Swift--Hohenberg}\label{sec:app_sh3d_ood}

We evaluate robustness under two out-of-distribution initial-condition priors while keeping the same PDE parameters, baseplate, and rollout solver.
The in-distribution (ID) prior follows the spectral Gaussian construction used throughout the paper, with modewise standard deviation
$\sigma_k=\mathrm{amp}/(1+\|k\|_2)^{\alpha}$.
OOD1 uses the same Gaussian family but removes spectral decay by setting $\alpha=0$, so that all retained Fourier modes have identical variance and the resulting fields carry increased high-frequency energy.
OOD2 uses a blocky two-phase prior that introduces sharp interfaces in physical space:
we first sample a piecewise-constant random field on a coarse grid of size $N_c^3$ with $N_c=8$,
\[
u_0^{(c)}(\xi)=-1+2B(\xi),
\qquad
B(\xi)\sim\mathrm{Bernoulli}(p),
\qquad
\xi\in\{1,\ldots,N_c\}^3,
\]
so that $\mathbb{P}[u_0^{(c)}(\xi)=+1]=p$ and $\mathbb{P}[u_0^{(c)}(\xi)=-1]=1-p$ independently over $\xi$.
We set $p=0.35$, upsample $u_0^{(c)}$ to the $N^3$ simulation grid, and apply the same amplitude normalization as in the ID setting.
Together, OOD1 and OOD2 probe robustness to increased high-frequency content in coefficient space and to non-smooth initial interfaces in physical space, respectively.

\section{Baseline configuration}\label{sec:app_baselines_1to3}

This section summarizes the baseline setups for the main-text solver-level comparisons.
We compare LegONet with two widely used supervised neural-operator baselines:
Fourier Neural Operator (FNO)~\cite{li2020fourier} and DeepONet~\cite{lu2019deeponet},
representing canonical spectral-convolution and branch--trunk architectures.
All methods are evaluated in closed loop under an identical rollout protocol, starting from the same initial condition.

For the 1D Burgers experiment, we additionally include a standard physics-informed neural network (PINN) baseline, which learns a continuous surrogate $u^{\boldsymbol{\theta}}(x,t)$ from PDE residual and boundary/initial constraints without trajectory supervision.
We do not include PINNs in 2D/3D because a like-for-like solver-level comparison would require long-horizon optimization over high-dimensional space–time fields under stiff/higher-order operators and coupled constraints.
This makes training highly sensitive to collocation design and loss balancing, preventing a controlled comparison in our setting.

\subsection{Baseline models}\label{sec:app_baseline_models}

\subsubsection*{FNO.}
We use an FNO time-stepper in residual-update form,
\begin{equation}\label{eq:fno_stepper}
\mathbf u_{n+1} \;=\; \mathbf u_n \;+\; \alpha\,\delta^{\boldsymbol\theta}(\mathbf u_n;\Delta t),
\end{equation}
where $\delta^{\boldsymbol\theta}$ is an FNO backbone that maps the current discrete field to an increment on the same set of nodes, and $\alpha>0$ is a residual scale.
Residual parameterizations are commonly used to stabilize long-horizon rollouts of neural time-steppers.

\subsubsection*{DeepONet.}
We use the standard DeepONet branch--trunk factorization to represent the one-step operator as a low-rank bilinear form with a residual update,
\[
\mathbf u_{n+1}(\mathbf{x}) \;=\; \mathbf u_n(\mathbf{x}) \;+\; \alpha \,\big\langle b^{\boldsymbol\theta}(\mathbf{s}_n),\, t^{\boldsymbol\theta}(\mathbf{x})\big\rangle,
\]
where $\alpha>0$ is a residual scale and $\langle\cdot,\cdot\rangle$ denotes the Euclidean inner product in $\mathbb{R}^r$.
The branch network $b^{\boldsymbol\theta}:\mathbb{R}^{N_s}\!\to\mathbb{R}^r$ takes as input sensor measurements
$\mathbf{s}_n=\big(u_n(\mathbf{x}_1),\ldots,u_n(\mathbf{x}_{N_s})\big)$ and outputs coefficients in a rank-$r$ latent space.
The trunk network $t^{\boldsymbol\theta}:\Omega\!\to\mathbb{R}^r$ maps a query location $\mathbf{x}\in\Omega$ to a location-dependent basis vector.
The increment is evaluated as
$\langle b^{\boldsymbol\theta}(\mathbf{s}_n), t^{\boldsymbol\theta}(\mathbf{x})\rangle
=\sum_{j=1}^{r} b^{\boldsymbol\theta}_j(\mathbf{s}_n)\, t^{\boldsymbol\theta}_j(\mathbf{x})$,
and is computed at the experiment-specific evaluation nodes.

For each experiment, baselines are trained to advance the experiment-specific discretization and are evaluated on the same macro-time snapshots and evaluation nodes.
We use standard instantiations of FNO and DeepONet and tune widths (and, for DeepONet, the rank) so that parameter counts are comparable to those of the corresponding LegONet blocks. Exact configurations are reported below.
A key failure mode of teacher-forced one-step training is train--test mismatch in closed-loop rollouts.
To mitigate this effect, we train FNO and DeepONet with a rollout-aware objective over a short unrolled window.
In contrast to LegONet block pretraining, which matches instantaneous operator labels via \eqref{eq:method_loss},
the supervised baselines minimize a rollout-aware $K_{\mathrm{roll}}$-step loss over short trajectory windows:
\begin{equation}\label{eq:baseline_rollout_loss}
\min_{\boldsymbol{\theta}}\;
\mathbb{E}_{(\mathbf u_n,\mathbf u_{n+1},\ldots,\mathbf u_{n+K_{\mathrm{roll}}})}\!\left[
\frac{1}{K_{\mathrm{roll}}}\sum_{k=1}^{K_{\mathrm{roll}}}
\big\|\mathbf u_{n+k} - \widehat{\mathbf u}_{n+k}^{\boldsymbol{\theta}}\big\|_{w,2}^2
\right],
\end{equation}
where $\widehat{\mathbf u}_{n}^{\boldsymbol{\theta}}=\mathbf u_n$ and
$\widehat{\mathbf u}_{n+k}^{\boldsymbol{\theta}}=\mathcal{S}_{\boldsymbol{\theta}}\!\left(\widehat{\mathbf u}_{n+k-1}^{\boldsymbol{\theta}}\right)$
for $k\ge 1$.
Here $\mathcal{S}_{\boldsymbol{\theta}}$ denotes the learned one-step snapshot map (FNO or DeepONet).
For each experiment, we choose $\Delta t_{\mathrm{eff}}$, $K_{\mathrm{roll}}$, and $N_{\mathrm{traj}}$ defined below so that the resulting number of rollout windows is matched, up to a small tolerance, to the total block-training samples used by the corresponding LegONet assembly.

\subsection*{Experiment 1 (1D Burgers)}

For baseline feasibility, we train supervised operator learners on a coarser effective time step rather than the fine micro step used by the reference solver.
When the snapshot spacing is too small, the one-step map becomes near-identity and teacher-forced training is dominated by trivial identity fitting, which exacerbates train--test mismatch in closed-loop rollouts.
We therefore use a shortened horizon $T=0.2$ and an effective step $\Delta t_{\mathrm{eff}}=10^{-3}$, generate $N_{\mathrm{traj}}=220$ reference rollouts, and store trajectories on Gauss--Legendre quadrature nodes with $Q=256$ points as a tensor of shape $(N_{\mathrm{traj}},T_m,Q)$, where $T_m=T/\Delta t_{\mathrm{eff}}+1=201$.
All reported errors are computed on these Gauss--Legendre nodes using the weighted norm $\|\cdot\|_{w,2}$ in \eqref{eq:err_metrics}.
Note that each LegONet block is pretrained by operator matching using $20{,}000$ independent coefficient samples drawn from the Gaussian prior, yielding $40{,}000$ samples total for the two blocks.
The resulting rollout-window count matches the LegONet training-sample budget for this experiment.
%For supervised baselines, we train on rollout windows extracted from the stored trajectories, which yields a number of training windows on the same numerical scale as the total block-training samples used by the corresponding LegONet assembly.
%For supervised baselines, we use $N_{\mathrm{traj}}(T_m-1)=200\times 200=40{,}000$ one-step transitions, matching the number of training pairs to the LegONet coefficient-sample budget.

While evaluation is always performed on the Gauss--Legendre nodes, the training interface depends on the baseline architecture.
DeepONet is trained directly on the quadrature representation.
FNO requires an equispaced grid to enable FFT-based spectral convolutions, so we resample each snapshot from the Gauss--Legendre nodes to an equispaced grid of the same resolution and use a fixed odd extension across the boundaries before the FFT layers, which is more consistent with the homogeneous Dirichlet boundary values.
%FNO predictions are then interpolated back to the Gauss--Legendre nodes for loss evaluation and reporting.
PINN is trained in continuous space--time using interior and boundary collocation points, and is evaluated by querying $u^\theta(x,t_n)$ at the Gauss--Legendre nodes.

To reduce capacity confounders, we tune the baseline widths and ranks so that their trainable parameter counts are comparable to those of the full LegONet assembly used in this experiment.
%To reduce capacity confounders, we match parameter counts between the baselines and the two LegONet blocks used in this experiment.
Our FNO uses retained Fourier modes $m=16$, channel width $40$, depth $2$, and residual scaling $\alpha=0.1$.
DeepONet follows the branch--trunk factorization: the branch ingests $N_s=64$ sensor values of the current snapshot (selected as a fixed subset of the $Q=256$ Gauss--Legendre nodes) and outputs a latent vector in $\mathbb{R}^{r}$ with $r=145$, while the trunk maps query locations $x$ to $\mathbb{R}^{r}$. We use $3$-layer MLPs (width $128$, GELU) for both branch and trunk, and apply the same residual scaling $\alpha=0.1$ in the one-step map.
We train FNO and DeepONet with the rollout-aware objective \eqref{eq:baseline_rollout_loss} using $K_{\mathrm{roll}}=25$ and AdamW (learning rate $5\times 10^{-4}$, weight decay $10^{-6}$), batch size $16$, for $2000$ epochs.

The PINN represents the solution as a continuous function on $x\in[-1,1]$ and $t\in[0,T]$ and is trained without trajectory supervision by minimizing a weighted sum of a PDE-residual loss and initial/boundary penalties,
\begin{equation}\label{eq:pinn_loss}
\min_{\boldsymbol{\theta}}\;
w_{\mathrm{phys}}\,\mathcal{L}_{\mathrm{phys}}
\;+\;
w_{\mathrm{ic}}\,\mathcal{L}_{\mathrm{ic}}
\;+\;
w_{\mathrm{bc}}\,\mathcal{L}_{\mathrm{bc}},
\qquad
(w_{\mathrm{phys}},w_{\mathrm{ic}},w_{\mathrm{bc}})=(1,10,10),
\end{equation}
where $\mathcal{L}_{\mathrm{phys}}$ is the mean-squared PDE residual evaluated at interior collocation points and
$\mathcal{L}_{\mathrm{ic}},\mathcal{L}_{\mathrm{bc}}$ are mean-squared penalties enforcing the initial and boundary conditions.
We use an MLP with width $149$, depth $6$, and $\tanh$ activations, optimized with Adam (learning rate $2\times10^{-4}$) for $8000$ epochs.
At each epoch, we resample collocation points uniformly in space--time and use $N_{\mathrm{int}}=4096$ interior points and $N_{\mathrm{bd}}=1024$ initial/boundary points.

\subsection*{Experiment 2 (2D Navier--Stokes)}

We generate $N_{\mathrm{traj}}=800$ reference rollouts on a $64\times 64$ grid up to $T=50$ and store effective snapshots with step size $\Delta t_{\mathrm{eff}}=1.0$.
This choice matches the LegONet sample budget up to a small tolerance.
We train an FNO time-stepper in residual-update form with scaling $\alpha=1$,
using a 2D spectral-convolution backbone with retained modes $m=12$, channel width $64$, and depth $4$.
Training uses the rollout-aware objective \eqref{eq:baseline_rollout_loss} with unroll length $K_{\mathrm{roll}}=10$ and AdamW
(learning rate $10^{-3}$, weight decay $10^{-4}$), batch size $4$, for $2000$ epochs.
DeepONet follows the standard branch--trunk construction on the same effective-time data, also in residual-update form with $\alpha=1$.
The branch network ingests $N_s=1024$ fixed spatial sensors (a uniform subsampling of the $64\times 64$ grid) and outputs a rank-$r$ latent vector with $r=128$,
while the trunk maps 2D query locations $\mathbf{x}=(x,y)$ to $\mathbb{R}^{r}$.
Both branch and trunk are implemented as MLPs (width $256$, depth $3$, GELU), and training follows the same rollout-aware protocol and model selection criteria as for FNO.
In this experiment, LegONet reuses a pretrained Laplacian diffusion block and a Poisson inversion block on the periodic Fourier baseplate with structured diagonal parametrization.
Supervised baselines, by contrast, must approximate the full nonlinear time-advance operator at step size $\Delta t_{\mathrm{eff}}$ and therefore require substantially higher-capacity networks for stable long-horizon rollouts.

%For data budgets, LegONet pretrains each block by operator matching using $20{,}000$ independent coefficient samples.
%The supervised baselines are trained from strided rollout transitions, yielding $N_{\mathrm{traj}}(T_{m}-1)=400\times 50=20{,}000$ effective one-step pairs.
%We thus match the number of training samples across paradigms, while noting that the supervision differs: baselines receive fieldwise one-step labels, whereas LegONet matches instantaneous mechanism operators on the retained coefficient interface.

\subsection*{Experiment 3 (3D Swift--Hohenberg)}

We construct a supervised rollout dataset using the Strang-splitting reference solver on a $64^3$ periodic grid up to $T=30$.
We record effective snapshots every $\Delta t_{\mathrm{eff}}=0.15$.
The dataset contains $N_{\mathrm{traj}}=100$ trajectories, yielding $T_m=T/\Delta t_{\mathrm{eff}}+1=201$ stored frames per trajectory.

We train a 3D FNO time-stepper using the residual update with fixed scale $\alpha=0.6$.
The model uses a spectral-convolution backbone with retained modes $m=16$, channel width $64$, depth $6$, and LayerNorm.
Training uses the rollout-aware objective \eqref{eq:baseline_rollout_loss} with unroll length $K_{\mathrm{roll}}=8$ and AdamW
(learning rate $6\times 10^{-4}$, weight decay $10^{-6}$) for $2000$ epochs.
We do not include a DeepONet baseline in 3D because the standard branch--trunk evaluation requires producing an $r$-dimensional trunk feature at every spatial query location to form a full-field update.
On a $64^3$ grid, this introduces a prohibitive compute and memory footprint for rollout-aware training under the same closed-loop field supervision used for FNO and LegONet.
Subsampling query points would change the evaluation interface and would no longer constitute a like-for-like solver-level comparison.

%LegONet reuses a pretrained Laplacian block learned by operator matching on $20{,}000$ independent coefficient samples.
%The supervised FNO baseline is trained from macro-time rollout transitions, yielding $N_{\mathrm{traj}}(T_m-1)=100\times 200=20{,}000$ one-step training pairs.
%We thus match the number of training samples across paradigms.
     % SI

\end{document}